\documentclass[dvips,ejs,preprint]{imsart}
\usepackage{times}
\usepackage{graphicx,epsfig,ifthen}
\usepackage{latexsym}
\usepackage{eepic}
\usepackage{color}
\usepackage[sort&compress,numbers]{natbib}
\usepackage{amsmath,amssymb}
\usepackage{amsthm}
\usepackage[mathscr]{eucal}
\usepackage[backref,pageanchor=true,plainpages=false,
colorlinks,citecolor=blue,urlcolor=blue,
pdfpagelabels, bookmarks,bookmarksnumbered,
]{hyperref}

\usepackage{hypernat}

\startlocaldefs

\newcommand{\region}{\mathcal{U}}
\newcommand{\Px}{\mathcal{P}}
\newcommand{\PXY}{\mathcal{P}_{XY}}
\newcommand{\PXYR}[1]{\mathcal{P}_{#1}}
\newcommand{\Risk}{{\rm R}_{\loss}}
\newcommand{\Rademacher}{\hat{\phi}}
\newcommand{\loss}{\ell} 
\newcommand{\maxloss}{\bar{\loss}} 
\newcommand{\transform}{\psi_{\loss}} 
\newcommand{\tsyba}{\alpha} 
\newcommand{\tsybca}{a} 
\newcommand{\tsybb}{\beta} 
\newcommand{\tsybcb}{b} 
\renewcommand{\dim}{d} 
\newcommand{\entc}{q} 
\newcommand{\entrho}{\rho} 
\newcommand{\g}{g}

\newcommand{\dc}{\theta} 
\newcommand{\capacity}{\tau} 
\newcommand{\mixcap}{\chi_{\loss}} 
\newcommand{\exreals}{\bar{\mathbb{R}}}
 
\newcommand{\D}{{\rm D}} 
\newcommand{\dist}{\Delta} 
\newcommand{\radius}{{\rm radius}} 

\newcommand{\Transform}{\Gamma_{\loss}} 
\newcommand{\InvTransform}{\Epsilon_{\loss}}
\newcommand{\Functions}{\F^{*}} 
\newcommand{\FunctionsXY}{\G^{*}} 
\newcommand{\Env}{{\rm F}} 
\newcommand{\bound}{\nu_{0}} 
\newcommand{\BJM}{\Psi_{\loss}}
\newcommand{\zo}{{\scriptscriptstyle{01}}}
\newcommand{\spec}[1]{\mathring{#1}} 
 
\newcommand{\phiinv}[1]{\ddot{#1}} 
\newcommand{\restinv}[1]{\dot{#1}}

\newcommand{\ERM}{{\rm ERM}}

\newcommand{\SC}{{\rm M}}
\newcommand{\vc}{{\rm vc}}
\newcommand{\covering}{{\cal{N}}}
\newcommand{\G}{{\cal{G}}}
\newcommand{\conv}{{\rm conv}}
\newcommand{\Fbound}{\bar{B}} 
\newcommand{\metricbound}{\bar{d}_{\loss}} 
\newcommand{\metric}{d_{\loss}} 
\newcommand{\Lip}{L} 
\newcommand{\convc}{C_{\loss}} 
\newcommand{\convr}{r_{\loss}} 
\newcommand{\modconv}{\bar{\delta}_{\loss}} 
\newcommand{\Bracket}[1]{[#1]} 
\newcommand{\Log}{{\rm Log}}
\newcommand{\Borel}{{\cal B}}
\newcommand{\BorelX}{{\cal B}_{\X}}

\newcommand{\ACAL}{Algorithm 1}
\newcommand{\conf}{\delta}
\newcommand{\eps}{\varepsilon}
\newcommand{\sfun}{\mathfrak{s}}
\newcommand{\X}{\mathcal X} 
 
\newcommand{\Y}{\mathcal Y} 
\newcommand{\alg}{\mathcal A} 
\renewcommand{\H}{\mathcal H} 
\renewcommand{\L}{\mathcal L} 
\newcommand{\U}{\mathcal U} 
\newcommand{\target}{f^{\star}} 
 
\newcommand{\er}{{\rm er}} 
 
\newcommand{\Ball}{{\rm B}}
\DeclareSymbolFont{bbold}{U}{bbold}{m}{n}
\DeclareSymbolFontAlphabet{\mathbbold}{bbold}
\newcommand{\ind}{\mathbbold{1}}

\newcommand{\C}{\mathcal{C}} 
\newcommand{\F}{\mathcal{F}}
\newcommand{\sign}{{\rm sign}} 
 
\renewcommand{\P}{\mathbb P}

\newcommand{\nats}{\mathbb{N}} 
\newcommand{\reals}{\mathbb{R}} 
\newcommand{\ints}{\mathbb{Z}}
\newcommand{\Epsilon}{\mathscr{E}}
\newcommand{\E}{\mathbb E}
\newcommand{\Data}{\mathbf{\mathcal Z}}

\newcommand{\diam}{{\rm diam}}
\newcommand{\DIS}{{\rm DIS}}
\newcommand{\DISF}{{\rm DISF}}
\newcommand{\DISFXY}{\overline{\DISF}}
\newcommand{\DISV}{D}

\newcommand{\argmin}{\mathop{\rm argmin}}

\makeatletter
\newcommand{\vast}{\bBigg@{3}}
\newcommand{\Vast}{\bBigg@{4}}
\makeatother

\newcommand{\longversion}[1]{}
\newcommand{\ignore}[1]{}
\newcommand{\todo}[1]{}
\newcommand{\oldstuff}[1]{}

\newtheorem{theorem}{Theorem}
\newtheorem{corollary}[theorem]{Corollary}
\newtheorem{lemma}[theorem]{Lemma}

\newtheorem{definition}[theorem]{Definition}

\newtheorem{condition}[theorem]{Condition}

\newsavebox{\savepar}

\newenvironment{bigboxit}{\begin{center}\begin{lrbox}{\savepar}
\begin{minipage}[h]{4.6in}
\normalfont
\begin{flushleft}}
{\end{flushleft}\end{minipage}\end{lrbox}\fbox{\usebox{\savepar}}
\end{center}}

\providecommand{\upthmend}[1]{}
  \renewcommand{\upthmend}[1]{%
    {\unskip\nobreak\hfil\penalty 50%
      \vskip #1\hskip 2em\hbox{}\nobreak\hfil $\diamond$%
      \parfillskip=0pt \finalhyphendemerits=0 \par}}

\endlocaldefs

\begin{document}

\begin{frontmatter}

\title{Surrogate Losses in Passive and Active Learning}
\runtitle{Surrogate Losses in Passive and Active Learning}

\author{
\fnms{Steve} \snm{Hanneke}\thanksref{t1}\corref{}\ead[label=e1]{steve.hanneke@gmail.com}}
\and
\author{\fnms{Liu} \snm{Yang}\thanksref{m1}\ead[label=e2]{liu.yang0900@outlook.com}}
\address{\thanksmark{t1}Toyota Technological Institute at Chicago\\\printead{e1}}
\address{\thanksmark{m1}\printead{e2}}

\runauthor{Hanneke and Yang}

\begin{abstract}
Active learning is a type of sequential design for supervised machine learning, in which the learning
algorithm sequentially requests the labels of selected
instances from a large pool of unlabeled data points.
The objective is to produce a classifier of relatively low risk, as measured under the $0$-$1$ loss, ideally using fewer label requests
than the number of random labeled data points sufficient to achieve the same.
This work investigates the potential
uses of surrogate loss functions in the context of active learning.
Specifically, it presents an
active learning algorithm based on an arbitrary classification-calibrated surrogate loss function,
along with an analysis of the number of label requests
sufficient
for the classifier returned by the algorithm to achieve a given risk
under the $0$-$1$ loss.
Interestingly, these results cannot be obtained by simply
optimizing the surrogate risk via active learning to an
extent sufficient to provide a guarantee on the $0$-$1$ loss,
as is common practice in the
analysis of surrogate losses for passive learning.
Some of the results have additional implications for the
use of surrogate losses in passive learning.
\end{abstract}



\end{frontmatter}


\section{Introduction}
\label{sec:intro}

In supervised machine learning, we are tasked with learning a classifier whose probability of making a mistake (i.e., error rate) is small.
The study of when it is possible to learn an accurate classifier
via a computationally efficient algorithm,
and how to go about doing so,
 is a subtle and difficult topic, owing largely to nonconvexity of the loss function: namely, the
$0$-$1$ loss.  While there is certainly an active literature on developing computationally efficient methods
that succeed at this task, even under various noise conditions \citep*[e.g.,][]{angluin:88,kearns:94a,kearns:98,kalai:05},
it seems fair to say that at present, many of these advances have not yet reached
the level of robustness, efficiency, and simplicity required for most applications.
In the mean time, practitioners have turned to various heuristics in the design of
practical learning methods, in attempts to circumvent these tough computational problems.
One of the most common such heuristics is the use of a convex \emph{surrogate} loss function
in place of the $0$-$1$ loss in various optimizations performed by the learning method.
The convexity of the surrogate loss allows these optimizations to be performed efficiently,
so that the methods can be applied within a reasonable execution time,
using
modest computational resources.
Although classifiers arrived at in this way are not always guaranteed to be good classifiers when
performance is measured under the $0$-$1$ loss, in practice this heuristic has often proven quite
effective.  In light of this fact, most modern learning methods either explicitly make use of a
surrogate loss in the formulation of optimization problems (e.g., SVM), or implicitly optimize a
surrogate loss via iterative descent (e.g., AdaBoost).
Indeed, the choice of a surrogate loss is often as fundamental
a part of the process of approaching a learning problem as the
choice of hypothesis class or learning bias.
Thus it seems 
essential that we come to some understanding of how best to make use of
surrogate losses in the design of learning methods, so that in the favorable scenario
that this heuristic actually does work, we have methods taking full advantage of it.

In this work, we are primarily interested in how best to use surrogate losses in the context of \emph{active learning},
which is a type of sequential design in which the learning algorithm is presented with a large pool
of unlabeled data points (i.e., only the covariates are observable), and can sequentially request to observe the labels (response variables) 
of individual instances from the pool.  The objective in active learning is to produce a classifier of low error rate
while accessing a smaller number of labels than would be required for a method based on random labeled data points (i.e., \emph{passive learning}) to achieve the same.
We take as our starting point that we have 
committed to use a given surrogate loss,
and we restrict our attention to just those scenarios in which this heuristic actually \emph{does} work:
specifically, where the minimizer of the surrogate risk also minimizes the error rate, and is contained in our function class.
We are then interested in how best to make use of the surrogate loss toward the goal of producing a
classifier with relatively small error rate.
%
%

In passive learning, the most common approach to using a surrogate loss is
to minimize the empirical surrogate risk on the labeled data.  One can then 
derive guarantees on the error rate of this strategy by bounding the 
surrogate risk via concentration inequalities, and then converting these 
guarantees on the surrogate risk into guarantees on the error rate,
a technique pioneered by \citet*{bartlett:06} and \citet*{zhang:04}.
Interestingly, we find that this direct approach is \emph{not} appropriate 
in the context of active learning: that is, optimizing the surrogate risk
to a sufficient extent to guarantee small error rate generally \emph{cannot} 
yield large improvements over passive learning.
While at first this finding might seem quite negative, it leaves open
the possibility of methods making use of the surrogate loss in alternative
ways, which still guarantee low error rate and computational efficiency, 
but for which these guarantees arise via a less direct route.
Indeed, since we are interested in the surrogate loss only insofar
as it helps us to optimize the error rate with computational efficiency,
we may even consider methods that provide \emph{no} guarantees on
the achieved surrogate risk whatsoever (even in the limit).

In the present work, we propose such an alternative approach to 
the use of surrogate losses in active learning.  
The insight leading to this approach is that,
if we are truly only interested in achieving
low $0$-$1$ loss, then once we have identified the \emph{sign}
of the optimal function at a given point, we need not
optimize the value of the function at that location any further,
and can therefore focus the label requests elsewhere.
Based on this insight, 
we construct an active learning strategy that
optimizes the empirical surrogate risk over increasingly
focused subsets of the instance space, and derive bounds
on the number of label requests the method requires to achieve
a given error rate.  In many cases, these bounds reflect strong
improvements over the analogous results for passive learning 
by minimizing the given surrogate loss.
As a byproduct of this analysis, we find this insight has
implications for the use of certain surrogate losses in passive
learning as well, though to a lesser extent.


Most of the mathematical tools used in this analysis are inspired by techniques
for the study of active learning developed over the past decade \citep*{balcan:09,hanneke:thesis,hanneke:11a,koltchinskii:10},
in conjunction with the results of \citet*{bartlett:06} bounding the excess error rate in terms of
the excess surrogate risk, and the works of \citet*{koltchinskii:06} and \citet*{bartlett:05}
on local Rademacher complexity bounds.

\subsection{Related Work}
\label{subsec:related}

There are many previous works on the topic of surrogate losses
in the context of passive learning.  Perhaps the most relevant
to our results below are the work of \citet*{bartlett:06}
and the related work of \citet*{zhang:04}.  These develop
a general theory for converting results on excess risk under
the surrogate loss into results on excess risk under the
$0$-$1$ loss.  Below, we describe the conclusions of that
work in detail, and we build on many of the basic definitions
and insights pioneered in it. 

Another related line of research, explored by
\citet*{audibert:07}, studies ``plug-in rules,''
which make use of regression estimates obtained by
optimizing a surrogate loss,
and are then rounded to $\{-1,+1\}$ values to obtain classifiers.
They prove minimax optimality results under smoothness assumptions on the actual regression function.
Under similar conditions, \citet*{minsker:12} studies an analogous active learning method,
which again makes use of a surrogate loss, and obtains improvements in label complexity
compared to the passive learning method of \citet*{audibert:07}.  Minsker's active learning 
work has also recently been strengthened and extended in \citep*{hanneke:np1,locatelli:07}.
Remarkably, as discussed by \citet*{audibert:07}, 
the rates of convergence obtained in such works are often better than 
the known results for methods that directly optimize the $0$-$1$ loss,
under analogous complexity assumptions on the Bayes optimal classifier
(rather than the regression function).
As a result, 
these works
raise interesting questions about whether the general analysis of
methods that optimize the $0$-$1$ loss remain tight under complexity
assumptions on the regression function, and potentially also about
the design of optimal methods for classification when assumptions
are phrased in terms of the regression function.  

In the present work, we focus our attention on scenarios where
the main purpose of using the surrogate loss is to ease the computational problems
associated with minimizing an empirical risk, so that
our statistical results might typically be strongest when the surrogate loss is
the $0$-$1$ loss itself, even if in some cases stronger results might in principle be achievable
from assumptions involving the surrogate loss 
\citep*[as in][]{audibert:07,minsker:12}.
As such, in the specific scenarios studied by \citet*{minsker:12},
our results are generally not optimal;
rather, the main strength of our analysis lies in its generality.
In this sense,
our results are more closely related to those of \citet*{bartlett:06} and \citet*{zhang:04}
than to those of \citet*{audibert:07} and \citet*{minsker:12}.
That said, we note that several important elements of the design and analysis of the active learning
method below are already hinted at to some extent in the work of \citet*{minsker:12}, 
albeit in a form that also relies heavily on the assumptions and function class specific to that work; 
the present work takes the general perspective, developing theory and methods applicable to 
any function class and surrogate loss function.

Our approach to the design of active learning methods below 
follows the well-studied strategy of \emph{disagreement-based}
active learning, an approach pioneered by \citet*{balcan:09},
and further developed by several later works \citep*[e.g.,][]{dasgupta:07,hanneke:11a,hanneke:12a,koltchinskii:10}.
The basic strategy maintains a set $V$ of plausible candidates 
for the optimal classifier, and requests the labels of samples
disagreed-upon by classifiers in $V$; it periodically updates 
the set $V$ by eliminating classifiers making an excessive number
of mistakes on the requested labels.  The analysis of the number
of label requests sufficient for this technique to achieve a given
error rate in the general case was explored by
\citet*{hanneke:07b,hanneke:11a}, \citet*{dasgupta:07}, \citet*{koltchinskii:10}, and others,
and the results are typically expressed in terms of a quantity
known as the \emph{disagreement coefficient}.  In the present
work, we modify the disagreement-based active learning strategy
by updating the set $V$, not based on the number of mistakes,
but rather based on the empirical surrogate risk on the queried
samples.  We derive bounds on the number of label requests
this method requires to achieve a given excess error rate, in terms
of properties of the surrogate loss.  In particular, when the surrogate
loss is chosen to be the $0$-$1$ loss itself, this method behaves
nearly-identically to previously-studied methods \citep*{koltchinskii:10,hanneke:12a},
and in this special case, our results match those
established in the literature (with some small refinements in the logarithmic factors).

There are several interesting works on active learning methods that
optimize a general loss function.  \citet*{beygelzimer:09} and \citet*{koltchinskii:10}
have both proposed such methods, and analyzed the number of
label requests the methods make before achieving a given excess risk
for that loss function.  The former method is based on importance weighted
sampling, while the latter makes clear an interesting connection to
local Rademacher complexities.  One natural idea for approaching
the problem of active learning with a surrogate loss is to run one of
these methods with the surrogate loss. The results of
\citet*{bartlett:06} allow us to determine a sufficiently
small value $\gamma$ such that any function with excess surrogate
risk at most $\gamma$ has excess error rate at most $\eps$.
Thus, by evaluating the established bounds on the number
of label requests sufficient for these active learning
methods to achieve excess surrogate risk $\gamma$,
we immediately have a result on the number of label requests
sufficient for them to achieve excess error rate $\eps$.
This is a common strategy for constructing and analyzing
passive learning methods based on a surrogate loss.
However, as we discuss below, this strategy does not
generally lead to the best results for active learning,
and often will not be much better
than results available for related passive learning methods.
Instead, the method we propose
does not aim to optimize the surrogate risk overall,
but rather optimizes it on a sequence of increasingly-focused subregions
of the instance space, and thereby 
provides a smaller bound on the number of label requests sufficient to guarantee excess error rate $\eps$.

\section{Definitions}
\label{sec:definitions}

Let $(\X,\BorelX)$ be a measurable space, where $\X$ is called the \emph{instance space}.
Let $\Y = \{-1,+1\}$, and equip the space $\X \times \Y$ with its
product $\sigma$-algebra: $\Borel = \BorelX \otimes 2^{\Y}$.
Let $\exreals = \reals\cup\{-\infty,\infty\}$, 
let $\Functions$ denote the set of all measurable functions $g : \X \to \exreals$, 
and let $\F \subseteq \Functions$, where $\F$ is called the \emph{function class}.
Throughout, we fix a distribution $\PXY$ over $\X \times \Y$,
and we denote by $\Px$ the marginal distribution of $\PXY$ over $\X$.
In the analysis below, we make the usual simplifying assumption that the events and functions
in the definitions and proofs are indeed measurable.
In most cases, this holds under simple conditions on $\F$ and $\PXY$ \citep*[see e.g.,][]{van-der-Vaart:11};
when this is not the case, one may turn to outer probabilities.  However, we will
not discuss these technical issues further.

For any $h \in \Functions$, and any distribution $P$ over $\X \times \Y$,
denote the \emph{error rate} by $\er(h;P) = P( (x,y) : \sign(h(x)) \neq y)$;
when $P = \PXY$, we abbreviate this as $\er(h) = \er(h;\PXY)$.
Also, let $\eta(X;P)$ be a version of $\P(Y=1|X)$, for $(X,Y) \sim P$;
when $P = \PXY$, abbreviate this as $\eta(X) = \eta(X;\PXY)$.
In particular, 
note that
$\er(h;P)$ is minimized at any $h$ with $\sign(h(\cdot)) = \sign(\eta(\cdot;P)-1/2)$. 
For any $\H \subseteq \Functions$,
define the \emph{region of sign-disagreement}
$\DIS(\H) = \{x \in \X : \exists h,g \in \H \text{ s.t. } \sign(h(x)) \neq \sign(g(x))\}$.
Additionally, 
denote by
$\Bracket{\H} = \{ f \in \Functions : \forall x \in \X, \inf_{h \in \H} h(x) \leq f(x) \leq \sup_{h\in\H} h(x)\}$
the minimal bracket set containing $\H$.

We will use standard big-$O$ notation to express asymptotic dependences.
Specifically, for $f,g : (0,\infty) \to [0,\infty)$, we write $f(\eps) = O(g(\eps))$ or $g(\eps) = \Omega(f(\eps))$
if $\limsup_{\eps \to 0} f(\eps)/g(\eps) < \infty$; we write $f(\eps) = \Theta(g(\eps))$ if both $f(\eps) = O(g(\eps))$ and $f(\eps) = \Omega(g(\eps))$,
and we write $f(\eps) = o(g(\eps))$ if $\limsup_{\eps \to 0} f(\eps)/g(\eps) = 0$.

Our interest here is learning from data, so let
$\Data = \{(X_1,Y_1),(X_2,Y_2),\ldots\}$ denote a sequence of independent $\PXY$-distributed
random variables, referred to as the \emph{labeled data} sequence, while
$\{X_1,X_2,\ldots\}$ is referred to as the \emph{unlabeled data} sequence.
For $m \in \nats$, we also denote $\Data_{m} = \{(X_1,Y_1),\ldots,(X_m,Y_m)\}$.
Throughout, we will let $\conf \in (0,1/4)$ denote an arbitrary confidence
parameter, which will be referenced in the methods and theorem statements.

The \emph{active learning} protocol is defined as follows.
An active learning algorithm is initially permitted access to the sequence
$X_1,X_2,\ldots$ of unlabeled data.  It may then select an index $i_{1} \in \nats$
and \emph{request} to observe $Y_{i_{1}}$; after observing $Y_{i_{1}}$, it may select
another index $i_{2} \in \nats$, request to observe $Y_{i_{2}}$, and so on.  After
a number of such label requests not exceeding a given budget $n$,
the algorithm halts and returns a function $\hat{h} \in \Functions$.
Formally, this protocol specifies a type of decision rule mapping the random sequence
$\Data$ to a function $\hat{h}$, where $\hat{h}$ is 
conditionally independent of $\Data$ given $X_1,X_2,\ldots$ and $(i_1, Y_{i_1}), (i_2, Y_{i_2}),\ldots, (i_n,Y_{i_n})$, 
where each $i_{k}$ is conditionally independent of $\Data$ and $i_{k+1},\ldots,i_{n}$ given $X_1,X_2,\ldots$ and $(i_{1}, Y_{i_1}),\ldots,(i_{k-1},Y_{i_{k-1}})$.

\subsection{Surrogate Loss Functions for Classification}
\label{subsec:surrogate-losses}

Throughout, 
we let 
$\loss : \exreals \to [0,\infty]$ 
denote an arbitrary \emph{surrogate loss function}.
For simplicity, 
suppose $|z| < \infty \Rightarrow \loss(z) < \infty$. 
Define
$\maxloss \!=\! 1 \lor \sup_{(x,y) \in \X\times\Y} \sup_{h \in \F} \loss(y h(x))$.
We will generally suppose $\maxloss < \infty$.
In practice, this is more often a constraint on $\F$ and $\X$
than on $\loss$: that is, we could have $\loss$ unbounded,
but due to some normalization of the functions $h \in \F$,
$\loss$ is bounded on the corresponding set of values.
For any $\g \in \Functions$ and distribution $P$ over $\X\times\Y$, 
let $\Risk(\g; P) = \E\left[ \loss(\g(X)Y) \right]$, where $(X,Y) \!\sim\! P$.
This is the \emph{$\loss$-risk} of $\g$ under $P$.
When $P \!=\! \PXY$, abbreviate this as 
$\Risk(\g) \!=\! \Risk(\g; \PXY)$.
%

We will be interested in loss functions $\loss$ whose
point-wise minimizer necessarily also optimizes the $0$-$1$ loss.  This property was
nicely characterized by \citet*{bartlett:06} as follows.
For $\eta_{0} \in [0,1]$, define $\loss^{\star}(\eta_{0}) = \inf_{z \in \exreals} (\eta_{0} \loss(z) + (1-\eta_{0}) \loss(-z))$,
and $\loss^{\star}_{-}(\eta_{0}) = \inf_{z \in \exreals : z (2\eta_{0}-1) \leq 0} (\eta_{0} \loss(z)$ $+ (1-\eta_{0})\loss(-z))$.
Then the surrogate loss 
$\loss$ is said to be \emph{classification-calibrated} if,
$\forall \eta_{0} \in [0,1] \setminus \{1/2\}$,
$\loss^{\star}_{-}(\eta_{0}) > \loss^{\star}(\eta_{0})$.
In our context, for $X \sim \Px$,
$\loss^{\star}(\eta(X))$ represents the minimum value of the conditional $\loss$-risk at $X$,
so that $\E[ \loss^{\star}(\eta(X))] = \inf_{h \in \Functions} \Risk(h)$,
while $\loss^{\star}_{-}(\eta(X))$ represents the minimum conditional $\loss$-risk at $X$,
subject to having a sub-optimal conditional error rate at $X$: i.e.,
$\sign(h(X)) \neq \sign( \eta(X) - 1/2 )$.
Thus, being classification-calibrated implies the minimizer of the conditional $\loss$-risk at $X$ necessarily has the
same sign as the minimizer of the conditional error rate at $X$.
Since we are only interested here in using $\loss$ as a reasonable surrogate for the $0$-$1$ loss,
for the remainder of this article we suppose $\loss$ is classification-calibrated.

Though not strictly necessary for our results below,
it will be convenient for us to suppose that, for all $\eta_{0} \in [0,1]$,
this infimum value $\loss^{\star}(\eta_{0})$ is actually \emph{obtained} as
$\eta_{0} \loss(z^{\star}(\eta_{0})) + (1-\eta_{0}) \loss(-z^{\star}(\eta_{0}))$
for some $z^{\star}(\eta_{0}) \in \exreals$ (not necessarily unique).
For instance, this is the case for any nonincreasing right-continuous $\loss$, 
or continuous and convex $\loss$,
which include most of the cases we are interested in using as surrogate losses anyway.
The proofs can be modified in a natural way to handle the general case, simply substituting any $z$
with conditional risk sufficiently close to the infimum value.
For any distribution $P$, denote $\target_{P}(x) = z^{\star}(\eta(x;P))$ for all $x \in \X$.
In particular, note that $\target_{P}$ obtains
$\Risk(\target_{P};P) = \inf_{\g \in \Functions} \Risk(\g;P)$.
Furthermore, since $\loss$ is classification-calibrated,
we have $\sign(\target_{P}(x)) = \sign(\eta(x;P)-1/2)$ for all $x \in \X$ with $\eta(x;P) \neq 1/2$,
and hence $\er(\target_{P};P) = \inf_{h \in \Functions} \er(h;P)$ as well.
When $P = \PXY$, we abbreviate by 
$\target = \target_{\PXY}$.

All of our main results below rely on the assumption that $\target \in \F$.
When combined with the fact that $\loss$ is classification-calibrated,
this essentially stands as a formal representation of the informal assumption 
that the surrogate loss $\loss$ was chosen wisely: that is, that functions
in $\F$ with relatively low surrogate risk necessarily have relatively low error rate.
%
%
However, it should be noted that this is often a very strong assumption, 
significantly restricting the allowed distributions $\PXY$.
For instance, 
for many losses $\loss$ in practical use (e.g., the quadratic loss),  
when $\F$ is a parametric family,
the assumption that $\target \in \F$
essentially restricts the allowed functions $\eta(\cdot)$ to also form 
a parametric family.
This fact underscores the need for great care in selecting a surrogate loss
when approaching a given learning problem in practice.
In principle, one can relax this assumption slightly, at the expense
of significantly more-complicated theorem statements,
and we include some superficial remarks on this in Appendix~\ref{app:fstar-assumption}.
However, it seems any truly-substantial relaxation 
would require a significantly different approach.

For any distribution $P$ over $\X\times\Y$, and any $h,g \in \Functions$,
define the \emph{loss distance} $\D_{\loss}(h,g;P) = \sqrt{\E\left[ \left(\loss(h(X)Y) - \loss(g(X)Y)\right)^2\right]}$, where $(X,Y) \sim P$.
Also define the \emph{loss diameter} of $\H \subseteq \Functions$ as $\D_{\loss}(\H;P) = \sup_{h,g \in \H} D_{\loss}(h,g;P)$,
and the $\loss$-risk $\eps$-minimal set of $\H$ as
$\H(\eps;\loss,P) = \{h \in \H : \Risk(h;P) - \inf_{g \in \H} \Risk(g;P) \leq \eps\}$.
When $P = \PXY$,
we abbreviate these as
$\D_{\loss}(h,g) = \D_{\loss}(h,g;\PXY)$, $\D_{\loss}(\H) = \D_{\loss}(\H;\PXY)$, and $\H(\eps;\loss) = \H(\eps;\loss,\PXY)$.
%
Also define analogous quantities for the $0$-$1$ loss. 
Define the \emph{distance} $\dist_{P}(h,g) = P((x,y) : \sign(h(x)) \neq \sign(g(x)))$
and \emph{radius} $\radius(\H;P) = \sup_{h \in \H} \dist_{P}(h,\target_{P})$.
Also define the $\eps$-minimal set of $\H$ as $\H(\eps;\zo,P) = \{h \in \H : \er(h;P) - \inf_{g \in \H}\er(g;P) \leq \eps\}$,
and for $r > 0$, define the $r$-ball centered at $h$ in $\H$ by
$\Ball_{\H,P}(h,r) = \{g \in \H : \dist_{P}(h,g) \leq r\}$.
When $P = \PXY$, we abbreviate these as
$\dist(h,g) = \dist_{\PXY}(h,g)$,
$\radius(\H) = \radius(\H;\PXY)$,
$\H(\eps;\zo) = \H(\eps;\zo,\PXY)$, and $\Ball_{\H}(h,r) = \Ball_{\H,\PXY}(h,r)$;
when $\H = \F$, further abbreviate $\Ball(h,r) = \Ball_{\F}(h,r)$.

The following definition will enable us to transform 
guarantees on the excess surrogate risk into guarantees on the excess error rate.
%
\begin{definition}
\label{defn:abstract-transform}
For any distribution $P$ over $\X \times \Y$, and any $\eps \in [0,1]$, define
\begin{equation*}
\Transform(\eps;P) = \sup \left(\{ \gamma > 0 : \Functions(\gamma;\loss,P) \subseteq \Functions(\eps;\zo,P)\} \cup \{0\}\right).
\end{equation*}
Also, for any $\gamma \in [0,\infty)$, define the inverse
\begin{equation*}
\InvTransform(\gamma;P) = \inf\left\{ \eps > 0 : \gamma \leq \Transform(\eps;P) \right\}.
\end{equation*}
When $P = \PXY$, abbreviate
$\Transform(\eps) = \Transform(\eps;\PXY)$ and $\InvTransform(\gamma) = \InvTransform(\gamma;\PXY)$.
\end{definition}
By definition, 
$\Transform$ has the property that
\begin{equation}
\label{eqn:risk-transformation}
\forall h \in \Functions, \forall \eps \in [0,1],~~\Risk(h) - \Risk(\target) < \Transform(\eps) \implies \er(h) - \er(\target) \leq \eps.
\end{equation}
In fact, $\Transform$ is defined to be maximal with this property,
in that \emph{any}
$\Transform^{\prime}$ for which \eqref{eqn:risk-transformation}
is satisfied must have $\Transform^{\prime}(\eps) \leq \Transform(\eps)$ for all $\eps \in [0,1]$.
%
For this reason, we will 
be interested in calculating lower bounds on $\Transform$. 
\citet*{bartlett:06} studied various ways to obtain concrete, calculable lower bounds of this type.
Specifically, for $\zeta \in [-1,1]$,
define $\tilde{\transform}(\zeta) = \loss^{\star}_{-}\left(\frac{1+\zeta}{2}\right) - \loss^{\star}\left(\frac{1+\zeta}{2}\right)$,
and let $\transform$ be the largest convex lower bound of $\tilde{\transform}$ on $[0,1]$,
which is well-defined in this context \citep*{bartlett:06};
for convenience, also define $\transform(x)$ for $x \in (1,\infty)$ arbitrarily, subject to maintaining convexity of $\transform$.
\citet*{bartlett:06} show $\transform$ is continuous and nondecreasing on $(0,1)$,
and in fact that $x \mapsto \transform\left(x\right) / x$ is nondecreasing on $(0,\infty)$.
They also show every $h \in \Functions$ has $\transform(\er(h) - \er(\target)) \leq \Risk(h) - \Risk(\target)$,
so that $\transform \leq \Transform$, 
and they find this inequality can be tight for a particular choice of $\PXY$.
They further study more subtle relationships between excess $\loss$-risk and excess error rate holding for any classification-calibrated $\loss$.
In particular, 
following the argument in the proof of their Theorem 3, one can show that 
$\forall h \in \Functions$,
\begin{equation*}
\dist(h,\target) \cdot \transform\left( \frac{\er(h) - \er(\target)}{2 \dist(h,\target)} \right) \leq \Risk(h) - \Risk(\target).
\end{equation*}
The implication of this in our context is the following.
Fix any nondecreasing function $\BJM : [0,1] \to [0,\infty)$ such that $\forall \eps \geq 0$,
\begin{equation}
\label{eqn:bjm-transform}
\BJM(\eps) \leq \radius(\Functions(\eps;\zo)) \transform\left(\frac{\eps}{2 \radius(\Functions(\eps;\zo))}\right).
\end{equation}
Any $h \in \Functions$ with $\Risk(h) - \Risk(\target) < \BJM(\eps)$ also has
$\dist(h,\target) \transform\left(\frac{\er(h)-\er(\target)}{2\dist(h,\target)}\right) < \BJM(\eps)$;
combined with the fact that $x \mapsto \transform(x)/x$ is nondecreasing on $(0,1)$, this implies
$\radius(\Functions(\er(h)-\er(\target);\zo)) \transform\left(\frac{\er(h)-\er(\target)}{2\radius(\Functions(\er(h)-\er(\target);\zo))}\right) < \BJM(\eps)$;
this means $\BJM(\er(h) - \er(\target)) < \BJM(\eps)$,
and monotonicity of $\BJM$ implies $\er(h) - \er(\target) < \eps$.
Altogether, this implies $\BJM(\eps) \leq \Transform(\eps)$,
so that $\Risk(h) - \Risk(\target) < \BJM(\eps) \implies \er(h) - \er(\target) < \eps$.
In fact, though we do not present the details here,
with only minor modifications to the proofs below,
when $\target \in \F$,
all of our results involving $\Transform(\eps)$
also hold while replacing $\Transform(\eps)$ with
any nondecreasing $\BJM^{\prime}$ s.t. $\forall \eps \geq 0$,
$\BJM^{\prime}(\eps) \leq \radius(\F(\eps;\zo)) \transform\left(\frac{\eps}{2 \radius(\F(\eps;\zo))}\right)$,
which can sometimes lead to tighter results.

Some of our stronger results below will be stated for a restricted family of losses,
originally explored by \citet*{bartlett:06}: namely, smooth losses with convexity
quantified by a polynomial, as described in the following condition.

\begin{condition}
\label{con:strong-convexity}
$\F$ is convex, with $\forall x \in \X, \sup_{f \in \F} |f(x)| \leq \Fbound$ for some constant $\Fbound \in (0,\infty)$,
and there exists a pseudometric
$\metric : [-\Fbound,\Fbound]^{2} \to [0,\metricbound]$ for some constant $\metricbound \in (0,\infty)$,
and constants $\Lip, \convc \in (0,\infty)$ and $\convr \in (0,\infty]$ such that
$\forall x,y \in [-\Fbound,\Fbound], |\loss(x)-\loss(y)| \leq L \metric(x,y)$,
and the function
\begin{equation*}
  \modconv(\eps) \!=\! \inf\!\left(\left\{ \frac{1}{2}\loss(x)\!+\!\frac{1}{2}\loss(y) - \loss\!\left(\frac{1}{2}x\!+\!\frac{1}{2}y\right) : x,y \!\in\! [-\Fbound,\!\Fbound], \metric(x,y) \geq \eps\right\} \!\cup\!\{\infty\}\right)
\end{equation*}
satisfies
$\forall \eps \in [0,\infty)$, $\modconv(\eps) \geq \convc \eps^{\convr}$.
\end{condition}

In particular, note that if $\F$ is convex, and the functions in $\F$ are uniformly bounded, and $\loss$ is convex and continuous,
then Condition~\ref{con:strong-convexity} is always satisfied (though possibly with $\convr = \infty$) by taking $\metric(x,y) = |x-y|/(4\Fbound)$.

\subsection{A Few Examples of Loss Functions}
\label{subsec:loss-examples}

Here we briefly mention a few loss functions $\loss$ in common practical use,
all of which are classification-calibrated.
These examples are taken directly from the work of \citet*{bartlett:06},
which additionally discusses many other interesting examples of classification-calibrated
loss functions and their corresponding $\transform$ functions.

\paragraph{Example 1}
The \emph{quadratic loss} (or squared loss), specified as $\loss(x) = (1-x)^{2}$,
is often used in so-called \emph{plug-in} classifiers \citep*{audibert:07},
which approach the problem of learning a classifier by estimating
the regression function $\E[Y|X=x] = 2\eta(x)-1$, and then taking the sign of this estimator
to get a binary classifier.  The quadratic loss has the convenient
property that for any distribution $P$ over $\X\times\Y$,
$\target_{P}(\cdot) = 2\eta(\cdot;P)-1$, so that it is straightforward to describe
the set of distributions $P$ satisfying the assumption $\target_{P} \in \F$.
In classification, this loss is sometimes modified as $\loss(x) = \max\{1-x,0\}^{2}$,
called the \emph{truncated} quadratic loss.
\citet*{bartlett:06} show that for the quadratic loss (with or without truncation),
$\transform(x) = x^{2}$, and Condition~\ref{con:strong-convexity} is satisfied
with $\Lip = 2(\Fbound+1)$, $\convc = 1/4$, 
\ignore{and}
$\convr = 2$.
\ignore{In fact, they also study the general family of losses
$\loss(x) = |1-x|^{p}$, for $p \in (1,\infty)$, and show 
$\transform(x)$ and $\convr$ exhibit a range of behaviors varying with $p$.}

\paragraph{Example 2}
The \emph{exponential loss} is specified as $\loss(x) = e^{-x}$.
This loss function appears in many contexts in machine learning;
for instance, the popular AdaBoost method can be viewed as
an algorithm that greedily optimizes the exponential loss \citep*{freund:97b}.
\citet*{bartlett:06} show that under the exponential loss, $\target(x) = \frac{1}{2} \ln\!\left(\frac{\eta(x)}{1-\eta(x)}\right)$ 
and $\transform(x) = 1 - \sqrt{1-x^2}$, which is tightly approximated
by $x^2/2$ for small $x$.  They also show this loss satisfies the
conditions on $\loss$ in Condition~\ref{con:strong-convexity}
with $\metric(x,y)=|x-y|$, $\Lip = e^{\Fbound}$, $\convc = e^{-\Fbound} / 8$,
and $\convr = 2$.  Note, however, that for noise-free
distributions, we would need $\target(x) = \pm \infty$,
which means most common function classes $\F$ could not be
expected to contain $\target$ for this loss in the noise-free case.

\paragraph{Example 3}
The \emph{hinge loss}, specified as $\loss(x) = \max\left\{1-x,0\right\}$,
is another common surrogate loss in machine learning practice today.
For instance, it is used in the objective of the Support Vector Machine
(along with a regularization term) \citep*{cortes:95}.
\citet*{bartlett:06} show that for the hinge loss, $\target(x) = \sign(\eta(x)-1/2)$ and 
$\transform(x) = |x|$.  The hinge loss is Lipschitz continuous, with Lipschitz constant $1$.
However, for the remaining conditions on $\loss$ in Condition~\ref{con:strong-convexity},
any $x,y \leq 1$ have $\frac{1}{2}\loss(x) + \frac{1}{2}\loss(y) = \loss(\frac{1}{2}x + \frac{1}{2}y)$,
so that $\modconv(\eps) = 0$; hence, $\convr=\infty$ is required.

\section{Methods Based on Optimizing the Surrogate Risk}
\label{sec:optimizing-surrogate}

Perhaps the simplest way to use a surrogate loss
is to optimize $\Risk(h)$ over $h \in \F$
until identifying $h \in \F$ with
$\Risk(h) - \Risk(\target) < \Transform(\eps)$,
at which point we are guaranteed $\er(h) - \er(\target) \leq \eps$.
In this section,
we introduce a classic passive learning method based on this strategy,
and discuss the potential drawbacks of this
approach for active learning.

\subsection{Passive Learning: Empirical Risk Minimization}
\label{subsec:erm}

In the context of passive learning,
the method of \emph{empirical $\loss$-risk minimization}
is one of the most-studied methods for optimizing $\Risk(h)$
over $h \in \F$.  To define this method, we first introduce some notation.
For any $m \in \nats$, $\g \!:\! \X \to \exreals$,
and $S \!=\! \{(x_1,y_1),\ldots,(x_m,y_m)\}$ $\in (\X\times\Y)^{m}$,
we overload the $\Risk(\g;\cdot)$ notation, 
defining the \emph{empirical} $\loss$-\emph{risk} as
$\Risk(\g;S) = m^{-1} \sum_{i=1}^{m} \loss(\g(x_i)y_i)$:
that is, $\Risk(\g;S)$ is the $\loss$-risk of $\g$ under the uniform distribution on $S$.
At times it will be convenient to keep track of the indices for a subsequence of $\Data$,
and for this reason we further overload the notation, so that for any
$Q = \{(i_1,y_1),\ldots,(i_m,y_m)\} \in (\nats \times \Y)^{m}$,
we define $S[Q] = \{(X_{i_1},y_1),\ldots,(X_{i_m},y_m)\}$
and $\Risk(\g;Q) = \Risk(\g;S[Q])$.
For completeness, we also generally define $\Risk(\g;\emptyset) = 0$.

The method of empirical $\loss$-risk minimization, here denoted by
$\ERM_{\loss}(\H,\Data_{m})$,
is characterized by the property that it returns $\hat{h} = \argmin_{h \in \H} \Risk(h;\Data_{m})$.
This is a well-studied and classical passive learning method, presently in popular
use in applications, and as such it will serve as our baseline passive learning method for comparison.
We review several known performance guarantees for $\ERM_{\loss}$ below.

\subsection{Negative Results for Active Learning}
\label{subsec:negative-results}

As mentioned, there are several active learning methods designed to
optimize a general loss function \citep*{beygelzimer:09,koltchinskii:10}.
However, it turns out that for many interesting loss functions,
the number of labels required for active learning to achieve a given
excess surrogate risk value is not significantly smaller than that
sufficient for passive learning by $\ERM_{\loss}$.

Specifically, 
consider a problem with $\X = \{x_0,x_1\}$,
a fixed $\Fbound \in (0,\infty)$, 
and $\F$ as the set of all functions $f$ with 
$(f(x_{0}),f(x_{1})) \in [-\Fbound,\Fbound] \times (0,\Fbound]$.
Let $z \in (0,1/2)$ be a constant, 
let $\eta(x_1) = 1/2+z$, and suppose that $\loss$ is a 
classification-calibrated loss with $\maxloss < \infty$ such that for any 
$\eta(x_0) \in [4/6,5/6]$, we have $\target \in \F$
(the latter condition could equivalently be stated as a constraint on $\Fbound$).
%
Given a small value $\eps \in (0,z)$, let 
$\Px(\{x_1\}) = \eps / (2z)$,
$\Px(\{x_0\}) = 1-\Px(\{x_1\})$.
For this problem, any function $h$ with 
$\sign(h(x_1)) = -1$
has $\er(h) - \er(\target) \geq \eps$, so that
$\Transform(\eps) \leq (\eps / (2z)) ( \loss^{\star}_{-}(\eta(x_1)) - \loss^{\star}(\eta(x_1)))$;
since $\loss$ is classification-calibrated and $\maxloss < \infty$, 
this implies $\Transform(\eps) \leq c \eps$, for some $\loss$-dependent $c \in (0,\infty)$.
Any function $h$ with $\Risk(h)-\Risk(\target) \leq c \eps$ 
for this problem must have 
$\E[\loss(h(X)Y)|X=x_{0}] - \E[\loss(\target(X)Y)|X=x_{0}] \leq 
c\eps / \Px(\{x_0\}) = O(\eps)$.
Existing results of \citet*{hanneke:10b} 
(with a slight modification to rescale for $\eta(x_0) \in [4/6,5/6]$)
imply that,
for many classification-calibrated losses $\loss$,
the minimax optimal number of labels sufficient for 
an active learning algorithm to achieve this
latter guarantee is 
$\Theta(1/\eps)$.
\citet*{hanneke:10b} specifically show this for losses 
$\loss$ that are strictly positive, 
decreasing, strictly convex, and
twice differentiable with continuous second derivative;
however, that result can easily be extended
to a wide variety of other classification-calibrated
losses, such as the quadratic loss, which satisfy
these conditions in a neighborhood of $0$.
It is also known \citep*{bartlett:06} (see also below)
that for many such losses (specifically, those satisfying
Condition~\ref{con:strong-convexity} with $\convr = 2$),
$\Theta(1/\eps)$ random labeled samples are sufficient
for $\ERM_{\loss}$ to achieve this same guarantee, so
that error bounds based purely on the surrogate risk of the function
produced by an active learning method in this scenario can be at most a constant factor smaller
than those provable for passive learning methods.

Below, we provide an active learning algorithm 
and analysis of its performance which, in the scenario above
(with $\convr=2$), guarantees 
expected
excess error rate less than
$\eps$, 
using a number of label requests $O(\log(1/\eps)\log\log(1/\eps))$.
The implication is that, to identify the improvements achievable by
active learning with a surrogate loss, it is not sufficient to merely analyze the surrogate
risk of the function produced by a given active learning algorithm.
Indeed, since we are not particularly interested in the surrogate risk
itself, we may even consider active learning algorithms that do not
actually optimize $\Risk(h)$ over $h \in \F$ (even in the limit).

\section{Alternative Use of the Surrogate Loss}
\label{sec:acal}

Given that we are interested in $\loss$ only insofar as it helps us to optimize the error rate with computational efficiency,
we might ask whether there is a method that makes more effective use of $\loss$ for optimizing the error rate, 
while maintaining the computational advantages.
To explore this question, we propose the following method, which generalizes the methods of \citet*{koltchinskii:10} and \citet*{hanneke:12a}.
Results similar to those proven below should also hold for analogous generalizations of the related methods of
\citep*{balcan:09,dasgupta:07,beygelzimer:09}.

\begin{bigboxit}
\ACAL :\\
Input:
surrogate loss $\loss$,
unlabeled sample budget $u$, labeled sample budget $n$\\
Output: classifier $\hat{h}$\\
{\vskip -2mm}\line(1,0){332}\\
0. $V \gets \F$, $Q \gets \{\}$, $m \gets 1$, $t \gets 0$ \\ 
1. While $m < u$ and $t < n$\\
2. \quad $m \gets m+1$\\
3. \quad If $X_m \in \DIS(V)$\\
4. \qquad Request label $Y_m$ and let $Q \gets Q \cup \{(m,Y_m)\}$, $t \gets t+1$\\
5. \quad If $\log_{2}(m) \in \nats$ \\ 
6. \qquad $V \gets \left\{ h \in V : \Risk(h ; Q) - \inf_{g \in V} \Risk(g;Q) \leq \hat{T}_{\loss}(V;Q,m) \right\}$\\
7. \qquad $Q \gets \{\}$ \\ 
8. Return $\hat{h} = \argmin_{h \in V} \Risk(h; Q)$
\end{bigboxit}

The intuition behind this algorithm is that, since we are only interested in achieving low error rate,
once we have identified $\sign(\target(x))$ for a given $x \in \X$, there is no need to further 
optimize the value $\E[\loss(\hat{h}(X)Y) | X = x]$.
Thus, as long as we maintain $\target \in V$,  the data points $X_m \notin \DIS(V)$ are typically less informative than those $X_m \in \DIS(V)$.
We therefore focus the label requests on those $X_m \in \DIS(V)$, since there remains some uncertainty about $\sign(\target(X_m))$ for these points.
The algorithm updates $V$ periodically (Step 6), removing those functions $h$ whose
excess empirical risks (under the current sampling distribution) are relatively large; by setting 
this threshold $\hat{T}_{\loss}$ appropriately, we can guarantee the excess empirical risk of 
$\target$ is smaller than $\hat{T}_{\loss}$.  Thus, the algorithm maintains $\target \in V$ as an invariant,
while shrinking the sampling region $\DIS(V)$.  The actual definition of $\hat{T}_{\loss}$ sufficient
for the results stated below will be specified in Section~\ref{subsec:hatT-spec} below, based on data-dependent
concentration inequalities.

In practice, the set $V$ can be maintained implicitly,
simply by keeping track of the constraints (Step 6) that
define it.  Then the condition in Step 3 can be checked
by solving two constraint satisfaction problems (one for each
sign).  Likewise, the value $\inf_{g \in V} \Risk(g;Q)$ in these constraints, 
as well as the final $\hat{h}$, 
can be found by solving constrained optimization problems.
Thus, for convex loss functions and convex finite-dimensional classes of function, these steps typically
have computationally efficient realizations as convex optimization problems, 
as long as the $\hat{T}_{\loss}$ values can also be obtained efficiently. 

We include general results on the performance of \ACAL~in Section~\ref{sec:abstract} below. 
For now, we briefly sketch the main ideas of the analysis, in rough outline.
For any measurable $\region \subseteq \X$,
and any $h,g \in \Functions$,
define the spliced function $h_{\region,g}(x) = h(x) \ind_{\region}(x) + g(x) \ind_{\X\setminus\region}(x)$.
For a set $\H \subseteq \Functions$, denote $\H_{\region,g} = \{ h_{\region,g} : h \in \H\}$.
In the special case $g=\target$, we abbreviate these as 
$h_{\region} = h_{\region,\target}$ and $\H_{\region} = \{h_{\region} : h \in \H\}$.
As mentioned, the idea in the analysis is to argue that \ACAL~maintains $\target \in V$,
while also removing from $V$ any function with relatively large error rate,
within a certain number of rounds.
%
More explicitly, upon reaching $m$ satisfying the condition in Step 5,
if we denote $\L_{m} = \{(1+m/2,Y_{1+m/2}),\ldots,(m,Y_{m})\}$,
then since every $(m^{\prime},Y_{m^{\prime}}) \in \L_{m}$ is either in $Q$ 
or else $X_{m^{\prime}} \notin \DIS(V)$, 
every 
$h \in V$
has
$(\Risk(h;Q) - \inf_{\!g \in V} \!\Risk(g;Q))|Q| \!=\! (\Risk(h_{\DIS(V)};\L_{m}) - \inf_{\!g \in V} \!\Risk(g_{\DIS(V)};\L_{m})) \frac{m}{2}$.
We therefore define $\hat{T}_{\loss}(V;Q,m)$ to provide a concentration inequality
$\Risk(\target;\L_{m}) - \inf_{g \in V} \Risk(g_{\DIS(V)};\L_{m}) \leq \frac{2|Q|}{m} \hat{T}_{\loss}(V;Q,m)$,
thus maintaining that $\target \in V$ in Step 6.
This also implies that, if $V_{\DIS(V)} \subseteq \Bracket{\F}(2^{2-j};\loss)$ upon reaching Step 5 (for some $j \in \ints$), 
then $V \subseteq \F(\InvTransform(2^{2-j});\zo)$.
One can then show that, upon reaching $m$ of a certain size $u_{j}$ (quantified below), the value $\frac{2|Q|}{m}\hat{T}_{\loss}(V;Q,m)$
will be small enough that, in combination with concentration of $\Risk(h_{\DIS(V)};\L_{m})$ values,
after the update in Step 6,
only functions $h \in V$ with $\Risk(h_{\DIS(V)}) - \Risk(\target) < 2^{-j}$ will remain:
that is, after the update, $V_{\DIS(V)} \subseteq \Bracket{\F}(2^{-j};\loss)$.
By induction, upon reaching $m$ of a sufficiently large size $u_{j_{\eps}}$ (quantified below),
every $h \in V$ has $\Risk(h_{\DIS(V)})-\Risk(\target) < \Transform(\eps)$,
which implies $\er(h) - \er(\target) \leq \eps$.
This provides a sufficient size of $u$ to obtain excess error rate $\eps$.
%
Next, we note that the algorithm requests a label $Y_{m}$ only if $X_{m} \in \DIS(V)$.
The above reveals that, if $u_{j-1} < m \leq u_{j}$, 
then $V \subseteq \F(\InvTransform(2^{2-j});\zo)$,
which implies $\DIS(V) \subseteq \DIS(\F(\InvTransform(2^{2-j});\zo))$.  Thus, the number
of labels the algorithm requests among indices $m$ with 
$u_{j-1} < m \leq u_{j}$ is at most the number with $X_{m} \in \DIS(\F(\InvTransform(2^{2-j});\zo))$,
a number which can easily by upper bounded by a simple Chernoff bound.
This provides a sufficient size of $n$ for the algorithm to obtain excess error rate $\eps$.

The number of label requests sufficient for \ACAL~to obtain excess error rate $\eps$
can often (though not always) be significantly smaller
than the number of random labeled data points
sufficient for $\ERM_{\loss}$ to achieve the same. 
This is typically the case when $\Px(\DIS(\F(\eps;\zo))) \to 0$ as $\eps \to 0$.
When this is the case, the number of labels requested by the algorithm is sublinear
in the number of unlabeled samples it processes.
Not surprisingly, the magnitude of the improvements of \ACAL~over $\ERM_{\loss}$
can be quantified in terms of the \emph{rate} at which $\Px(\DIS(\F(\eps;\zo)))$ vanishes as $\eps \to 0$.
In the next section, we quantify this rate in terms of 
a complexity measure known as the disagreement coefficient.

\section{Main Results}
\label{sec:explicit}

We provide a general analysis of \ACAL~in Section~\ref{subsec:abstract} below.
For now, we summarize a few of the most interesting implications of that analysis,
under commonly-studied complexity conditions: namely, 
VC subgraph classes and entropy conditions.
Detailed derivations for all of these results (from the abstract theorems) 
are included in Section~\ref{sec:derivations} below.
Appendix~\ref{subsec:vc-major} further includes a brief discussion of VC major classes and VC hull classes. 
In the interest of making the results more concise and explicit,
we express them in terms of well-known conditions relating distances to excess risks.
We also express them in terms of a lower bound on $\Transform(\eps)$ of the type in \eqref{eqn:bjm-transform},
with convenient properties that allow for closed-form expression of the results.
Throughout, we use the convenient notation $\Log(x) = \max\{\ln(x),1\}$, defined for all $x \in (0,\infty)$.

\subsection{Diameter Conditions}
\label{subsec:noise}

To begin, we first state some general characterizations relating distances to excess risks;
these characterizations will make it easier to express our results more concretely below,
and make for a more straightforward comparison between results for the above methods.
The following condition, introduced by \citet*{mammen:99} and \citet*{tsybakov:04},
is a well-known noise condition, about which there is now an extensive literature
\citep*[e.g.,][]{koltchinskii:06,bartlett:06,hanneke:11a,hanneke:12a}.

\begin{condition}
\label{con:tsybakov-01}
For some $\tsybca \in [1,\infty)$ and $\tsyba \in [0,1]$,
for every $\g \in \Functions$,
\begin{equation*}
\dist\left(\g, \target\right) \leq \tsybca \left(\er(\g) - \er(\target)\right)^{\tsyba}.
\end{equation*}
\end{condition}

Condition~\ref{con:tsybakov-01} is equivalently expressed in terms of certain noise conditions
\citep*{mammen:99,tsybakov:04,bartlett:06}.
Specifically, satisfying Condition~\ref{con:tsybakov-01} with some $\tsyba < 1$ is equivalent
to the existence of some $\tsybca^{\prime} \in [1,\infty)$ such that, for all $\eps > 0$,
$\Px\left( x : |\eta(x) - 1/2| \leq \eps\right) \leq \tsybca^{\prime} \eps^{\tsyba / (1-\tsyba)}$,
which is often referred to as a \emph{low noise} condition.
Additionally, satisfying Condition~\ref{con:tsybakov-01} with $\tsyba=1$ is equivalent to having
some $\tsybca^{\prime} \in [1,\infty)$ such that
$\Px\left( x : |\eta(x) - 1/2| \leq 1 / \tsybca^{\prime} \right) = 0$,
often referred to as a \emph{bounded noise} condition.

For simplicity, we formulate our results in terms of $\tsybca$ and $\tsyba$ from Condition~\ref{con:tsybakov-01}.
However, for the abstract results in this section, the results remain valid under the weaker condition
that replaces $\Functions$ by $\F$, and adds the condition that $\target \in \F$.
In fact, the specific results in this section also remain valid using this weaker condition while additionally 
replacing \eqref{eqn:bjm-transform} with
the $\F$-specific $\BJM^{\prime}$ requirement mentioned in Section~\ref{subsec:surrogate-losses},
as remarked above.

An analogous condition can be defined for the surrogate loss function, as follows.
Essentially-similar notions have been explored by \citet*{bartlett:06} and \citet*{koltchinskii:06}.

\begin{condition}
\label{con:tsybakov-sur}
For some $\tsybcb \in [1,\infty)$ and $\tsybb \in [0,1]$, for every $\g \in \Bracket{\F}$, 
\begin{equation*}
\D_{\loss}\left( \g, \target_{P} ; P\right)^{2}
\leq \tsybcb \left( \Risk(\g ; P) - \Risk(\target_{P} ; P) \right)^{\tsybb}.
\end{equation*}
\end{condition}

Note that 
these conditions are \emph{always} satisfied for \emph{some} values of $\tsybca,\tsybcb,\tsyba,\tsybb$,
since $\tsyba=\tsybb=0$ trivially satisfies the conditions.  However, in more benign
scenarios, values of $\tsyba$ and $\tsybb$ strictly greater than $0$ can be satisfied.
Furthermore, for some loss functions $\loss$,
Condition~\ref{con:tsybakov-sur} can even be satisfied \emph{universally},
in the sense that it holds for a particular value of $\tsybb > 0$ for \emph{all} distributions.
In particular, \citet*{bartlett:06} show that this is the case under Condition~\ref{con:strong-convexity},
as stated in the following lemma (see \citep*{bartlett:06} for the proof).

\begin{lemma}
\label{lem:bjm-strong-convexity}
Suppose Condition~\ref{con:strong-convexity} is satisfied.
Let $\tsybcb = (2 \convc \metricbound^{\min\{\convr-2,0\}})^{-\tsybb} \Lip^2$ and $\tsybb = \min\{1,\frac{2}{\convr}\}$.
Then \emph{every} distribution $P$ over $\X\times\Y$ with $\target_{P} \in \Bracket{\F}$ 
satisfies Condition~\ref{con:tsybakov-sur} with these values of $\tsybcb$ and $\tsybb$.
\end{lemma}

Under Condition~\ref{con:tsybakov-01}, it is particularly straightforward to obtain bounds on $\Transform(\eps)$
based on a function $\BJM(\eps)$ satisfying \eqref{eqn:bjm-transform}.  For instance, since
$x \mapsto x \transform(1/x)$ is nonincreasing on $(0,\infty)$ \citep*{bartlett:06}, the function
\begin{equation}
\label{eqn:bjm-tsybakov-transform}
\BJM(\eps) = \tsybca \eps^{\tsyba} \transform\left( \eps^{1-\tsyba}/(2\tsybca)\right)
\end{equation}
satisfies $\BJM(\eps) \leq \Transform(\eps)$ \citep*{bartlett:06}.
Furthermore, for classification-calibrated $\loss$,
$\BJM$ in \eqref{eqn:bjm-tsybakov-transform}
is strictly increasing, nonnegative, and continuous on $(0,1)$ \citep*{bartlett:06}, 
and has $\BJM(0) = 0$;
thus, the inverse,
defined for $\gamma > 0$ by
$\BJM^{-1}(\gamma) = \inf(\{\eps > 0 : \gamma \leq \BJM(\eps)\}\cup\{1\})$,
is strictly increasing, nonnegative, and continuous on $(0,\BJM(1))$. 
Furthermore, one can easily show $x \mapsto \BJM^{-1}(x) / x$ is nonincreasing on $(0,\infty)$.
Also note that $\forall \gamma > 0, \InvTransform(\gamma) \leq \BJM^{-1}(\gamma)$.

For any distribution $P$ over $\X\times\Y$ and any $\H \subseteq \Bracket{\F}$ with $\target_{P} \in \H$,
let
\begin{align}
\G_{\H} &= \{ (x,y) \mapsto \loss(h(x)y) : h \in \H\}, \notag
\\ \text{and }
\G_{\H,P} & = \{ (x,y) \mapsto \loss(h(x)y) - \loss(\target_{P}(x)y) : h \in \H\}. \label{eqn:G-sub-H}
\end{align}
Below, we let $\FunctionsXY$ denote the set of measurable functions $g : \X\times\Y \to \exreals$.
Also, for $\G \subseteq \FunctionsXY$, let $\Env(\G) = \sup_{g \in \G}|g|$ denote the minimal \emph{envelope} function for $\G$, 
and for $g \in \FunctionsXY$ let $\|g\|_{P}^{2} = \int g^2 {\rm d}P$ denote the squared $L_{2}(P)$ seminorm of $g$;
we will generally assume $\Env(\G)$ is measurable in the discussion below.

\subsection{The Disagreement Coefficient}
\label{subsec:disagreement-coefficient}

In order to more concisely state our results, it will be convenient to
bound $\Px(\DIS(\H))$ by a linear function of $\radius(\H)$, for $\radius(\H)$ in a given range.
This type of relaxation has been used extensively in the active learning
literature \citep*{hanneke:07b,hanneke:thesis,hanneke:11a,hanneke:12a,wang:11,koltchinskii:10,friedman:09,beygelzimer:09,hanneke:10a,dasgupta:07,raginsky:11},
and the coefficient in the linear function is typically referred to as the \emph{disagreement coefficient}.
Specifically, the following definition is due to \citet*{hanneke:07b,hanneke:11a}; related quantities
have been explored by \citet*{alexander:87} and \citet*{gine:06}.

\begin{definition}
\label{defn:disagreement-coefficient}
For any $r_0 > 0$,
define the \emph{disagreement coefficient} of a function $h : \X \to \reals$ with respect to $\F$ under $\Px$ as
\begin{equation*}
\dc_{h}(r_0) = \sup_{r > r_0} \frac{\Px(\DIS(\Ball(h,r)))}{r} \lor 1.
\end{equation*}
If $\target \in \F$, define the disagreement coefficient of the class $\F$ as
$\dc(r_0) = \dc_{\target}(r_0)$.
\end{definition}

The value of $\dc(\eps)$ has been studied and bounded for various function classes $\F$ under various conditions on $\Px$.
In many cases of interest, $\dc(\eps)$ is known to be bounded by a finite constant \citep*{hanneke:07b,hanneke:11a,hanneke:10a,friedman:09,mahalanabis:11},
while in other cases, $\dc(\eps)$ may have an interesting dependence on $\eps$ \citep*{wang:11,hanneke:10a,raginsky:11}.
The reader is referred to the works of \citet*{hanneke:11a,hanneke:12a} for detailed discussions on the disagreement coefficient.

\subsection{VC Subgraph Classes}
\label{subsec:vc}

We begin with results for VC subgraph classes.
For a collection $\mathcal{A}$ of sets,
a set of points $\{z_1,\ldots,z_k\}$ is said to be \emph{shattered}
by $\mathcal{A}$ if $|\{ A \cap \{z_1,\ldots,z_k\} : A \in \mathcal{A}\}| = 2^{k}$.
The VC dimension $\vc(\mathcal{A})$ of $\mathcal{A}$ is then defined as the
largest integer $k$ for which there exist $k$ points $\{z_1,\ldots,z_k\}$
shattered by $\mathcal{A}$ \citep*{vapnik:71}; if no such largest $k$ exists,
we define $\vc(\mathcal{A}) = \infty$.
For a set $\G$ of real-valued functions,
denote by $\vc(\G)$ the VC dimension of the collection
$\{ \{ (x,y) : y < g(x)\} : g \in \G\}$
of subgraphs of functions in $\G$ (called the pseudo-dimension \citep*{pollard:90,haussler:92});
to simplify the 
results below,
we adopt the convention that when the VC dimension of this collection is $0$,
we let $\vc(\G) = 1$. 
$\G$ is said to be a \emph{VC subgraph} class if $\vc(\G) < \infty$ \citep*{van-der-Vaart:96}.

Because we are interested in results concerning values of $\Risk(h)-\Risk(\target)$,
for functions $h$ in certain subsets $\H \subseteq \Bracket{\F}$,
we will formulate results below in terms of $\vc(\G_{\H})$. 
In some special cases, such as monotonic $\loss$, 
these results can 
be rephrased directly in terms of $\vc(\H)$ if desired
\citep*[e.g.,][]{dudley:87,haussler:92}.

Following \citet*{gine:06}, for $r > 0$, define $\Ball_{\H,P}(\target_{P}, r ; \loss) = \{g \in \H : \D_{\loss}(g,\target_{P} ; P)^{2} \leq r\}$, 
and for $r_{0} \geq 0$, define
\begin{equation*}
\capacity_{\loss}(r_{0} ; \H, P) = \sup_{r > r_{0}} \frac{\left\|\Env\left(\G_{\Ball_{\H,P}(\target_{P}, r ; \loss),P} \right)\right\|_{P}^{2}}{r} \lor 1.
\end{equation*}
When $P = \PXY$, abbreviate this as $\capacity_{\loss}(r_{0} ; \H) = \capacity_{\loss}(r_{0} ; \H, \PXY)$,
and when $\H = \F$, further abbreviate $\capacity_{\loss}(r_{0}) = \capacity_{\loss}(r_{0} ; \F, \PXY)$.

We can now state 
the following theorem, providing a sample size sufficient for $\ERM_{\loss}$ to obtain excess error rate $\eps$.
This result is implicit in the work of \citet*{gine:06}.
\begin{theorem}
\label{thm:vc-erm}
For a universal constant $c \in [1,\infty)$,
if $\PXY$ satisfies Condition~\ref{con:tsybakov-01} and Condition~\ref{con:tsybakov-sur},
$\loss$ is classification-calibrated,
$\target \in \F$, and $\BJM$ is as in \eqref{eqn:bjm-tsybakov-transform},
then for any $\eps \in (0,1)$, letting $\capacity_{\loss} = \capacity_{\loss}\left( \tsybcb \BJM(\eps)^{\tsybb} \right)$,
for any $m \in \nats$ with
\begin{equation}
\label{eqn:vc-erm-m}
m \geq c \left(\frac{\tsybcb}{\BJM(\eps)^{2-\tsybb}} + \frac{\maxloss}{\BJM(\eps)}\right) \left( \vc(\G_{\F}) \Log\left(\capacity_{\loss}\right) + \Log\left(1/\conf\right)\right),
\end{equation}
with probability at least $1-\conf$,
$\ERM_{\loss}(\F,\Data_{m})$ produces $\hat{h}$ with $\er(\hat{h}) - \er(\target) \leq \eps$.
\end{theorem}

As noted by \citet*{gine:06},
in the special case when $\loss$ is itself the $0$-$1$ loss ($\loss = \ind_{[-\infty,0]}$) and $\F$ is a set of $\{-1,+1\}$-valued classifiers,
\eqref{eqn:vc-erm-m}
simplifies quite nicely, since then
$\|\Env(\G_{\Ball_{\F,\PXY}(\target, r ; \loss),\PXY})\|_{\PXY}^{2}$ $= \Px\left(\DIS\left(\Ball\left(\target,r\right)\right)\right)$,
so that $\capacity_{\loss}(r_{0}) = \dc(r_{0})$;
in this case, we also have 
$\vc(\G_{\F}) = \vc(\F)$ 
and $\BJM(\eps) = \eps/2$, and we can take $\tsybb = \tsyba$ and $\tsybcb = \tsybca$,
so that it suffices to have
\begin{equation*}
m \geq c \tsybca \eps^{\tsyba-2} \left( \vc(\F) \Log\left( \dc \right) + \Log\left(1 / \conf\right)\right),
\end{equation*}
where $\dc = \dc\left( \tsybca \eps^{\tsyba}\right)$ and $c \in [1,\infty)$ is a universal constant.
This is sometimes proportional to the minimax 
number of samples 
for passive learning \citep*{castro:08,hanneke:11a,raginsky:11}.

Next, we turn to the analysis of \ACAL~under these same conditions.
Suppose $\PXY$ satisfies Conditions \ref{con:tsybakov-01} and \ref{con:tsybakov-sur},
and for $\gamma_{0} \geq 0$, define
\begin{equation*}
\mixcap(\gamma_{0}) = \sup_{\gamma > \gamma_{0}} \frac{\Px\left( \DIS\left( \Ball\left(\target, \tsybca \InvTransform\left( \gamma \right)^{\tsyba}\right) \right) \right)}{\tsybcb \gamma^{\tsybb}} \lor 1.
\end{equation*}
We claim the following theorem, bounding the number of samples (labeled and unlabeled) sufficient for \ACAL~to obtain excess error
rate $\eps$, under the same conditions as Theorem~\ref{thm:vc-erm}.
As mentioned above, the specific definition of $\hat{T}_{\loss}$
sufficient for this theorem will be formally specified in Section~\ref{subsec:hatT-spec}.
Also, the specification of $\hat{\sfun}$ will be given in the proof, in Appendix~\ref{app:applications}.

\begin{theorem}
\label{thm:vc-subgraph-abstract}
For a universal constant $c \in [1,\infty)$,
if $\PXY$ satisfies Condition~\ref{con:tsybakov-01} and Condition~\ref{con:tsybakov-sur},
$\loss$ is classification-calibrated,
$\target \in \F$, and $\BJM$ is as in \eqref{eqn:bjm-tsybakov-transform},
for any $\eps \in (0,1)$,
letting $\dc = \dc\left( \tsybca \eps^{\tsyba} \right)$, $\mixcap = \mixcap(\BJM(\eps))$,
$A_{1} = \vc(\G_{\F}) \Log(\mixcap \maxloss) + \Log(1/\conf)$,
$C_{1} = \min\left\{\frac{1}{1-2^{(\tsyba-1)}}, \Log(\maxloss/\BJM(\eps))\right\}$,
and $B_{1} = \min\left\{C_{1}, \frac{1}{1-2^{(\tsybb-1)}}\right\}$,
if $u,n \in \nats$ satisfy
\begin{equation}
\label{eqn:vc-subgraph-abstract-u}
u \geq c \left( \frac{\tsybcb}{\BJM(\eps)^{2-\tsybb}} + \frac{\maxloss}{\BJM(\eps)}\right) A_{1},
\end{equation}
\begin{equation}
\label{eqn:vc-subgraph-abstract-n}
n \geq
c \dc \tsybca \eps^{\tsyba} \left( \frac{\tsybcb (A_{1} + \Log(B_{1})) B_{1}}{\BJM(\eps)^{2-\tsybb}} + \frac{\maxloss (A_{1} + \Log(C_{1}))C_{1}}{\BJM(\eps)}\right),
\end{equation}
then, with arguments $\loss$, $u$, and $n$, and an appropriate $\hat{\sfun}$ function,
\ACAL~uses at most $u$ unlabeled samples and makes at most $n$ label requests,
and with probability at least $1-\conf$, returns a function $\hat{h}$ with $\er(\hat{h}) - \er(\target) \leq \eps$.
\end{theorem}

To be clear, in specifying $B_{1}$ and $C_{1}$, we adopt the convention that $1/0 = \infty$ 
so that $B_{1}$ and $C_{1}$ are well-defined even when $\tsyba = 1$ or $\tsybb=1$.
When $\tsyba < 1$, 
the dependence on $\eps$ in \eqref{eqn:vc-subgraph-abstract-n} is
$O\left( \dc \eps^{\tsyba} \BJM(\eps)^{\tsybb-2} \Log(\mixcap)\right)$,
while in the case $\tsyba = \tsybb = 1$, it is
$O\left( \dc \Log(1/\eps)  ( \Log(\dc) + \Log(\Log(1/\eps))) \right)$.
Comparing Theorem~\ref{thm:vc-subgraph-abstract} to Theorem~\ref{thm:vc-erm},
the conditions on $u$ in \eqref{eqn:vc-subgraph-abstract-u} and $m$ in \eqref{eqn:vc-erm-m}
are almost identical, aside from a logarithmic factor, so that the total number of data points indicated is roughly the same.
However, the number of \emph{labels} indicated by \eqref{eqn:vc-subgraph-abstract-n} may often be significantly
smaller than the condition in \eqref{eqn:vc-erm-m}, multiplying it by roughly $\dc \tsybca \eps^{\tsyba}$.
This reduction is particularly strong when $\dc$ is bounded by a finite constant and $\tsyba$ is large.
Moreover, this is the same \emph{type} of improvement known to occur when $\loss$ is itself the $0$-$1$ loss \citep*{hanneke:11a};
in particular, in this special case, \eqref{eqn:vc-subgraph-abstract-n} is sometimes nearly minimax \citep*{hanneke:11a,raginsky:11}.
Regarding the slight difference between \eqref{eqn:vc-subgraph-abstract-u} and \eqref{eqn:vc-erm-m} from replacing
$\capacity_{\loss}$ by $\mixcap \maxloss$, the effect is somewhat mixed, and which of these is smaller may depend on $\F$ and $\loss$.
For $\loss$ the $0$-$1$ loss, $\capacity_{\loss} = \mixcap\maxloss = \dc(\tsybca (\eps/2)^{\tsyba})$.

In the 
case when $\loss$ satisfies Condition~\ref{con:strong-convexity},
we can derive the following sometimes-stronger result with the help of Lemma~\ref{lem:bjm-strong-convexity}.

\begin{theorem}
\label{thm:vc-subgraph-active-strong}
For a universal constant $c \in [1,\infty)$,
if $\PXY$ satisfies Condition~\ref{con:tsybakov-01},
$\loss$ is classification-calibrated and satisfies Condition~\ref{con:strong-convexity},
$\target \in \F$, $\BJM$ is as in \eqref{eqn:bjm-tsybakov-transform}, 
and $\tsybcb$ and $\tsybb$ are as in Lemma~\ref{lem:bjm-strong-convexity},
then for any $\eps \in (0,1)$, letting $\dc = \dc(\tsybca \eps^{\tsyba})$
and $A_{2} = \vc(\G_{\F})\Log\left(\left(\maxloss^{2} / \tsybcb\right)\left(\tsybca \dc \eps^{\tsyba} / \BJM(\eps)\right)^{\tsybb}\right)+\Log\left(1 / \conf\right)$,
and letting $C_{1}$ be as in Theorem~\ref{thm:vc-subgraph-abstract},
if $u,n \in \nats$ satisfy
\begin{equation}
\label{eqn:vc-subgraph-active-strong-u}
u \geq c \left( \frac{\tsybcb \left(\tsybca \dc \eps^{\tsyba}\right)^{1-\tsybb}}{\BJM(\eps)^{2-\tsybb}} + \frac{\maxloss}{\BJM(\eps)}\right) A_{2},
\end{equation}
\begin{equation}
\label{eqn:vc-subgraph-active-strong-n}
n \geq c \left( \tsybcb \left(\frac{\tsybca \dc \eps^{\tsyba}}{\BJM(\eps)}\right)^{2-\tsybb} + \maxloss \left(\frac{\tsybca \dc \eps^{\tsyba}}{\BJM(\eps)}\right)\right) (A_{2}+\Log(C_{1}))C_{1},
\end{equation}
then, with arguments $\loss$, $u$, and $n$, and an appropriate $\hat{\sfun}$ function, 
\ACAL~uses at most $u$ unlabeled samples and makes at most $n$ label requests,
and with probability at least $1-\conf$, returns a function $\hat{h}$ with $\er(\hat{h}) - \er(\target) \leq \eps$.
\end{theorem}

The constraint on $u$ in \eqref{eqn:vc-subgraph-active-strong-u} has
$O\left( \frac{\left(\dc \eps^{\tsyba}\right)^{1-\tsybb}}{\BJM(\eps)^{2-\tsybb}} \Log\left( \left(\frac{\dc \eps^{\tsyba}}{\BJM(\eps)}\right)^{\tsybb}\right) \right)$
dependence on $\eps$,
while 
the constraint on $n$ in \eqref{eqn:vc-subgraph-active-strong-n} has
$O\left( \left(\frac{\dc \eps^{\tsyba}}{\BJM(\eps)}\right)^{2-\tsybb} \Log\left( \left(\frac{\dc \eps^{\tsyba}}{\BJM(\eps)}\right)^{\tsybb}\right)\right)$
in the case $\tsyba < 1$, or
$O\left( \dc^{2-\tsybb} \Log(1/\eps) \Log\left( \dc^{\tsybb} \Log(1/\eps)\right) \right)$
in the case $\tsyba = 1$.
This is noteworthy when $\dc$ is small while $\tsyba > 0$ and $\convr > 2$, for at least two reasons.
First, the sufficient size of $n$ 
in \eqref{eqn:vc-subgraph-active-strong-n}
is smaller than that in Theorem~\ref{thm:vc-subgraph-abstract},
multiplying by roughly $\left(\tsybca \dc \eps^{\tsyba}\right)^{1-\tsybb}$. 
Second, even the sufficient number of \emph{unlabeled} samples
in \eqref{eqn:vc-subgraph-active-strong-u} may be smaller than the 
sufficient number of \emph{labeled} samples for $\ERM_{\loss}$
from Theorem~\ref{thm:vc-erm}, 
again multiplying by roughly 
$\left(\tsybca \dc \eps^{\tsyba}\right)^{1-\tsybb}$.
Thus, in the case $\loss$ satisfies Condition~\ref{con:strong-convexity} with $\convr > 2$,
when Theorem~\ref{thm:vc-erm} is tight, even with access to a \emph{fully labeled} data set, 
we may \emph{still} prefer to use \ACAL~rather than $\ERM_{\loss}$.
This is somewhat surprising, since (as \eqref{eqn:vc-subgraph-active-strong-n} indicates) we expect \ACAL~to ignore the
vast majority of the labels in this case.
That said, it is not clear whether there exist natural 
losses $\loss$ of this type
for which Theorem~\ref{thm:vc-erm} 
is competitive with results for methods 
directly based on the $0$-$1$ loss.
Thus, these improvements in $u$ and $n$ 
in Theorem~\ref{thm:vc-subgraph-active-strong} may simply
indicate that \ACAL~is, to some extent, \emph{compensating} for a choice of
$\loss$ that would otherwise lead to suboptimal error rates.

\subsection{Entropy Conditions}
\label{subsec:entropy}

In this section, we consider characterizations of the complexity of $\F$ in terms of \emph{entropy conditions}.
As with the above results, detailed derivations of all of these results are presented in Section~\ref{subsec:entropy-derivation} below, 
based on the abstract theorems presented in Section~\ref{subsec:abstract}.

For a distribution $P$ over $\X\times\Y$,
a set $\G \subseteq \FunctionsXY$, and $\eps \geq 0$,
let $\covering(\eps,\G,L_{2}(P))$ denote the size of a minimal
$\eps$-cover of $\G$ (that is, the minimum number of balls of radius at most $\eps$
sufficient to cover $\G$), where distances are measured
in terms of the $L_{2}(P)$ pseudo-metric: $(f,g) \mapsto \|f-g\|_{P}$.
Also, for functions $g_1 \leq g_2$, a \emph{bracket} $[g_1,g_2]$ is the set of functions $g \in \FunctionsXY$ with $g_1 \leq g \leq g_2$;
$[g_1,g_2]$ is called an $\eps$-bracket under $L_{2}(P)$ if $\|g_1 - g_2\|_{P} < \eps$.
Then $\covering_{[]}(\eps,\G,L_{2}(P))$ denotes the smallest number of $\eps$-brackets (under $L_{2}(P)$) sufficient to cover $\G$. 

The following represent two commonly-studied conditions. 
\begin{condition}
\label{con:entropy}
For some $\entc \geq 1$, $\entrho \in (0,1)$, $\Env \geq \Env(\G_{\F,\PXY})$,
either $\forall \eps > 0$,
\begin{equation}
\label{eqn:bracketing-entropy-bound}
\ln \covering_{[]}(\eps \|\Env\|_{\PXY}, \G_{\F}, L_{2}(\PXY)) \leq \entc \eps^{-2\entrho},
\end{equation}
or for all finitely discrete $P$, $\forall \eps > 0$,
\begin{equation}
\label{eqn:uniform-entropy-bound}
\ln \covering(\eps \|\Env\|_{P}, \G_{\F}, L_{2}(P)) \leq \entc \eps^{-2\entrho}.
\end{equation}
\end{condition}

The following theorem is a classic result on the performance of $\ERM_{\loss}$ under the above conditions \citep*[e.g.,][]{bartlett:06,van-der-Vaart:96}.

\begin{theorem}
\label{thm:ent-erm}
For a universal constant $c \in [1,\infty)$,
if $\PXY$ satisfies Condition~\ref{con:tsybakov-01} and Condition~\ref{con:tsybakov-sur},
$\F$ and $\PXY$ satisfy Condition~\ref{con:entropy},
$\loss$ is classification-calibrated, $\target \in \F$, and $\BJM$ is as in \eqref{eqn:bjm-tsybakov-transform},
then for any $\eps \in (0,1)$ and $m$ with 
\begin{multline*}
m \geq c \frac{\entc \|\Env\|_{\PXY}^{2\entrho}}{(1-\entrho)^{2}} \left( \frac{\tsybcb^{1-\entrho}}{\BJM(\eps)^{2 - \tsybb (1-\entrho)}} + \frac{\maxloss^{1-\entrho}}{\BJM(\eps)^{1+\entrho}}\right)
\\ + c\left(\frac{\tsybcb}{\BJM(\eps)^{2-\tsybb}} + \frac{\maxloss}{\BJM(\eps)}\right) \Log \left(\frac{1}{\conf}\right),
\end{multline*}
with probability at least $1-\conf$,
$\ERM_{\loss}(\F,\Data_{m})$ produces $\hat{h}$ with $\er(\hat{h}) - \er(\target) \leq \eps$.
\end{theorem}

Turning to the analogous setting for active learning, we are able to establish the following theorem on the performance of \ACAL~under these same conditions.

\begin{theorem}
\label{thm:ent-active}
For a universal constant $c \in [1,\infty)$,
if $\PXY$ satisfies Condition~\ref{con:tsybakov-01} and Condition~\ref{con:tsybakov-sur},
$\F$ and $\PXY$ satisfy Condition~\ref{con:entropy},
$\loss$ is classification-calibrated,
$\target \in \F$, and $\BJM$ is as in \eqref{eqn:bjm-tsybakov-transform},
then for any $\eps \in (0,1)$, letting $B_{1}$ and $C_{1}$ be as in Theorem~\ref{thm:vc-subgraph-abstract},
$B_{2} = \min\left\{ B_{1}, \frac{1}{1-2^{-\entrho}} \right\}$,
$C_{2} = \min\left\{ C_{1}, \frac{1}{1-2^{-\entrho}} \right\}$,
and abbreviating $\dc = \dc\left(\tsybca \eps^{\tsyba}\right)$, if $u,n \in \nats$ satisfy
\begin{multline}
\label{eqn:entropy-abstract-u}
u \geq c \frac{\entc \|\Env\|_{\PXY}^{2\entrho}}{(1-\entrho)^{2}} \left( \frac{\tsybcb^{1-\entrho}}{\BJM(\eps)^{2 - \tsybb (1-\entrho)}} + \frac{\maxloss^{1-\entrho}}{\BJM(\eps)^{1+\entrho}}\right)
\\ + c\left(\frac{\tsybcb}{\BJM(\eps)^{2-\tsybb}} + \frac{\maxloss}{\BJM(\eps)}\right) \Log \left(\frac{1}{\conf}\right),
\end{multline}
\begin{multline}
\label{eqn:entropy-abstract-n}
n \geq c \dc \tsybca \eps^{\tsyba} \frac{\entc \|\Env\|_{\PXY}^{2\entrho}}{(1-\entrho)^{2}} \left( \frac{\tsybcb^{1-\entrho} B_{2}}{\BJM(\eps)^{2 - \tsybb (1-\entrho)}} + \frac{\maxloss^{1-\entrho} C_{2}}{\BJM(\eps)^{1+\entrho}}\right)
\\ + c \dc \tsybca \eps^{\tsyba} \left(\frac{\tsybcb B_{1} \Log(B_{1}/\conf)}{\BJM(\eps)^{2-\tsybb}} + \frac{\maxloss C_{1} \Log(C_{1}/\conf)}{\BJM(\eps)}\right),
\end{multline}
then, with arguments $\loss$, $u$, and $n$, and an appropriate $\hat{\sfun}$ function, 
\ACAL~uses at most $u$ unlabeled samples and makes at most $n$ label requests,
and with probability at least $1-\conf$, returns a function $\hat{h}$ with $\er(\hat{h}) - \er(\target) \leq \eps$.
\end{theorem}

The constraint on $u$ in \eqref{eqn:entropy-abstract-u} 
is identical (up to constant factors)
to the 
sample size
in Theorem~\ref{thm:ent-erm} sufficient
for $\ERM_{\loss}$ to achieve the same. 
In contrast, when $\dc$ is small, 
the constraint on $n$ in \eqref{eqn:entropy-abstract-n} 
improves this,
multiplying by a factor $\propto \dc \tsybca \eps^{\tsyba}$.

As before, 
when $\loss$ satisfies Condition~\ref{con:strong-convexity},
we can derive sometimes-stronger results via Lemma~\ref{lem:bjm-strong-convexity}. 
In this case, we will distinguish between the cases
of \eqref{eqn:uniform-entropy-bound} and \eqref{eqn:bracketing-entropy-bound}, as we find
a slightly stronger result for the former.
We begin with the following result, under the uniform entropy condition \eqref{eqn:uniform-entropy-bound}.

\begin{theorem}
\label{thm:uniform-entropy-strong-convexity}
For a universal constant $c \in [1,\infty)$,
if $\PXY$ satisfies Condition~\ref{con:tsybakov-01},
$\loss$ is classification-calibrated and satisfies Condition~\ref{con:strong-convexity}, $\target \in \F$,
$\BJM$ is as in \eqref{eqn:bjm-tsybakov-transform},
$\tsybcb$ and $\tsybb$ are as in Lemma~\ref{lem:bjm-strong-convexity},
and \eqref{eqn:uniform-entropy-bound} is satisfied with $\Env \leq \maxloss$ 
($\forall$ finitely discrete $P$, $\forall \eps > 0$),
then $\forall \eps \in (0,1)$, for $C_{1}$ as in Theorem~\ref{thm:vc-subgraph-abstract}
and $\dc = \dc\left(\tsybca \eps^{\tsyba}\right)$,
if 
\begin{multline*}
u  \geq c \left(\frac{\entc \maxloss^{2\entrho}}{(1-\entrho)^{2}}\right)\left(
\left(\frac{\tsybcb^{1-\entrho}}{\BJM(\eps)}\right)\left(\frac{ \tsybca \dc \eps^{\tsyba}}{\BJM(\eps)}\right)^{1-\tsybb(1-\entrho)}
+ \left(\frac{\maxloss^{1-\entrho}}{\BJM(\eps)}\right)\left(\frac{\tsybca \dc \eps^{\tsyba}}{\BJM(\eps)}\right)^{\entrho}\right)
\\ + c \left( \left(\frac{\tsybcb}{\BJM(\eps)}\right) \left(\frac{\tsybca \dc \eps^{\tsyba}}{\BJM(\eps)}\right)^{1-\tsybb}
+ \frac{\maxloss}{\BJM(\eps)}\right)\Log\left(\frac{1}{\conf}\right),
\end{multline*}
\begin{multline*}
n \geq
c \left(\frac{\entc \maxloss^{2\entrho} C_{1}}{(1-\entrho)^{2}}\right)
\left(\tsybcb^{1-\entrho} \left(\frac{\tsybca \dc \eps^{\tsyba}}{\BJM(\eps)}\right)^{2-\tsybb(1-\entrho)}
+ \maxloss^{1-\entrho}\left(\frac{\tsybca \dc \eps^{\tsyba}}{\BJM(\eps)}\right)^{1+\entrho} \right)
\\ + c \left( \tsybcb \left(\frac{\tsybca \dc \eps^{\tsyba}}{\BJM(\eps)}\right)^{2-\tsybb}
+ \maxloss \left(\frac{\tsybca \dc \eps^{\tsyba}}{\BJM(\eps)}\right)\right) C_{1} \Log\left(\frac{C_{1}}{\conf}\right),
\end{multline*}
then, with arguments $\loss$, $u$, and $n$,
and an appropriate $\hat{\sfun}$ function, 
\ACAL~uses at most $u$ unlabeled samples and makes at most $n$ label requests,
and with probability at least $1-\conf$, returns a function $\hat{h}$ with $\er(\hat{h}) - \er(\target) \leq \eps$.
\end{theorem}

Compared to Theorem~\ref{thm:ent-active}, 
the constraints for $u$ and $n$ here may have improved dependences on $\eps$, multiplying by
$O\left(\left(\dc \eps^{\tsyba}\right)^{1-\tsybb(1-\entrho)}\right)$.
Furthermore, 
for small $\dc$,
these 
are also smaller than 
the 
size of $m$
for $\ERM_{\loss}(\F,\Data_{m})$
from Theorem~\ref{thm:ent-erm}.

Next, we turn to the bracketing entropy condition \eqref{eqn:bracketing-entropy-bound}.
For simplicity, we will only consider the case that \eqref{eqn:bracketing-entropy-bound} 
is satisfied with $\Env = \maxloss$ constant.
In this case, we have the following result.

\begin{theorem}
\label{thm:bracketing-entropy-strong-convexity}
For a universal constant $c \in [1,\infty)$,
if $\PXY$ satisfies Condition~\ref{con:tsybakov-01},
$\loss$ is classification-calibrated and satisfies Condition~\ref{con:strong-convexity},
$\target \in \F$,
$\BJM$ is as in \eqref{eqn:bjm-tsybakov-transform},
$\tsybcb$ and $\tsybb$ are as in Lemma~\ref{lem:bjm-strong-convexity},
and \eqref{eqn:bracketing-entropy-bound} is satisfied with
$\Env = \maxloss$, 
then
$\forall \eps \in (0,1)$, letting
$C_{1}$ be as in Theorem~\ref{thm:vc-subgraph-abstract},
$C_{2}$ be as in Theorem~\ref{thm:ent-active},
and $\dc = \dc\left(\tsybca \eps^{\tsyba}\right)$,
if 
\begin{multline*}
u \geq c \left(\frac{\entc \maxloss^{2\entrho}}{(1-\entrho)^{2}}\right)
\left(\left(\frac{\tsybcb^{1-\entrho}}{\BJM(\eps)^{1+\entrho}}\right)\left(\frac{\tsybca \dc \eps^{\tsyba}}{\BJM(\eps)}\right)^{(1-\tsybb)(1-\entrho)}
+
\frac{\maxloss^{1-\entrho}}{\BJM(\eps)^{1+\entrho}}\right)
\\ +
c \left( \left(\frac{\tsybcb}{\BJM(\eps)}\right)
\left(\frac{\tsybca \dc \eps^{\tsyba}}{\BJM(\eps)}\right)^{1-\tsybb}
+ \frac{\maxloss}{\BJM(\eps)}\right)\Log\left(\frac{1}{\conf}\right),
\end{multline*}
\begin{multline*}
n \geq c \left(\frac{\entc \maxloss^{2\entrho} C_{2}}{(1-\entrho)^{2}}\right)
\left(\left(\frac{\tsybcb^{1-\entrho}}{\BJM(\eps)^{\entrho}}\right)
\left(\frac{\tsybca \dc \eps^{\tsyba}}{\BJM(\eps)}\right)^{1+(1-\tsybb)(1-\entrho)}
+
\frac{\maxloss^{1-\entrho}\tsybca \dc \eps^{\tsyba}}{\BJM(\eps)^{1+\entrho}}\right)
\\ +
c \left( \tsybcb \left(\frac{\tsybca \dc \eps^{\tsyba}}{\BJM(\eps)}\right)^{2-\tsybb}
+ \maxloss \left(\frac{\tsybca \dc \eps^{\tsyba}}{\BJM(\eps)}\right)\right) C_{1} \Log\left(\frac{C_{1}}{\conf}\right),
\end{multline*}
then, with arguments $\loss$, $u$, and $n$,
and an appropriate $\hat{\sfun}$ function, 
\ACAL~uses at most $u$ unlabeled samples
and makes at most $n$ label requests, and with probability at least $1-\conf$,
returns a function $\hat{h}$ with $\er(\hat{h}) - \er(\target) \leq \eps$.
\end{theorem}
Compared to Theorem~\ref{thm:ent-active}, 
the dependence on $\eps$ in
the 
sizes for both $u$ and $n$ may be smaller here, multiplying by
$O\!\left(\left(\dc\eps^{\tsyba}\right)^{(1-\tsybb)(1-\entrho)}\right)$,
which is sometimes significant, though not quite as 
dramatic a reduction
as we found under \eqref{eqn:uniform-entropy-bound} in Theorem~\ref{thm:uniform-entropy-strong-convexity}.
As with Theorem~\ref{thm:uniform-entropy-strong-convexity}, when $\dc(\eps^{\tsyba}) = o(\eps^{-\tsyba})$,
the sizes of $u$ and $n$ indicated by 
Theorem~\ref{thm:bracketing-entropy-strong-convexity} are smaller
than the results for $\ERM_{\loss}(\F,\Data_{m})$ from Theorem~\ref{thm:ent-erm}.

\subsection{An Example: Discrete Distributions}
\label{subsec:example}


%
As a concrete example applying the above results,
we find that \ACAL~generally provides some 
benefits for discrete $\Px$ distributions.
%
To describe these benefits quantitatively, consider the special case
where $\exists x_{1},x_{2},\ldots \!\in \X$ with $\Px(\{x_{i}\}) = \frac{90}{\pi^{4} i^{4}}$,
and $\eta(x) \in [0,\bound] \cup [1-\bound,1]$ for each $x \in \X$, where $\bound \in [0,1/2)$ is a constant.
Set $\F = \{ f \in \Functions : \sup_{x \in \X} |f(x)| \leq 1 \}$, 
and take $\loss$ to be the quadratic loss (in which case $\maxloss = 4$).
In particular, since $\target(x) = 2\eta(x)-1 \in [-1,1]$, the condition $\target \in \F$ is satisfied in this scenario.
We will use Theorem~\ref{thm:ent-active} to bound the number of labels sufficient for \ACAL~to achieve excess error rate $\eps$.
For any $g \in \Functions$, we have $\er(g)-\er(\target) = \sum_{i \in \nats} \ind_{\DIS(\{g,\target\})}(x_{i}) |1-2\eta(x_{i})| \Px(\{x_{i}\}) \geq (1-2\bound) \dist(g,\target)$,
so that Condition~\ref{con:tsybakov-01} is satisfied with $\tsyba=1$ and $\tsybca = 1/(1-2\bound)$.
Furthermore, $\F$ is convex, and this $\loss$ satisfies Condition~\ref{con:strong-convexity},
with $\tsybb = 1$ and $\tsybcb = 32$ in Lemma~\ref{lem:bjm-strong-convexity}.
Also, since $\transform(x) = x^{2}$ 
here \citep*{bartlett:06},
we have that $\BJM(\eps) = \eps^{2-\tsyba} / (4 \tsybca) = (1-2\bound) \eps / 4$.
%
Additionally, this scenario satisfies 
\eqref{eqn:bracketing-entropy-bound} in Condition~\ref{con:entropy}
with $\entc = \frac{7}{\omega}$ and $\entrho = \frac{1}{3}+\omega$, for any choice of $\omega \in (0,1/2]$;
we include a simple proof of this fact in Appendix~\ref{subsec:example-derivations}. 
Finally, we bound 
$\dc(r_{0})$ for $r_{0} \in (0,1]$.
For any $r \in (0,1)$, we have 
$\DIS(\Ball(\target,r)) \cap \{x_{i} : i \in \nats\} $ $= \left\{x_{i} : \frac{90}{\pi^{4} i^{4}} \leq r\right\}$,
so that $\Px(\DIS(\Ball(\target,r))) \lesssim \sum_{i \gtrsim r^{-1/4}} i^{-4} \lesssim r^{3/4}$.
Therefore, $\dc(r_{0}) \lesssim r_{0}^{-1/4}$. 

Plugging these values into Theorem~\ref{thm:ent-active}, 
and choosing $\omega = \left(\ln\left(\frac{1}{(1-2\bound)\eps}\right)\right)^{-1}$, 
we find that there is a label budget $n$,
sufficient to guarantee $\er(\hat{h}) - \er(\target) \leq \eps$ with probability at least $1-\conf$ in \ACAL, 
with dependence $\Theta\left( \eps^{-7/12} \Log(1/\eps) \right)$ on $\eps$. 
For comparison, the corresponding bound for $\ERM_{\loss}$ from Theorem~\ref{thm:ent-erm} has dependence $\Theta\left( \eps^{-4/3} \Log(1/\eps) \right)$. 
This is larger than the above bound by a factor $\Theta\left( \eps^{-3/4} \right)$.
%
%
%
%
%
%
Furthermore, one can show an $\Omega(\eps^{-4/3})$ \emph{lower bound} on the sample size necessary 
to obtain $\eps$ \emph{minimax} expected excess error rate for passive learning in this scenario.
%
Thus, \ACAL~achieves a significant improvement
over the guarantees achievable by \emph{all} passive learning methods.  
The details of this minimax lower bound
are included in Appendix~\ref{subsec:example-derivations}.



\ignore{That said, it is worth noting that \ACAL~itself is not an optimal active learning algorithm
for this scenario.  In particular, by a more direct approach, one can easily show that the 
minimax optimal number of labels sufficient for active learning to achieve expected excess error $\eps$
in this scenario has dependence on $\eps$ of roughly $\eps^{-1/3}$ (up to logarithmic factors).}

\subsection{An Example: Linear Functions}
\label{subsec:linsep}

As another example applying the above results, consider the class of \emph{homogeneous linear functions}.
Specifically, fix any $k \in \nats$ with $k \geq 5$, $\X = \{ x \in \reals^{k} : \|x\| \leq 1 \}$,
and consider the class $\F = \{x \mapsto w \cdot x : w \in \reals^{k}, \|w\| \leq 1\}$. 
Take $\loss$ as the 
quadratic loss (in which case $\maxloss = 4$).  Together with the assumption of $\target \in \F$, 
this restricts $\PXY$ to have $\eta(x) = (w \cdot x + 1)/2$ (almost everywhere),
for some $w \in \reals^{k}$ with $\|w\| \leq 1$.
Furthermore, this $\loss$ satisfies Condition~\ref{con:strong-convexity},
with $\tsybb = 1$ and $\tsybcb = 32$ in Lemma~\ref{lem:bjm-strong-convexity},
and has $\BJM(\eps) = \eps^{2-\tsyba} / (4 \tsybca)$.
It is also known that $\vc(\G_{\F}) \lesssim k$ (following from arguments of \citep*{dudley:78,haussler:92}).
Additionally, for this class $\F$,
it is known that if $\Px$ has a density (with respect to Lebesgue measure),
then $\dc(\eps) = o(1/\eps)$ \citep*{hanneke:survey}.
Together, these facts imply that, if $\Px$ has a density,
the sufficient size of $n$ in Theorem~\ref{thm:vc-subgraph-active-strong}
has dependence on $\eps$ that is $o\left( \eps^{\tsyba-2} \Log(1/\eps) \right)$.
%
We also note that, 
by varying $\Px$, it is possible to realize any $\tsyba$
value in $(0,1]$ in Condition~\ref{con:tsybakov-01} \citep*[see][]{cavallanti:11,dekel:12}.

To exhibit a concrete example, consider the simple scenario 
of $\Px$ uniform on\break $\{x \in \reals^{k} : \|x\| = 1\}$,
and suppose $\PXY$ is such that $\target \in \F$.
For simplicity, also suppose the $w \in \reals^{k}$ with $\target(x) = w \cdot x$
satisfies $\|w\| = 1$. 
In this case, one can show that Condition~\ref{con:tsybakov-01} is 
satisfied with $\tsybca \propto k^{1/4}$ and $\tsyba = 1/2$. 
For completeness, 
a proof of this is included in Appendix~\ref{subsec:linsep-derivations}.
It is also known that $\dc(\eps) \leq \pi \sqrt{k}$ for this scenario \citep*{hanneke:07b}.
Plugging all of this into Theorem~\ref{thm:vc-subgraph-active-strong} 
reveals that, for \ACAL~to achieve excess error rate $\eps$ with probability at least 
$1-\conf$ (given sufficiently large $u$), it suffices to have a label budget $n$ of size at least 
\begin{equation*} 
c \frac{k}{\eps} \left( k \Log\left( \frac{k}{\eps} \right) + \Log\left(\frac{1}{\conf}\right) \right),
\end{equation*}
for a universal constant $c > 0$.
In contrast, Theorem~\ref{thm:vc-erm} gives a sufficient sample size for $\ERM_{\loss}(\F,\cdot)$ 
proportional to $\frac{k^{1/4}}{\eps^{3/2}} \left( k \Log(k) + \Log(1/\conf) \right)$, 
which is significantly larger than the above size of $n$ for $\eps$ sufficiently small. 
To our knowledge, it is not presently known what the optimal sample complexity of passive learning is for this scenario,
so that in contrast to the previous example, here we can only claim an improvement in the upper bound.
%
We note that \citet*{dekel:12} have also studied active learning with 
this $\F$ and $\loss$ under the same assumption of $\target \in \F$,
and established a similar result to the above (with slightly better dependence on $k$
but slightly worse logarithmic factors),
via a learning method tailored specifically to this function class.

\section{General Theorems}
\label{sec:abstract}

The remainder of the article is devoted to a general analysis of \ACAL, 
from which we derive the more-explicit theorems stated above.
The results are formulated analogously to localization arguments
common in the literature on empirical risk minimization, but with a 
slight twist to introduce a relevant subregion to the argument.  As such, 
we begin with a discussion of general localized sample complexity bounds.

\subsection{Localized Sample Complexities}
\label{subsec:rademacher}

The derivation of localized excess risk bounds is essentially motivated as follows.
We are interested in bounding the excess $\loss$-risk of
the $\hat{h}$ returned by $\ERM_{\loss}(\H,\Data_{m})$.
%
Suppose we have a coarse guarantee $U_{\loss}(\H,m)$ on this value:
that is, $\Risk(\hat{h})-\inf_{h \in \H} \Risk(h)$ $\leq U_{\loss}(\H,m)$.
In a sense, this guarantee identifies a set $\H^{\prime}$ $\subseteq$ $\H$ of functions that
a priori may have the \emph{potential} to be returned by $\ERM_{\loss}(\H,\Data_{m})$ (namely,
$\H^{\prime} = \H(U_{\loss}(\H,m);\loss)$), while those in $\H \setminus \H^{\prime}$ do not.
With this information in hand, we can think of $\H^{\prime}$
as a kind of \emph{effective} function class, and we can 
think of $\ERM_{\loss}(\H,\Data_{m})$
as equivalent to $\ERM_{\loss}(\H^{\prime},\Data_{m})$.
We may then repeat this same reasoning, now thinking of $\hat{h}$ as the function 
returned by $\ERM_{\loss}(\H^{\prime},\Data_{m})$: 
that is, we calculate $U_{\loss}(\H^{\prime},m)$ to determine
a further subset $\H^{\prime\prime} = \H^{\prime}(U_{\loss}(\H^{\prime},m);\loss) \subseteq \H^{\prime}$
of functions that we again expect to contain the empirical minimizer $\hat{h}$, so that
$\ERM_{\loss}(\H^{\prime},\Data_{m}) = \ERM_{\loss}(\H^{\prime\prime},\Data_{m})$, and so on.
This repeats until we
identify a fixed-point set $\H^{(\infty)}$ of functions such that
$\H^{(\infty)}(U_{\loss}(\H^{(\infty)},m);\loss)$ $= \H^{(\infty)}$,
so that no further reduction is possible.
Following this chain of reasoning back to the beginning, we find that
$\ERM_{\loss}(\H,\Data_{m}) = \ERM_{\loss}(\H^{(\infty)},\Data_{m})$, so that the function $\hat{h}$ returned by
$\ERM_{\loss}(\H,\Data_{m})$ has excess $\loss$-risk at most $U_{\loss}(\H^{(\infty)},m)$, which may be
significantly smaller than $U_{\loss}(\H,m)$, depending on 
how $U_{\loss}(\H,m)$ varies with $\H$.

To formalize this fixed-point argument for $\ERM_{\loss}(\H,\Data_{m})$,
\citet*{koltchinskii:06} makes use of the following quantities
to define the coarse bound $U_{\loss}(\H,m)$ \citep*[see also][]{bartlett:05,gine:06}.
For any $\H \subseteq \Bracket{\F}$, $m \in \nats$, $s \in [1,\infty)$,
and any distribution $P$ on $\X \times \Y$, letting $S \sim P^{m}$, define
\begin{align*}
\phi_{\loss}(\H; m,P) &= \E\left[\sup_{h,g \in \H} \left(\Risk(h;P) - \Risk(g;P)\right) - \left(\Risk(h;S) - \Risk(g;S)\right) \right],
\\ \bar{U}_{\loss}(\H; P, m, s) & = \bar{K}_{1} \phi_{\loss}(\H; m,P) + \bar{K}_{2} \D_{\loss}(\H;P) \sqrt{\frac{s}{m}} + \frac{\bar{K}_{3} \maxloss s}{m},\\
\tilde{U}_{\loss}(\H; P, m,s) & = \tilde{K} \left( \phi_{\loss}(\H; m,P) + \D_{\loss}(\H;P) \sqrt{\frac{s}{m}} + \frac{\maxloss s}{m}\right),
\end{align*}
where
$\bar{K}_{1}$, $\bar{K}_{2}$, $\bar{K}_{3}$, and $\tilde{K}$ are appropriately chosen constants.

We will be interested in having access to these quantities in the context of our algorithms; however,
since $\PXY$ is not directly accessible to the algorithm, we will need to approximate these by data-dependent estimators.
Toward this end, we define the following quantities, again taken from the work of \citet*{koltchinskii:06}.
For any $\H \subseteq \Bracket{\F}$, $q \in \nats$,
and $S = \{(x_1,y_1),\ldots,(x_q,y_q)\} \in (\X \times \{-1,+1\})^{q}$,
let $\H(\eps;\loss,S) = \{h \in \H : \Risk(h;S) - \inf_{g \in \H} \Risk(g;S) \leq \eps\}$;
then for any sequence $\Xi = \{\xi_k\}_{k=1}^{q} \in \{-1,+1\}^{q}$, and any $s \in [1,\infty)$,
define
\begin{align*}
\Rademacher_{\loss}(\H;S,\Xi) &= \sup_{h,g \in \H} \frac{1}{q} \sum_{k=1}^{q} \xi_{k} \cdot \left( \loss(h(x_k)y_k) -\loss(g(x_k)y_k) \right),
\\ \hat{\D}_{\loss}(\H;S)^{2} &= \sup_{h,g \in \H} \frac{1}{q} \sum_{k=1}^{q} \left( \loss(h(x_k)y_k) - \loss(g(x_k)y_k) \right)^{2},
\\ \hat{U}_{\loss}(\H;S,\Xi,s) &= 12 \Rademacher_{\loss}(\H;S,\Xi) + 34\hat{\D}_{\loss}(\H;S) \sqrt{\frac{s}{q}} + \frac{752 \maxloss s}{q}.
\end{align*}
For completeness,
let
$\Rademacher_{\loss}(\H;\emptyset,\emptyset) = \hat{\D}_{\loss}(\H;\emptyset) = 0$, and
$\hat{U}_{\loss}(\H; \emptyset, \emptyset, s) = 752 \maxloss s$.

The above $U$ quantities (with appropriate choices of
$\bar{K}_{1}$, $\bar{K}_{2}$, $\bar{K}_{3}$, and $\tilde{K}$)
can be formally related to each other and to the excess $\loss$-risk of functions in $\H$ via the
following general result; this variant is due to \citet*{koltchinskii:06}.

\begin{lemma}
\label{lem:koltchinskii}
For any $\H \subseteq \Bracket{\F}$, $s \in [1,\infty)$,
distribution $P$ over $\X \times \Y$,
and any $m \in \nats$,
if $S \sim P^{m}$ and $\Xi = \{\xi_1,\ldots,\xi_m\} \sim {\rm Uniform}(\{-1,+1\})^{m}$ are independent,
and $h^{*} \in \H$ has $\Risk(h^{*};P) = \inf_{h \in \H} \Risk(h;P)$,
then with probability at least $1 - 6e^{-s}$, the following claims hold.
\begin{align*}
\forall h \in \H, \Risk(h;P) - \Risk(h^{*};P) & \leq \Risk(h ; S) - \Risk(h^{*} ; S) + \bar{U}_{\loss}(\H;P,m,s), 
\\ \forall h \in \H, \Risk(h ; S) - \inf_{\g \in \H} \Risk(\g ; S) & \leq\Risk(h;P) - \Risk(h^{*};P) + \bar{U}_{\loss}(\H;P,m,s), 
\\ \bar{U}_{\loss}(\H;P,m,s) & < \hat{U}_{\loss}(\H;S,\Xi,s) < \tilde{U}_{\loss}(\H;P,m,s). 
\end{align*}
\end{lemma}

We typically expect the quantities $\bar{U}$, $\hat{U}$, and $\tilde{U}$ to be roughly within constant factors of each other.
Following \citet*{koltchinskii:06} and \citet*{gine:06},
we can use this result
to derive localized bounds on the
number of samples sufficient for $\ERM_{\loss}(\H,\Data_{m})$
to achieve a given excess $\loss$-risk.
Specifically, for $\H \subseteq \Bracket{\F}$, distribution $P$ over $\X \times \Y$,
values $\gamma,\gamma_{1},\gamma_{2} \geq 0$, 
$s \in [1,\infty)$,
and any function $\sfun : (0,\infty)^2 \to [1,\infty)$,
define the following quantities.
\begin{align*}
\bar{\SC}_{\loss}(\gamma_{1},\gamma_{2};\H,P,s) & =
\min\left\{m \in \nats : \bar{U}_{\loss}(\H(\gamma_{2};\loss,P);P,m,s) < \gamma_{1}\right\},
\\ \bar{\SC}_{\loss}(\gamma;\H,P,\sfun) & = \sup_{\gamma^{\prime} \geq \gamma} \bar{\SC}_{\loss}(\gamma^{\prime}/2,\gamma^{\prime};\H,P,\sfun(\gamma,\gamma^{\prime})),
\end{align*}
\begin{align*}
\tilde{\SC}_{\loss}(\gamma_{1},\gamma_{2}; \H,P,s) & = \min\left\{ m \in \nats : \tilde{U}_{\loss}(\H(\gamma_{2};\loss,P);P,m,s) \leq \gamma_{1}\right\},
\\ \tilde{\SC}_{\loss}(\gamma; \H,P,\sfun) & = \sup_{\gamma^{\prime} \geq \gamma} \tilde{\SC}_{\loss}(\gamma^{\prime}/2,\gamma^{\prime}; \H,P,\sfun(\gamma,\gamma^{\prime})).
\end{align*}
These quantities are well-defined for
$\gamma_1,\gamma_2,\gamma > 0$ when
$\lim_{m\to\infty} \phi_{\loss}(\H;m,P) = 0$.
In other cases, for completeness,
we define them to be $\infty$.

In particular, the quantity $\bar{\SC}_{\loss}(\gamma;\F,\PXY,\sfun)$ is used
in Theorem~\ref{thm:erm} below to quantify the performance
of $\ERM_{\loss}(\F,\Data_{m})$.
The primary practical challenge in calculating $\bar{\SC}_{\loss}(\gamma;\H,P,\sfun)$ is handling the $\phi_{\loss}(\H(\gamma^{\prime};\loss,P);m,P)$ quantity.
In the literature, the typical (only?) way such calculations are approached
is by first deriving a bound on $\phi_{\loss}(\H^{\prime};m,P)$ for every $\H^{\prime} \subseteq \H$
in terms of some natural measure of complexity for the full class $\H$ (e.g., entropy numbers)
and some very basic measure of complexity for $\H^{\prime}$:
most often $\D_{\loss}(\H^{\prime};P)$ and sometimes a seminorm of an envelope function.
After this, one then proceeds to bound these basic measures of complexity for the specific subsets
$\H(\gamma^{\prime};\loss,P)$, as a function of $\gamma^{\prime}$.
Composing these two results is then sufficient to bound $\phi_{\loss}(\H(\gamma^{\prime};\loss,P);m,P)$.
For instance, bounds based on an entropy integral tend to follow this strategy.
This approach effectively decomposes the problem of calculating the complexity of $\H(\gamma^{\prime};\loss,P)$
into the problem of calculating the complexity of $\H$ and the problem of calculating some more basic properties
of $\H(\gamma^{\prime};\loss,P)$.
See \citep*{van-der-Vaart:96,koltchinskii:06,gine:06,bartlett:06},
or Section~\ref{subsec:phi-bounds} below, for several explicit examples of this technique.

Another technique often (though not always) used in conjunction with the above strategy
when deriving explicit rates of convergence is to relax
$\D_{\loss}(\H(\gamma^{\prime};\loss,P);P)$ to $\D_{\loss}(\Functions(\gamma^{\prime};\loss,P);P)$
or $\D_{\loss}(\Bracket{\H}(\gamma^{\prime};\loss,P);P)$.
This relaxation can sometimes be a source of slack;
however,
in many interesting cases, such as
for certain losses or noise conditions,
this approach can still lead to nearly tight bounds \citep*{bartlett:06,mammen:99,tsybakov:04}.

For our purposes, it 
is convenient to make these
common techniques explicit in the results.
This will make the benefits of our proposed method more apparent,
while still allowing us to state results in a form abstract enough to encompass 
the 
more-specific complexity measures referenced in the theorems of Section~\ref{sec:explicit}.
%
Toward this end, we have the following definition 
(recall the definitions of $h_{\region,g}$ and $\H_{\region,g}$ from Section~\ref{sec:acal} above).
\begin{definition}
\label{defn:abstract-ophi}
For every distribution $P$ over $\X \times \Y$,
let $\spec{\phi}_{\loss}(\sigma,\H;m,P)$ be a quantity defined for every
$\sigma \in [0,\infty]$, $\H \subseteq \Bracket{\F}$, and $m \in \nats$,
such that the following conditions are satisfied when $\target_{P} \in \H$.
\begin{align}
&\text{If } 0 \leq \sigma \leq \sigma^{\prime}, \H \subseteq \H^{\prime} \subseteq \Bracket{\F}, \region \subseteq \X, \text{ and } m^{\prime} \leq m, \notag
\\ &\text{then } \spec{\phi}_{\loss}(\sigma,\H_{\region,\target_{P}};m,P) \leq \spec{\phi}_{\loss}(\sigma^{\prime},\H^{\prime};m^{\prime},P). \label{eqn:ophi-1}
\\ &\forall \sigma \geq \D_{\loss}(\H;P), \phi_{\loss}(\H;m,P) \leq \spec{\phi}_{\loss}(\sigma,\H;m,P). \label{eqn:ophi-2}
\end{align}
\end{definition}
For instance, most bounds based on entropy integrals can be made to satisfy this.
Section~\ref{subsec:phi-bounds} states explicit examples of quantities
$\spec{\phi}_{\loss}$ from the literature that satisfy this definition.
Given a function $\spec{\phi}_{\loss}$ of this type, we define the following quantity
for $m \in \nats$, $s \in [1,\infty)$, $\zeta \in [0,\infty]$, $\H \subseteq \Bracket{\F}$, and a distribution $P$ over $\X\times\Y$.
\begin{align*}
& \spec{U}_{\loss}(\H,\zeta; P, m,s)
\\ & = \tilde{K} \left( \spec{\phi}_{\loss}(\D_{\loss}(\Bracket{\H}(\zeta;\loss,P);P),\H; m,P) + \D_{\loss}(\Bracket{\H}(\zeta;\loss,P);P) \sqrt{\frac{s}{m}} + \frac{\maxloss s}{m}\right).
\end{align*}
Note that when $\target_{P} \in \H$,
since $\D_{\loss}(\Bracket{\H}(\gamma;\loss,P);P) \geq \D_{\loss}(\H(\gamma;\loss,P);P)$,
Definition~\ref{defn:abstract-ophi} implies
$\phi_{\loss}(\H(\gamma;\loss,P);m,\!P) \leq \spec{\phi}_{\loss}(\D_{\loss}(\Bracket{\H}(\gamma;\loss,P);P),\!\H(\gamma;\loss,P);m,\!P)$,
and furthermore $\H(\gamma;\loss,P) \subseteq \H$ so that
$\spec{\phi}_{\loss}(\D_{\loss}(\Bracket{\H}\!(\gamma;\loss,P);P),\H(\gamma;\loss,P);m,P) \leq$\break
$\spec{\phi}_{\loss}(\D_{\loss}(\Bracket{\H}(\gamma;\loss,P);P),\H;m,P)$.
Thus,
\begin{equation}
\label{eqn:oU-vs-tildeU}
\tilde{U}_{\loss}(\H(\gamma;\loss,P);P,m,s)
\leq \spec{U}_{\loss}(\H(\gamma;\loss,P),\gamma;P,m,s)
\leq \spec{U}_{\loss}(\H,\gamma;P,m,s).
\end{equation}
Furthermore, when $\target_{P} \in \H$,  
for any measurable $\region \subseteq \region^{\prime} \subseteq \X$,
any $\gamma^{\prime} \geq \gamma \geq 0$,
and any $\H^{\prime} \subseteq \Bracket{\F}$ with $\H \subseteq \H^{\prime}$,
\begin{equation}
\label{eqn:oUHsmall-vs-oUHbig}
\spec{U}_{\loss}(\H_{\region,\target_{P}},\gamma;P,m,s) \leq \spec{U}_{\loss}(\H^{\prime}_{\region^{\prime},\target_{P}},\gamma^{\prime};P,m,s).
\end{equation}
Note that the fact that we use $\D_{\loss}(\Bracket{\H}(\gamma;\loss,P);P)$
instead of $\D_{\loss}(\H(\gamma;\loss,P);P)$ in the definition of $\spec{U}_{\loss}$
is crucial for these inequalities to hold;
specifically, it is not necessarily true that
$\D_{\loss}(\H_{\region,\target_{P}}(\gamma;\loss,P);P) \leq \D_{\loss}(\H_{\region^{\prime},\target_{P}}(\gamma;\loss,P);P)$,
but it is always the case that
$\Bracket{\H_{\region,\target_{P}}}(\gamma;\loss,P) \subseteq \Bracket{\H_{\region^{\prime},\target_{P}}}(\gamma;\loss,P)$
when $\target_{P} \in \Bracket{\H}$,
and therefore
$\D_{\loss}(\Bracket{\H_{\region,\target_{P}}}(\gamma;\loss,P);P) \leq \D_{\loss}(\Bracket{\H_{\region^{\prime},\target_{P}}}(\gamma;\loss,P);P)$.

Finally, for $\H \subseteq \Bracket{\F}$,
distribution $P$ over $\X \times \Y$,
values $\gamma,\gamma_{1},\gamma_{2} \geq 0$, $s \in [1,\infty)$,
and any function $\sfun : (0,\infty)^2 \to [1,\infty)$,
define
\begin{align*}
\spec{\SC}_{\loss}(\gamma_{1},\gamma_{2};\H,P,s) & = \min\left\{ m \in \nats : \spec{U}_{\loss}(\H,\gamma_{2};P,m,s) \leq \gamma_{1}\right\},
\\ \spec{\SC}_{\loss}(\gamma;\H,P,\sfun) & = \sup_{\gamma^{\prime} \geq \gamma} \spec{\SC}_{\loss}(\gamma^{\prime}/2,\gamma^{\prime}; \H,P,\sfun(\gamma,\gamma^{\prime})).
\end{align*}
For completeness, define $\spec{\SC}_{\loss}(\gamma_{1},\gamma_{2};\H,P,s) = \infty$
when
$\spec{U}_{\loss}(\H,\gamma_{2};P,m,s) > \gamma_{1}$
for every $m \in \nats$.

It will often be convenient to isolate the terms in $\spec{U}_{\loss}$ when inverting for a sufficient $m$,
thus arriving at an upper bound on $\spec{\SC}_{\loss}$.
Specifically, define
\begin{align*}
& \restinv{\SC}_{\loss}(\gamma_{1},\gamma_{2};\H,P,s) = \min\left\{ m \in \nats :
\D_{\loss}(\Bracket{\H}(\gamma_{2};\loss,P);P) \sqrt{\frac{s}{m}} + \frac{\maxloss s}{m} \leq \gamma_{1}\right\},
\\ & \phiinv{\SC}_{\loss}(\gamma_{1},\gamma_{2};\H,P) = \min\left\{ m \in \nats : \spec{\phi}_{\loss}\left(\D_{\loss}(\Bracket{\H}(\gamma_{2};\loss,P);P),\H;m,P\right) \leq \gamma_{1}\right\}.
\end{align*}
This way, for $\tilde{c} = 1/(2\tilde{K})$, we have
\begin{equation}
\label{eqn:spec-split-bound}
\spec{\SC}_{\loss}(\gamma_{1},\gamma_{2};\H,P,s)
\leq \max\left\{\phiinv{\SC}_{\loss}(\tilde{c} \gamma_{1}, \gamma_{2}; \H,P), \restinv{\SC}_{\loss}(\tilde{c} \gamma_{1}, \gamma_{2}; \H,P,s)\right\}.
\end{equation}
Also note that we clearly have
\begin{equation}
\label{eqn:restinv-bound}
\restinv{\SC}_{\loss}(\gamma_{1},\gamma_{2};\H,P,s) \leq
s \cdot \max\left\{\frac{4\D_{\loss}(\Bracket{\H}(\gamma_{2};\loss,P);\loss,P)^{2}}{\gamma_{1}^{2}}, \frac{2\maxloss}{\gamma_{1}}\right\},
\end{equation}
so that, in the task of bounding $\spec{\SC}_{\loss}$, we can simply focus on bounding $\phiinv{\SC}_{\loss}$.

We will express our main abstract results below in terms of the incremental values $\spec{\SC}_{\loss}(\gamma_{1},\gamma_{2};\H,\PXY,s)$;
the quantity
$\spec{\SC}_{\loss}(\gamma;\H,\PXY,\sfun)$
will also be useful in deriving explicit 
results for $\ERM_{\loss}$.
When $\target_{P} \in \H$, \eqref{eqn:oU-vs-tildeU} implies 
\begin{equation}
\label{eqn:sc-inequalities}
\bar{\SC}_{\loss}(\gamma;\H,P,\sfun)
\leq \tilde{\SC}_{\loss}(\gamma;\H,P,\sfun)
\leq \spec{\SC}_{\loss}(\gamma;\H,P,\sfun).
\end{equation}

\subsection{General Analysis of Empirical Risk Minimization}
\label{subsec:passive}

Based on Lemma~\ref{lem:koltchinskii} and the above definitions,
one can derive a bound on the number of labeled data
points $m$ sufficient for $\ERM_{\loss}(\F,\Data_{m})$
to achieve a given excess error rate.  Specifically,
the following theorem is due to \citet*{koltchinskii:06}
(slightly modified here, following \citet*{gine:06}, to allow for general $\sfun$ functions).  
It will 
be useful for deriving Theorems~\ref{thm:vc-erm} and \ref{thm:ent-erm}.
For $\eps > 0$, let $\ints_{\eps} = \{j \in \ints : 2^j \geq \eps\}$.
\begin{theorem}
\label{thm:erm}
Fix any function $\sfun : (0,\infty)^{2} \to [1,\infty)$.
If $\target \in \F$,
then for any $m \geq \bar{\SC}_{\loss}(\Transform(\eps);\F,\PXY,\sfun)$,
with probability at least
$1 - \sum_{j \in \ints_{\Transform(\eps)}} 6 e^{-\sfun(\Transform(\eps),2^{j})}$,
$\ERM_{\loss}(\F,\Data_{m})$ produces a function $\hat{h}$ such that $\er(\hat{h}) - \er(\target) \leq \eps$.
\end{theorem}

\subsection{Specification of $\hat{T}_{\loss}$ in \ACAL}
\label{subsec:hatT-spec}

The quantity $\hat{T}_{\loss}$ in \ACAL~can be defined in one of several possible ways.
In our present abstract context, we consider the following definition.
Let $\{\xi_{k}^{\prime}\}_{k \in \nats}$ denote independent Rademacher random variables
(i.e., uniform in $\{-1,+1\}$), also independent from $\Data$;
these should be considered internal random variables used by the algorithm,
which is therefore a randomized algorithm.
For any $q \in \nats \cup \{0\}$ and $Q = \{(i_1,y_1),\ldots,(i_q,y_q)\} \in (\nats \times \{-1,+1\})^{q}$,
let $\Xi[Q] = \{\xi_{i_k}^{\prime}\}_{k=1}^{q}$,
and for $s \geq 1$, 
define
$\hat{U}_{\loss}(\H;Q,s) = \hat{U}_{\loss}(\H; S[Q], \Xi[Q], s)$,
where $S[Q] = \{(X_{i_1},y_1),\ldots,(X_{i_q},y_q)\}$, as previously defined.
Then we can define the quantity $\hat{T}_{\loss}$ in the method above as
\begin{equation}
\label{eqn:hatT-acal-U}
\hat{T}_{\loss}(\H;Q,m) = \hat{U}_{\loss}(\H;Q,\hat{\sfun}(m)),
\end{equation}
for some $\hat{\sfun} : \nats \to [1,\infty)$.
This definition has the appealing property that it allows us to interpret the update
in Step 6 in two complementary ways: as comparing the empirical risks of functions in $V$ under 
samples from the conditional distribution of $(X,Y)$ given $X \in \DIS(V)$, 
and as comparing the empirical risks of the functions in $V_{\DIS(V)}$ under samples from the original distribution $\PXY$.
Our abstract results below are based on this definition of $\hat{T}_{\loss}$.   
This can sometimes be problematic due to the computational challenge of 
the optimization problems in the definitions of $\Rademacher_{\loss}$ and $\hat{\D}_{\loss}$.  There 
has been considerable work on calculating and bounding $\Rademacher_{\loss}$ for various
classes $\F$ and losses $\loss$ \citep*[e.g.,][]{koltchinskii:01,bartlett:02}, but it is not always feasible.  
However, the specific theorems stated in Section~\ref{sec:explicit} above continue to hold if we instead take $\hat{T}_{\loss}$
based on a well-chosen upper bound on the respective $\spec{U}_{\loss}$ function, such as 
those obtained in the derivations of those respective results below; we provide descriptions
of such efficiently-computable relaxations, for each of these results, in 
Appendix~\ref{sec:hatT} 
(though in some cases, these bounds have a mild dependence on $\PXY$ via certain parameters
of the specific 
noise conditions considered there).

\subsection{General Analysis of \ACAL}
\label{subsec:abstract}

The following theorem represents our main abstract result.
The key steps in its proof were already sketched above in Section~\ref{sec:acal}.
The complete proof is included in 
Appendix~\ref{app:a2}. 

\begin{theorem}
\label{thm:abstract-active}
Fix any function $\hat{\sfun} : \nats \to [1,\infty)$.  Let $j_{\loss} \!= -\lceil \log_{2}(\maxloss) \rceil$, 
$u_{j_{\loss}-2} = u_{j_{\loss}-1} = 1$,
and for each integer $j \geq j_{\loss}$,
let $\F_{j} \!=\! \F(\InvTransform(2^{2-j});\zo)_{\DIS(\F(\InvTransform(2^{2-j});\zo))}$,
$\region_{j} = \DIS(\F_{j})$,
and suppose $u_{j} \in \nats$ satisfies $\log_{2}(u_{j}) \in \nats$ and
\begin{equation}
\label{eqn:uj}
u_{j} \geq 2\spec{\SC}_{\loss}(2^{-j-1},2^{2-j};\F_{j},\PXY,\hat{\sfun}(u_{j})) \lor u_{j-1} \lor 2 u_{j-2}.
\end{equation}
Suppose $\target \in \F$.
For any $\eps \in (0,1)$, $s \in [1,\infty)$, letting $j_{\eps} = \lceil \log_{2}(1/\Transform(\eps)) \rceil$,
if
\begin{align*}
& u \geq u_{j_{\eps}}
&& \text{ and } &&
n \geq s + 2 e \sum_{j = j_{\loss} }^{j_{\eps}} \Px(\region_{j}) u_{j},
\end{align*}
then, with arguments $\loss$, $u$, and $n$,
\ACAL~uses at most $u$ unlabeled samples, requests at most $n$ labels,
and with probability at least
$1-2^{-s} - \sum_{i=1}^{\log_{2}( u_{j_{\eps}} )} 6 e^{-\hat{\sfun}( 2^{i} )}$,
returns a function $\hat{h}$ with $\er(\hat{h}) - \er(\target) \leq \eps$.
\end{theorem}

In defining and calculating the values $\spec{\SC}_{\loss}$ in Theorem~\ref{thm:abstract-active},
it is sometimes convenient to use the alternative interpretation of \ACAL,
in terms of sampling the set $S[Q]$ from 
the conditional distribution given the region of disagreement.
Specifically, 
for any measurable $\region \subseteq \X$ with $\Px(\region) > 0$, 
define the probability measure
$\PXYR{\region}(\cdot) = \PXY(\cdot | \region\times\Y)$: 
that is, $\PXYR{\region}$ is the conditional distribution of $(X,Y) \sim \PXY$ given that $X \in \region$.
Generally, for any probability measure $P$ on $\X\times\Y$, 
and any measurable $\region \subseteq \X\times\Y$ with $P(\region) > 0$,
define $P_{\region}(\cdot) = P(\cdot|\region)$.
Also, for any $\H \subseteq \Functions$, define the \emph{region of value-disagreement}
$\DISF(\H) = \{x \in \X : \exists h,g \in \H \text{ s.t. } h(x) \neq g(x)\}$,
and denote by 
$\DISFXY(\H) = \DISF(\H) \times \Y$.
The following lemma then allows us to replace calculations in terms of $\F_j$ and $\PXY$
with calculations in terms of $\F(\InvTransform(2^{1-j});\zo)$ and $\PXYR{\DIS(\F_j)}$.
Its proof is included in 
Appendix~\ref{app:a2}.

\begin{lemma}
\label{lem:ophi-conditional}
Let $\spec{\phi}_{\loss}$ be any function satisfying Definition~\ref{defn:abstract-ophi}.
Let $P$ be any distribution over $\X \times \Y$.
For any
$\sigma \geq 0$, $\H \subseteq \Bracket{\F}$, 
$m \in \nats$,
if $P\left(\DISFXY(\H)\right) > 0$, define
\begin{multline}
\label{eqn:ophi-conditional}
\spec{\phi}_{\loss}^{\prime}(\sigma,\H;m,P) =
\\ 32 \left(\inf_{\substack{\region = \region^{\prime} \times \Y : \\ \region^{\prime} \supseteq \DISF(\H)}} P(\region) \spec{\phi}_{\loss}\!\left(\frac{\sigma}{\sqrt{P(\region)}}, \H; \lceil (1/2)P(\region)m \rceil, P_{\region}\right) + \frac{\maxloss}{m} + \sigma\sqrt{\frac{1}{m}}\right),
\end{multline}
and otherwise 
$\spec{\phi}_{\loss}^{\prime}(\sigma,\H;m,P) = 0$.
Then 
$\spec{\phi}_{\loss}^{\prime}$ also satisfies Definition~\ref{defn:abstract-ophi}.
\end{lemma}

Plugging this $\spec{\phi}_{\loss}^{\prime}$ function into Theorem~\ref{thm:abstract-active}
immediately yields the following corollary; the proof is included in 
Appendix~\ref{app:a2}.

\begin{corollary}
\label{cor:abstract-conditionals}
Fix any function $\hat{\sfun} : \nats \to [1,\infty)$.
Let $j_{\loss} = -\lceil \log_{2}(\maxloss) \rceil$,
define $u_{j_{\loss}-2} = u_{j_{\loss}-1} = 1$,
and for each integer $j \geq j_{\loss}$,
let $\F_{j}$ and $\region_{j}$ be as in Theorem~\ref{thm:abstract-active},
and if $\Px(\region_{j}) > 0$,
suppose $u_{j} \in \nats$ satisfies $\log_{2}(u_{j}) \in \nats$ and
\begin{equation}
\label{eqn:uj-conditional}
u_j \geq 4 \Px(\region_{j})^{-1} \spec{\SC}_{\loss}\left( \frac{2^{-j-7}}{\Px(\region_j)}, \frac{2^{2-j}}{\Px(\region_j)}; \F_{j}, \PXYR{\region_j}, \hat{\sfun}(u_j)\right) \lor u_{j-1} \lor 2 u_{j-2}.
\end{equation}
If $\Px(\region_{j}) = 0$, let $u_{j} \in \nats$ satisfy $\log_{2}(u_{j}) \in \nats$ and
$u_{j} \geq \tilde{K} \maxloss \hat{\sfun}(u_{j}) 2^{j+2} \lor u_{j} \lor 2u_{j-2}$.
Suppose $\target \in \F$.  For any $\eps \in (0,1)$, $s \in [1,\infty)$, letting $j_{\eps} = \lceil \log_{2}(1/\Transform(\eps)) \rceil$, if
\begin{align*}
& u \geq u_{j_{\eps}}
&& \text{and } &&
n \geq s + 2e \sum_{j=j_{\loss}}^{j_{\eps}} \Px(\region_{j}) u_j,
\end{align*}
then, with arguments $\loss$, $u$, and $n$, \ACAL~uses at most $u$ unlabeled samples, requests at most $n$ labels,
and with probability at least
$1-2^{-s} - \sum_{i=1}^{\log_{2}(u_{j_{\eps}})} 6 e^{-\hat{\sfun}( 2^{i} )}$,
returns a function $\hat{h}$ with $\er(\hat{h}) - \er(\target) \leq \eps$.
\end{corollary}


\section{Derivations of the Explicit Results}
\label{sec:derivations}

We are now ready to present derivations of the explicit results of Section~\ref{sec:explicit},
based on the general results of the previous section.
To simplify the presentation, we often omit numerical constant factors in the
inequalities below, and for this we use the common notation $f(x) \lesssim g(x)$ to mean
that $f(x) \leq c g(x)$ for some implicit numerical constant $c \in (0,\infty)$.

\subsection{Specification of $\spec{\phi}_{\loss}$}
\label{subsec:phi-bounds}

We begin by recalling a few well-known bounds on the $\phi_{\loss}$ function, which lead to a more concrete instance of a function $\spec{\phi}_{\loss}$ satisfying Definition~\ref{defn:abstract-ophi}.

{\vskip 3mm}\noindent\textit{Uniform Entropy}:
The first bound is based on the work of \citet*{van-der-Vaart:11}; 
related bounds have been studied by \citet*{van-der-Vaart:96,gine:03,gine:06}, and others.
For $\sigma \geq 0$ and $\Env \in \FunctionsXY$, define the function
\begin{equation*}
J(\sigma,\G, \Env) = \sup_{\Pi} \int_{0}^{\sigma} \sqrt{1 + \ln \covering( \eps \|\Env\|_{\Pi},\G,L_{2}(\Pi))} {\rm d}\eps,
\end{equation*}
where $\Pi$ ranges over all finitely discrete probability measures.

Fix any distribution $P$ on $\X \times \Y$.
Since $J(\sigma, \G_{\H}, \Env) = J(\sigma, \G_{\H,P}, \Env)$,
it follows from Theorem 2.1 of \citet*{van-der-Vaart:11} (and a triangle inequality)
that for some universal constant $c \in [1,\infty)$, 
for any $m \in \nats$, $\Env \geq \Env(\G_{\H,P})$, and $\sigma \geq \D_{\loss}(\H;P)$,
\begin{align}
\phi_{\loss} & (\H;m,P) \leq \label{eqn:vw-uniform}
\\ & c  J\left(\frac{\sigma}{\|\Env\|_{P}},\G_{\H},  \Env\right) \|\Env\|_{P} \left(\frac{1}{\sqrt{m}} + \frac{ J\left(\frac{\sigma}{\|\Env\|_{P}},\G_{\H},  \Env\right) \|\Env\|_{P} \maxloss}{\sigma^2 m}\right). \notag
\end{align}
Based on \eqref{eqn:vw-uniform}, it is straightforward to 
define a function $\spec{\phi}_{\loss}$ 
that satisfies 
Definition~\ref{defn:abstract-ophi}.
Specifically, define
\begin{multline}
\label{eqn:uniform-ophi}
 \spec{\phi}_{\loss}^{(1)}(\sigma,\H;m,P) =
\\ \inf_{\Env \geq \Env(\G_{\H,P})} \inf_{\lambda \geq \sigma} c J\left(\frac{\lambda}{\|\Env\|_{P}},\G_{\H}, \Env\right) \|\Env\|_{P} \left(\frac{1}{\sqrt{m}} + \frac{ J\left(\frac{\lambda}{\|\Env\|_{P}},\G_{\H}, \Env\right) \|\Env\|_{P} \maxloss}{\lambda^2 m}\right),
\end{multline}
for $c$ as in \eqref{eqn:vw-uniform}.
By \eqref{eqn:vw-uniform}, $\spec{\phi}_{\loss}^{(1)}$ satisfies \eqref{eqn:ophi-2}.
Also note that $m \mapsto \spec{\phi}_{\loss}^{(1)}(\sigma,\H;m,P)$ is nonincreasing,
while $\sigma \mapsto \spec{\phi}_{\loss}^{(1)}(\sigma,\H;m,P)$ is nondecreasing.
Furthermore,
$\H \mapsto \covering(\eps, \G_{\H}, L_{2}(\Pi))$ is nondecreasing for all $\Pi$,
so that $\H \mapsto J(\sigma, \G_{\H}, \Env)$ is nondecreasing as well;
since $\H \mapsto \Env(\G_{\H,P})$ is also nondecreasing, we see that
$\H \mapsto \spec{\phi}_{\loss}^{(1)}(\sigma,\H;m,P)$ is nondecreasing.
Similarly, for $\region \subseteq \X$,
$\covering(\eps, \G_{\H_{\region,\target_{P}}}, L_{2}(\Pi))$ $\leq \covering(\eps, \G_{\H}, L_{2}(\Pi))$ for all $\Pi$,
so that $J(\sigma, \G_{\H_{\region,\target_{P}}}, \Env) \leq J(\sigma, \G_{\H}, \Env)$.
Since $\Env(\G_{\H_{\region,\target_{P}},P})$ $\leq \Env(\G_{\H,P})$,
we have
$\spec{\phi}_{\loss}^{(1)}(\sigma, \H_{\region,\target_{P}};m,P) \leq \spec{\phi}_{\loss}^{(1)}(\sigma, \H;m,P)$ as well.
Thus, to satisfy Definition~\ref{defn:abstract-ophi}, it suffices to take $\spec{\phi}_{\loss} = \spec{\phi}_{\loss}^{(1)}$.

{\vskip 3mm}\noindent\textit{Bracketing Entropy}:
Our second bound is a classic result in empirical process theory. 
For $\sigma \geq 0$, define the function
\begin{equation*}
J_{[]}(\sigma,\G, P) = \int_{0}^{\sigma} \sqrt{1 + \ln \covering_{[]}(\eps, \G, L_{2}(P))} {\rm d}\eps.
\end{equation*}
Fix any $\H \subseteq \Bracket{\F}$, and let $\G_{\H}$ and $\G_{\H,P}$ be as above.
Then since $J_{[]}(\sigma,\G_{\H},P) = J_{[]}(\sigma,\G_{\H,P},P)$,
Lemma 3.4.2 of \citep*{van-der-Vaart:96} and a triangle inequality imply that
for some universal constant $c \in [1,\infty)$, for any $m \in \nats$ and $\sigma \geq \D_{\loss}(\H;P)$,
\begin{equation}
\label{eqn:vw-bracketing}
\phi_{\loss}(\H;m,P) \leq c J_{[]}\left(\sigma, \G_{\H}, P\right) \left( \frac{1}{\sqrt{m}} + \frac{J_{[]}\left(\sigma,\G_{\H},P\right) \maxloss}{\sigma^{2} m}\right).
\end{equation}
As-is, the right side of \eqref{eqn:vw-bracketing} nearly satisfies Definition~\ref{defn:abstract-ophi} already.
Only a small 
change
is needed 
for the requirement of monotonicity in $\sigma$.
Specifically, define
\begin{equation}
\label{eqn:bracketing-ophi}
\spec{\phi}_{\loss}^{(2)}(\sigma,\H;m,P) =
\inf_{\lambda \geq \sigma} c J_{[]}\left(\lambda, \G_{\H}, P\right) \left( \frac{1}{\sqrt{m}} + \frac{J_{[]}\left(\lambda,\G_{\H},P\right) \maxloss}{\lambda^{2} m}\right),
\end{equation}
for $c$ as in \eqref{eqn:vw-bracketing}.
Then taking $\spec{\phi}_{\loss} = \spec{\phi}_{\loss}^{(2)}$ suffices to satisfy Definition~\ref{defn:abstract-ophi}.

{\vskip 3mm}
Since Definition~\ref{defn:abstract-ophi} is satisfied for both $\spec{\phi}_{\loss}^{(1)}$ and $\spec{\phi}_{\loss}^{(2)}$,
it is also satisfied for
$\spec{\phi}_{\loss} = \min\left\{ \spec{\phi}_{\loss}^{(1)}, \spec{\phi}_{\loss}^{(2)}\right\}$.  
The remainder of this section 
takes this as the specification of the $\spec{\phi}_{\loss}$ function.

\subsection{VC Subgraph Classes}
\label{subsec:vc-derivation}

The following is a 
classic 
result for VC subgraph classes \citep*[see e.g.,][]{van-der-Vaart:96},
derived from the works of \citet*{pollard:84} and \citet*{haussler:92}.
\begin{lemma}
\label{lem:vc-subgraph-covering}
For any $\G \subseteq \FunctionsXY$,
for any measurable $\Env \geq \Env(\G)$,
for any distribution $\Pi$ such that $\|\Env\|_{\Pi} > 0$,
for any $\eps \in (0,1)$,
\begin{equation*}
\covering(\eps \|\Env\|_{\Pi}, \G, L_{2}(\Pi)) \leq A(\G) \left(\frac{1}{\eps}\right)^{2\vc(\G)},
\end{equation*}
where $A(\G) \lesssim (\vc(\G) +1)(16 e)^{\vc(\G)}$.
\end{lemma}

In particular, Lemma~\ref{lem:vc-subgraph-covering} implies that any $\G \subseteq \FunctionsXY$ has,
$\forall \sigma \in (0,1]$,
\begin{equation}
\label{eqn:vc-subgraph-J-bound-bound}
J(\sigma,\G,\Env) 
\leq \int_{0}^{\sigma} \!\!\sqrt{\ln(eA(\G)) + 2\vc(\G) \ln(1/\eps)} {\rm d}\eps
\lesssim \sigma \sqrt{\vc(\G) \Log(1/\sigma)}.
\end{equation}
Applying these observations to 
$J(\sigma,\G_{\H,P},\Env)$ for $\H \subseteq \Bracket{\F}$ and $\Env \geq \Env(\G_{\H,P})$,
noting $J(\sigma,\G_{\H},\Env) = J(\sigma,\G_{\H,P},\Env)$ and $\vc(\G_{\H,P}) = \vc(\G_{\H})$,
and plugging the resulting bound into \eqref{eqn:uniform-ophi}
yields the following well-known bound on $\spec{\phi}_{\loss}^{(1)}$ due to \citet*{gine:06}.
For any $m \in \nats$ and $\sigma > 0$,
\begin{multline}
\label{eqn:vc-subgraph-phi-bound}
\spec{\phi}_{\loss}^{(1)}(\sigma, \H; m, P) 
\\ \lesssim \inf_{\lambda \geq \sigma} \lambda \sqrt{\frac{\vc(\G_{\H}) \Log\left(\frac{\|\Env(\G_{\H,P})\|_{P}}{\lambda}\right)}{m}} + \frac{\vc(\G_{\H}) \maxloss \Log\left(\frac{\|\Env(\G_{\H,P})\|_{P}}{\lambda}\right)}{m}.
\end{multline}
Specifically, to arrive at \eqref{eqn:vc-subgraph-phi-bound}, we relaxed the $\inf_{\Env \geq \Env(\G_{\H,P})}$
in \eqref{eqn:uniform-ophi} by taking $\Env \geq \Env(\G_{\H,P})$ such that $\|\Env\|_{P} = \max\{ \sigma, \|\Env(\G_{\H,P})\|_{P} \}$,
thus maintaining $\lambda / \|\Env\|_{P} \in (0,1]$ for the minimizing $\lambda$ value, so that \eqref{eqn:vc-subgraph-J-bound-bound} remains valid;
we also used the fact that $\Log \geq 1$, which gives us $\Log(\|\Env\|_{P} / \lambda) = \Log(\|\Env(\G_{\H,P})\|_{P} / \lambda)$ for this case.

In particular, \eqref{eqn:vc-subgraph-phi-bound} implies
\begin{multline}
\label{eqn:vc-subgraph-pre-phiinv}
\phiinv{\SC}_{\loss}(\gamma_1,\gamma_2;\H,P)
\\ \lesssim \inf_{\sigma \geq \D_{\loss}(\Bracket{\H}(\gamma_{2};\loss,P);P)} \left(\frac{\sigma^{2}}{\gamma_{1}^{2}} + \frac{\maxloss}{\gamma_{1}}\right) \vc(\G_{\H}) \Log\left( \frac{\| \Env(\G_{\H,P}) \|_{P}}{\sigma}\right).
\end{multline}
For $\lambda > 0$, when $\target_{P} \in \H$ and $P$ satisfies Condition~\ref{con:tsybakov-sur}, \eqref{eqn:vc-subgraph-pre-phiinv} implies that,
\begin{multline}
\label{eqn:vc-subgraph-phiinv}
\sup_{\gamma \geq \lambda} \phiinv{\SC}_{\loss}(\gamma/(4 \tilde{K}),\gamma; \H(\gamma;\loss,P), P)
\\ \lesssim \left(\frac{\tsybcb}{\lambda^{2-\tsybb}} + \frac{\maxloss}{\lambda}\right) \vc(\G_{\H}) \Log\left( \capacity_{\loss}\left( \tsybcb \lambda^{\tsybb} ; \H, P \right) \right).
\end{multline}

Combining this observation with
\eqref{eqn:oU-vs-tildeU}, \eqref{eqn:spec-split-bound}, \eqref{eqn:restinv-bound}, \eqref{eqn:sc-inequalities}, and Theorem~\ref{thm:erm},
we arrive at a result for the sample complexity of
empirical $\loss$-risk minimization with a general VC subgraph class under
Conditions~\ref{con:tsybakov-01} and \ref{con:tsybakov-sur}.
Specifically, for $\sfun \!:\! (0,\infty)^{2} \to [1,\infty)$,
when $\target \in \F$, 
\eqref{eqn:oU-vs-tildeU} implies that
\begin{align}
 \bar{\SC}_{\loss}(\Transform(\eps);\F,\PXY,\sfun)
& \leq \tilde{\SC}_{\loss}(\Transform(\eps); \F, \PXY, \sfun) \notag
\\ & = \sup_{\gamma \geq \Transform(\eps)} \tilde{\SC}_{\loss}(\gamma/2,\gamma; \F(\gamma;\loss), \PXY, \sfun(\Transform(\eps),\gamma)) \notag
\\ &\leq \sup_{\gamma \geq \Transform(\eps)} \spec{\SC}_{\loss}(\gamma/2,\gamma; \F(\gamma;\loss), \PXY, \sfun(\Transform(\eps),\gamma)). \label{eqn:spec-sc-bound-on-bar-sc}
\end{align}
For $\PXY$ satisfying Conditions \ref{con:tsybakov-01} and \ref{con:tsybakov-sur},
applying \eqref{eqn:spec-split-bound}, \eqref{eqn:restinv-bound}, and \eqref{eqn:vc-subgraph-phiinv} to \eqref{eqn:spec-sc-bound-on-bar-sc},
and taking $\sfun(\lambda,\gamma) = \Log\left(\frac{12 \gamma}{\lambda \conf}\right)$,
we arrive at Theorem~\ref{thm:vc-erm} (which is implicit in 
\citep*{gine:06}).

Next, we turn to 
Theorem~\ref{thm:vc-subgraph-abstract}.
Note that 
$\vc(\G_{\F_{j}}) \leq \vc(\G_{\F(\InvTransform(2^{2-j});\zo)}) \leq \vc(\G_{\F})$.
Also, $\| \Env(\G_{\F_{j},\PXY}) \|_{\PXY}^{2} \leq \maxloss^{2} \Px\left(\DIS\left( \F\left( \InvTransform\left(2^{2-j}\right); \zo\right)\right)\right)$.
Thus, for $j_{\loss} \leq j \leq \lceil \log_{2}(1/\BJM(\eps)) \rceil$,
\eqref{eqn:vc-subgraph-pre-phiinv} implies
\begin{equation}
\label{eqn:vc-subgraph-mixed}
\phiinv{\SC}_{\loss}(2^{-j-2} \tilde{K}^{-1}, 2^{2-j};\F_{j},\PXY) \lesssim
 \left(\tsybcb 2^{j(2-\tsybb)} + \maxloss 2^{j}\right) \vc(\G_{\F}) \Log\left( \mixcap\left(\BJM(\eps)\right) \maxloss\right).
\end{equation}

With a little additional work to define an appropriate $\hat{\sfun}$ function
and derive closed-form bounds on the summation in Theorem~\ref{thm:abstract-active},
we arrive at Theorem~\ref{thm:vc-subgraph-abstract}.
The remaining 
details 
appear in 
Appendix~\ref{app:applications}.

When $\loss$ satisfies Condition~\ref{con:strong-convexity},
we can derive the sometimes-stronger result in Theorem~\ref{thm:vc-subgraph-active-strong} via Corollary~\ref{cor:abstract-conditionals}.
Specifically, combining \eqref{eqn:vc-subgraph-pre-phiinv}, \eqref{eqn:spec-split-bound}, \eqref{eqn:restinv-bound},
and Lemma~\ref{lem:bjm-strong-convexity},
we have that if $\target \in \F$ and Condition~\ref{con:strong-convexity} is satisfied,
then for $j \geq j_{\loss}$
in Corollary~\ref{cor:abstract-conditionals},
\begin{align}
& \spec{\SC}_{\loss}\left( \frac{2^{-j-7}}{\Px(\region_{j})}, \frac{2^{2-j}}{\Px(\region_{j})}; \F_{j}, \PXYR{\region_{j}}, s\right) \label{eqn:vc-subgraph-strong-conditional}
\\ & \lesssim  \left( \tsybcb \left(2^{j} \Px(\region_{j})\right)^{2-\tsybb} + 2^{j} \maxloss \Px(\region_{j})\right)\left( \vc(\G_{\F}) \Log\left(\maxloss^{2} 2^{j \tsybb}\Px(\region_{j})^{\tsybb} / \tsybcb\right)
+ s\right), \notag
\end{align}
where $\tsybcb$ and $\tsybb$ are as in Lemma~\ref{lem:bjm-strong-convexity}.
Plugging this into Corollary~\ref{cor:abstract-conditionals}, 
we arrive at Theorem~\ref{thm:vc-subgraph-active-strong};
the remaining details proceed
similarly to those of Theorem~\ref{thm:vc-subgraph-abstract}, 
and a detailed sketch appears in 
Appendix~\ref{app:applications}.

\subsection{Entropy Conditions}
\label{subsec:entropy-derivation}

Next we turn to problems satisfying entropy conditions.
Note that when $\F$ satisfies Condition~\ref{con:entropy}, for $0 \leq \sigma \leq 2\|\Env\|_{\PXY}$,
\begin{equation}
\label{eqn:ent-phi-bound}
\spec{\phi}_{\loss}(\sigma,\F;m,\PXY) \lesssim
\max\left\{\frac{\sqrt{\entc} \|\Env\|_{\PXY}^{\entrho} \sigma^{1-\entrho}}{(1-\entrho) m^{1/2}}, \frac{\maxloss^{\frac{1-\entrho}{1+\entrho}} \entc^{\frac{1}{1+\entrho}} \|\Env\|_{\PXY}^{\frac{2\entrho}{1+\entrho}}}{(1-\entrho)^{\frac{2}{1+\entrho}} m^{\frac{1}{1+\entrho}}}\right\}.
\end{equation}
Since $\D_{\loss}(\Bracket{\F}) \leq 2 \|\Env\|_{\PXY}$, this implies that for any numerical constant $c \in (0,1]$,
for every $\gamma \in (0,\infty)$, if $\PXY$ satisfies Condition~\ref{con:tsybakov-sur}, then
\begin{equation}
\label{eqn:ent-phiinv-tsybakov}
\phiinv{\SC}_{\loss}(c \gamma, \gamma;\F,\PXY) \lesssim
\frac{\entc \|\Env\|_{\PXY}^{2\entrho}}{(1-\entrho)^{2}} \max\left\{ \tsybcb^{1-\entrho} \gamma^{\tsybb (1-\entrho) - 2} , \maxloss^{1-\entrho} \gamma^{-(1+\entrho)}\right\}.
\end{equation}
Combined with \eqref{eqn:spec-split-bound}, \eqref{eqn:restinv-bound}, \eqref{eqn:sc-inequalities}, and Theorem~\ref{thm:erm},
taking $\sfun(\lambda,\gamma) = \Log\left(\frac{12 \gamma}{\lambda\conf}\right)$,
we arrive at the classic result in Theorem~\ref{thm:ent-erm} \citep*[e.g.,][]{bartlett:06,van-der-Vaart:96}.

The corresponding result for \ACAL, namely Theorem~\ref{thm:ent-active}, 
follows by combining \eqref{eqn:ent-phiinv-tsybakov} with \eqref{eqn:spec-split-bound}, \eqref{eqn:restinv-bound}, and Theorem~\ref{thm:abstract-active}.
The details of the proof follow analogously to that of Theorem~\ref{thm:vc-subgraph-abstract}, and are therefore omitted for brevity.

Next, we turn to deriving the corresponding results stated above under Condition~\ref{con:strong-convexity}.
As discussed above, we treat separately 
the cases
of \eqref{eqn:uniform-entropy-bound} and \eqref{eqn:bracketing-entropy-bound}. 

First, suppose \eqref{eqn:uniform-entropy-bound} holds (for all $P$, $\eps$) with $\Env \leq \maxloss$.
Following the derivation of \eqref{eqn:ent-phiinv-tsybakov} above,
combined with \eqref{eqn:restinv-bound}, \eqref{eqn:spec-split-bound}, and Lemma~\ref{lem:bjm-strong-convexity},
for $j \geq j_{\loss}$
in Corollary~\ref{cor:abstract-conditionals},
\begin{multline*}
 \spec{\SC}_{\loss}\left(\frac{2^{-j-7}}{\Px(\region_{j})},  \frac{2^{2-j}}{\Px(\region_{j})}; \F_{j}, \PXYR{\region_{j}}, s\right)
\lesssim  \left( \tsybcb \left(2^{j} \Px(\region_{j})\right)^{2-\tsybb}
+ \maxloss 2^{j} \Px(\region_{j})\right) s
\\ + \frac{\entc \maxloss^{2\entrho}}{(1-\entrho)^{2}}
\left(\tsybcb^{1-\entrho} \left(2^{j} \Px(\region_{j})\right)^{2-\tsybb(1-\entrho)} +
\maxloss^{1-\entrho} \left(2^{j} \Px(\region_{j})\right)^{1+\entrho}\right),
\end{multline*}
where $\tsybcb$ and $\tsybb$ are from Lemma~\ref{lem:bjm-strong-convexity}.
This immediately leads to Theorem~\ref{thm:uniform-entropy-strong-convexity} by
reasoning analogous to the proof of Theorem~\ref{thm:vc-subgraph-active-strong}.

The case 
\eqref{eqn:bracketing-entropy-bound} 
can be treated similarly, though the result we obtain 
(Theorem~\ref{thm:bracketing-entropy-strong-convexity}) is slightly weaker.
Suppose \eqref{eqn:bracketing-entropy-bound} 
is satisfied 
with $\Env = \maxloss$ constant.
In this case, 
$\maxloss \geq \Env(\G_{\F_{j},\PXYR{\region_{j}}})$, 
while
$\covering_{[]}(\eps \maxloss, \G_{\F_{j}}, L_{2}(\PXYR{\region_{j}})) = \covering_{[]}(\eps \maxloss \sqrt{\Px(\region_{j})}, \G_{\F_{j}}, L_{2}(\PXY)) 
\leq$\break $\covering_{[]}(\eps \maxloss \sqrt{\Px(\region_{j})}, \G_{\F}, L_{2}(\PXY))$,
so that $\F_{j}$ and 
$\PXYR{\region_{j}}$ also satisfy \eqref{eqn:bracketing-entropy-bound} with $\Env = \maxloss$:
\begin{equation*}
\ln \covering_{[]}\left( \eps \maxloss, \G_{\F_{j}}, L_{2}(\PXYR{\region_{j}})\right) \leq \entc \Px(\region_{j})^{-\entrho} \eps^{-2\entrho}.
\end{equation*}
Thus, based on \eqref{eqn:ent-phiinv-tsybakov},
\eqref{eqn:spec-split-bound}, \eqref{eqn:restinv-bound},
and Lemma~\ref{lem:bjm-strong-convexity},
we have that if $\target \in \F$ and Condition~\ref{con:strong-convexity} is satisfied,
then for $j \geq j_{\loss}$ in Corollary~\ref{cor:abstract-conditionals},
\begin{multline*}
\spec{\SC}_{\loss}\left( \frac{2^{-j-7}}{\Px(\region_{j})}, \frac{2^{2-j}}{\Px(\region_{j})}; \F_{j}, \PXYR{\region_{j}}, s\right)
\lesssim \left(\tsybcb \left(2^{j} \Px(\region_{j})\right)^{2-\tsybb} + \maxloss 2^{j} \Px(\region_{j}) \right) s
\\ + \left(\frac{\entc \maxloss^{2\entrho}}{(1-\entrho)^{2}}\right)
\Px(\region_{j})^{-\entrho}\left( \tsybcb^{1-\entrho} \left(2^{j} \Px(\region_{j})\right)^{2-\tsybb(1-\entrho)} + \maxloss^{1-\entrho} \left(2^{j} \Px(\region_{j})\right)^{1+\entrho} \right),
\end{multline*}
where $\tsybcb$ and $\tsybb$ are as in Lemma~\ref{lem:bjm-strong-convexity}.
Combining this with Corollary~\ref{cor:abstract-conditionals}
and reasoning analogously to the proof of Theorem~\ref{thm:vc-subgraph-active-strong},
we obtain Theorem~\ref{thm:bracketing-entropy-strong-convexity}.

\bibliographystyle{abbrvnat}
\bibliography{learning}                

\appendix

\section{Main Proofs}
\label{app:a2}

This appendix includes the proofs of the main abstract results from Section~\ref{sec:abstract}.

\begin{proof}[Proof of Theorem~\ref{thm:abstract-active}]
Fix any $\eps \in (0,1)$, $s \in [1,\infty)$, values $u_{j}$ satisfying \eqref{eqn:uj},
and consider running \ACAL~with values of $u$ and $n$ satisfying the conditions 
specified in Theorem~\ref{thm:abstract-active}.
The proof has two main components:
first, showing that, with high probability, $\target \in V$ is maintained as an invariant,
and second, showing that, with high probability, the set $V$ will be sufficiently reduced
to provide the guarantee on $\hat{h}$ after at most the stated number of label requests, given
the value of $u$ is as large as stated.
Both of these components are served by the following application of
Lemma~\ref{lem:koltchinskii}.

Let $S$ denote the set of values of $m$ obtained in \ACAL~for which $\log_{2}(m) \in \nats$.
For each $m \in S$,
let $V^{(m)}$ and $Q_{m}$ denote the values of $V$ and $Q$ (respectively)
upon reaching Step 5 on the round that \ACAL~obtains that value of $m$,
and let $\tilde{V}^{(m)}$ denote the value of $V$ upon completing Step 6 on that round;
also denote $\DISV_{m} = \DIS(V^{(m)})$ and $\L_{m} = \{(1+m/2,Y_{1+m/2}),\ldots,(m,Y_{m})\}$,
and define $\tilde{V}^{(1)} = \F$ and $\DISV_{1} = \DIS(\F)$.

Consider any $m \in S$, and note that $\forall h,g \in V^{(m)}$,
\begin{multline}
\label{eqn:empirical-risk-swap-Q-m}
(|Q_m| \lor 1) \left(\Risk(h;Q_m) - \Risk(g;Q_m)\right) 
\\ = \frac{m}{2}\left(\Risk(h_{\DISV_{m}};\L_m) - \Risk(g_{\DISV_{m}};\L_m)\right),
\end{multline}
and furthermore that
\begin{equation}
\label{eqn:hatU-swap-Q-m}
(|Q_m|\lor 1)\hat{U}_{\loss}(V^{(m)};Q_m, \hat{\sfun}(m)) 
= \frac{m}{2} \hat{U}_{\loss}( V^{(m)}_{\DISV_{m}}; \L_m, \hat{\sfun}(m)).
\end{equation}
Applying Lemma~\ref{lem:koltchinskii} under the conditional distribution given $V^{(m)}$,
combined with the law of total probability,
we have that, for every $m \in \nats$ with $\log_{2}(m) \in \nats$,
on an event of probability at least $1- 6 e^{-\hat{\sfun}(m)}$,
if $\target \in V^{(m)}$ and $m \in S$,
then letting $\hat{U}_{m} = \hat{U}_{\loss}\left(V^{(m)}_{\DISV_{m}};\L_m,\hat{\sfun}(m)\right)$,
every $h_{\DISV_{m}} \in V^{(m)}_{\DISV_{m}}$ has
\begin{align}
& \Risk(h_{\DISV_{m}}) - \Risk(\target)  <  \Risk(h_{\DISV_{m}}; \L_m) - \Risk(\target;\L_m) + \hat{U}_{m}, \label{eqn:a2-kolt-1}
\\ & \Risk(h_{\DISV_{m}};\L_m) - \min_{g_{\DISV_{m}} \in V^{(m)}_{\DISV_{m}}} \Risk(g_{\DISV_{m}};\L_m) < \Risk(h_{\DISV_{m}}) - \Risk(\target) + \hat{U}_{m}, \label{eqn:a2-kolt-2}
\end{align}
and furthermore
\begin{equation}
\hat{U}_{m} < \tilde{U}_{\loss}\left(V^{(m)}_{\DISV_{m}};\PXY,m/2,\hat{\sfun}(m)\right). \label{eqn:a2-kolt-3}
\end{equation}
By a union bound, on an event of probability at least
$1 - \sum_{i = 1}^{\log_{2}(u_{j_{\eps}})} 6 e^{-\hat{\sfun}(2^{i})}$,
for every $m \in S$ with $m \leq u_{j_{\eps}}$ and $\target \in V^{(m)}$,
the inequalities \eqref{eqn:a2-kolt-1}, \eqref{eqn:a2-kolt-2}, and \eqref{eqn:a2-kolt-3} hold.
Call this event $E$. 

In particular, note that on the event $E$, for any $m \in S$ with $m \leq u_{j_{\eps}}$ and $\target \in V^{(m)}$,
since $\target_{\DISV_{m}} = \target$, \eqref{eqn:empirical-risk-swap-Q-m}, \eqref{eqn:a2-kolt-2}, and \eqref{eqn:hatU-swap-Q-m} imply
\begin{multline*}
(|Q_{m}| \lor 1) \left(\Risk(\target;Q_{m}) - \inf_{g \in V^{(m)}} \Risk(g;Q_{m})\right)
\\ = \frac{m}{2} \left( \Risk(\target;\L_{m}) - \inf_{g_{\DISV_{m}} \in V^{(m)}_{\DISV_{m}}} \Risk(g_{\DISV_{m}};Q_{m}) \right)
\\ < \frac{m}{2} \hat{U}_{m}
= (|Q_{m}| \lor 1) \hat{U}_{\loss}(V^{(m)};Q_{m},\hat{\sfun}(m)),
\end{multline*}
so that $\target \in \tilde{V}^{(m)}$ as well.  Since $\target \in V^{(2)}$, and every $m \in S$ with $m > 2$ has $V^{(m)} = \tilde{V}^{(m/2)}$,
by induction we have that, on the event $E$, every $m \in S$ with $m \leq u_{j_{\eps}}$ has $\target \in V^{(m)}$ and $\target \in \tilde{V}^{(m)}$;
this also implies that \eqref{eqn:a2-kolt-1}, \eqref{eqn:a2-kolt-2}, and \eqref{eqn:a2-kolt-3} all hold for these values of $m$ on the event $E$.

We next prove by induction that, on the event $E$, $\forall j \in \{j_{\loss}-2,j_{\loss}-1,j_{\loss},\ldots,j_{\eps}\}$, if $u_{j} \in S \cup \{1\}$,
then $\tilde{V}^{(u_{j})}_{\DISV_{u_{j}}} \subseteq \Bracket{\F}(2^{-j};\loss)$ and $\tilde{V}^{(u_{j})} \subseteq \F\left( \InvTransform(2^{-j}) ; \zo \right)$.
This claim is trivially satisfied for $j \in \{ j_{\loss}-2, j_{\loss}-1\}$, since in that case
$\Bracket{\F}(2^{-j};\loss) = \Bracket{\F} \supseteq \tilde{V}^{(u_{j})}_{\DISV_{u_{j}}}$ 
and $\F(\InvTransform(2^{-j});\zo) = \F$, so that these values can serve as our base case.
Now take as an inductive hypothesis that, for some $j \in \{j_{\loss},\ldots,j_{\eps}\}$,
if $u_{j-2} \in S \cup \{1\}$, then on the event $E$, 
$\tilde{V}^{(u_{j-2})}_{\DISV_{u_{j-2}}} \subseteq \Bracket{\F}(2^{2-j};\loss)$
and $\tilde{V}^{(u_{j-2})} \subseteq \F\left( \InvTransform(2^{2-j});\zo \right)$, 
and suppose the event $E$ occurs.
If $u_{j} \notin S$, the claim is trivially satisfied; otherwise, suppose $u_{j} \in S$, which further implies $u_{j-2} \in S \cup \{1\}$.
Since $u_{j} \leq u_{j_{\eps}}$, for any $h \in \tilde{V}^{(u_{j})}$, \eqref{eqn:a2-kolt-1} implies
\begin{equation*}
\frac{u_{j}}{2}\left(\Risk(h_{\DISV_{u_{j}}}) - \Risk(\target)\right)
< \frac{u_{j}}{2}\left(\Risk(h_{\DISV_{u_{j}}};\L_{u_{j}}) - \Risk(\target;\L_{u_{j}}) + \hat{U}_{u_{j}}\right).
\end{equation*}
Since we have already established that $\target \in V^{(u_{j})}$, \eqref{eqn:empirical-risk-swap-Q-m} and \eqref{eqn:hatU-swap-Q-m} imply
\begin{multline*}
\frac{u_{j}}{2}\left(\Risk(h_{\DISV_{u_{j}}};\L_{u_{j}}) - \Risk(\target;\L_{u_{j}}) + \hat{U}_{u_{j}}\right)
\\ = (|Q_{u_{j}}| \lor 1)\left( \Risk(h;Q_{u_{j}}) - \Risk(\target;Q_{u_{j}}) + \hat{U}_{\loss}(V^{(u_{j})};Q_{u_{j}},\hat{\sfun}(u_{j})) \right).
\end{multline*}
The definition of $\tilde{V}^{(u_{j})}$ from Step 6 implies
\begin{multline*}
(|Q_{u_{j}}| \lor 1)\left( \Risk(h;Q_{u_{j}}) - \Risk(\target;Q_{u_{j}}) + \hat{U}_{\loss}(V^{(u_{j})};Q_{u_{j}},\hat{\sfun}(u_{j})) \right)
\\ \leq (|Q_{u_{j}}| \lor 1)\left( 2\hat{U}_{\loss}(V^{(u_{j})};Q_{u_{j}},\hat{\sfun}(u_{j})) \right).
\end{multline*}
By \eqref{eqn:hatU-swap-Q-m} and \eqref{eqn:a2-kolt-3},
\begin{equation*}
(|Q_{u_{j}}| \!\lor\! 1)\!\left( 2\hat{U}_{\loss}(V^{(u_{j})};Q_{u_{j}},\hat{\sfun}(u_{j})) \right)
\!=\! u_{j} \hat{U}_{u_{j}}
\!<\! u_{j} \tilde{U}_{\loss}\!\left(V^{(u_{j})}_{\DISV_{u_{j}}};\PXY,u_{j}/2,\hat{\sfun}(u_{j})\right)\!.
\end{equation*}
Altogether, we have that, $\forall h \in \tilde{V}^{(u_{j})}$,
\begin{equation}
\label{eqn:abstract-risk-tilde-bound}
\Risk(h_{\DISV_{u_{j}}}) - \Risk(\target) < 2 \tilde{U}_{\loss}\left(V^{(u_{j})}_{\DISV_{u_{j}}};\PXY,u_{j}/2,\hat{\sfun}(u_{j})\right).
\end{equation}
By definition of $\spec{\SC}_{\loss}$, monotonicity of $m \mapsto \spec{U}_{\loss}(\cdot,\cdot;\cdot,m,\cdot)$, 
and the condition on $u_{j}$ in \eqref{eqn:uj}, we know that
\begin{equation*}
\spec{U}_{\loss}\left(\F_{j},2^{2-j};\PXY,u_{j}/2,\hat{\sfun}(u_{j})\right) \leq 2^{-j-1}.
\end{equation*}
The fact that $u_{j} \geq 2 u_{j-2}$, combined with the inductive hypothesis, implies
\begin{equation*}
V^{(u_{j})} \subseteq \tilde{V}^{(u_{j-2})} \subseteq \F\left( \InvTransform(2^{2-j}); \zo \right).
\end{equation*}
This also implies $\DISV_{u_{j}} \subseteq \DIS(\F(\InvTransform(2^{2-j});\zo))$.
Combined with \eqref{eqn:oUHsmall-vs-oUHbig}, these imply
\begin{equation*}
\spec{U}_{\loss}\left(V^{(u_{j})}_{\DISV_{u_{j}}},2^{2-j};\PXY,u_{j}/2,\hat{\sfun}(u_{j})\right) \leq 2^{-j-1}.
\end{equation*}
Together with \eqref{eqn:oU-vs-tildeU}, this implies
\begin{equation*}
\tilde{U}_{\loss}\left(V^{(u_{j})}_{\DISV_{u_{j}}}(2^{2-j};\loss);\PXY,u_{j}/2,\hat{\sfun}(u_{j})\right) \leq 2^{-j-1}.
\end{equation*}
The inductive hypothesis implies $V^{(u_{j})}_{\DISV_{u_{j}}} = V^{(u_{j})}_{\DISV_{u_{j}}}(2^{2-j};\loss)$, which means
\begin{equation*}
\tilde{U}_{\loss}\left(V^{(u_{j})}_{\DISV_{u_{j}}};\PXY,u_{j}/2,\hat{\sfun}(u_{j})\right) \leq 2^{-j-1}.
\end{equation*}
Plugging this into \eqref{eqn:abstract-risk-tilde-bound} implies, $\forall h \in \tilde{V}^{(u_{j})}$,
\begin{equation}
\label{eqn:abstract-risk-increment}
\Risk(h_{\DISV_{u_{j}}}) - \Risk(\target) < 2^{-j}.
\end{equation}
In particular, since $\target \in \F$, we always have $\tilde{V}^{(u_{j})}_{\DISV_{u_{j}}} \subseteq \Bracket{\F}$,
so that \eqref{eqn:abstract-risk-increment} establishes that $\tilde{V}^{(u_{j})}_{\DISV_{u_{j}}} \subseteq \Bracket{\F}(2^{-j};\loss)$.
Furthermore, since $\target \in V^{(u_{j})}$ on $E$, $\sign(h_{\DISV_{u_{j}}}) = \sign(h)$ for every $h \in \tilde{V}^{(u_{j})}$, 
so that every $h \in \tilde{V}^{(u_{j})}$ has $\er(h) = \er(h_{\DISV_{u_{j}}})$, and therefore (by definition of $\InvTransform(\cdot)$),
\eqref{eqn:abstract-risk-increment} implies
\begin{equation*}
\er(h) - \er(\target) = \er(h_{\DISV_{u_{j}}}) - \er(\target) \leq \InvTransform\left(2^{-j}\right).
\end{equation*}
This implies $\tilde{V}^{(u_{j})} \subseteq \F\left( \InvTransform(2^{-j}); \zo \right)$, 
which completes the inductive proof.
This implies that, on the event $E$, if $u_{j_{\eps}} \in S$, then (by monotonicity of $\InvTransform(\cdot)$ and the fact that $\InvTransform(\Transform(\eps)) \leq \eps$)
\begin{equation*}
\tilde{V}^{(u_{j_{\eps}})} \subseteq \F( \InvTransform(2^{-j_{\eps}}); \zo ) \subseteq \F( \InvTransform( \Transform(\eps) ); \zo ) \subseteq \F(\eps;\zo).
\end{equation*}
In particular, since the update in Step 6 always keeps at least one element in $V$, the function $\hat{h}$ in Step 8
exists, and has $\hat{h} \in \tilde{V}^{(u_{j_{\eps}})}$ (if $u_{j_{\eps}} \in S$).  Thus, on the event $E$, if $u_{j_{\eps}} \in S$, 
then $\er(\hat{h}) - \er(\target) \leq \eps$.  Therefore, since $u \geq u_{j_{\eps}}$, to complete the proof it suffices 
to show that taking $n$ of the size indicated in the theorem statement suffices to guarantee $u_{j_{\eps}} \in S$, 
on an event (which includes $E$) having at least the stated probability.

Note that for any $j \in \{j_{\loss}, \ldots, j_{\eps}\}$ with $u_{j-1} \in S \cup \{1\}$,
every $m \in \{u_{j-1} + 1,\ldots,u_{j}\} \cap S$ has $V^{(m)} \subseteq \tilde{V}^{(u_{j-1})}$;
furthermore, we showed above that on the event $E$, if $u_{j-1} \in S$, then 
$\tilde{V}^{(u_{j-1})} \subseteq \F( \InvTransform(2^{1-j}); \zo)$,
so that $\DIS(V^{(m)}) \subseteq \DIS(\tilde{V}^{(u_{j-1})}) \subseteq \DIS(\F(\InvTransform(2^{1-j});\zo)) \subseteq \region_{j}$. 
Thus, on the event $E$, to guarantee $u_{j_{\eps}} \in S$, it suffices to have
\begin{equation*}
n \geq \sum_{j= j_{\loss}}^{j_{\eps}} \sum_{m = u_{j-1}+1}^{u_{j}} \ind_{\region_{j}}(X_{m}).
\end{equation*}
Noting that this is a sum of independent Bernoulli random variables, 
a Chernoff bound implies that on an event $E^{\prime}$ of probability at least $1 - 2^{-s}$,
\begin{multline*}
\sum_{j= j_{\loss}}^{j_{\eps}} \sum_{m = u_{j-1}+1}^{u_{j}} \ind_{\region_{j}}(X_{m})
\leq s + 2 e \sum_{j= j_{\loss}}^{j_{\eps}} \sum_{m = u_{j-1}+1}^{u_{j}} \Px(\region_{j})
\\ = s + 2 e \sum_{j= j_{\loss}}^{j_{\eps}} \Px(\region_{j}) (u_{j} - u_{j-1})
\leq s + 2 e \sum_{j= j_{\loss}}^{j_{\eps}} \Px(\region_{j}) u_{j}.
\end{multline*}
Thus, for $n$ satisfying the condition in the theorem statement, on the event $E \cap E^{\prime}$,
we have $u_{j_{\eps}} \in S$, and therefore (as proven above) $\er(\hat{h}) - \er(\target) \leq \eps$.
Finally, a union bound implies that the event $E \cap E^{\prime}$ has probability at least
\begin{equation*}
1 - 2^{-s} - \sum_{i=1}^{\log_{2}(u_{j_{\eps}})} 6 e^{-\hat{\sfun}(2^{i})},
\end{equation*}
as required.
\end{proof}

\begin{proof}[Proof of Lemma~\ref{lem:ophi-conditional}]
If $P\left(\DISFXY(\H)\right) = 0$, then $\phi_{\loss}(\H;m,P) = 0$, so that in this case, $\spec{\phi}_{\loss}^{\prime}$  trivially satisfies \eqref{eqn:ophi-2}.
Otherwise, suppose $P\left(\DISFXY(\H)\right) > 0$.
By the classic symmetrization inequality \citep*[e.g.,][Lemma 2.3.1]{van-der-Vaart:96},
\begin{equation*}
\phi_{\loss}(\H;m,P)\leq 2 \E\left[ \left| \Rademacher_{\loss}(\H;S,\Xi_{[m]}) \right| \right],
\end{equation*}
where $S \sim P^{m}$ and $\Xi_{[m]} = \{\xi_1,\ldots,\xi_m\} \sim {\rm Uniform}(\{-1,+1\}^{m})$ are independent.
Fix any measurable $\region \supseteq \DISFXY(\H)$.
Then
\begin{equation}
\label{eqn:ophi-conditional-intersect}
\E\left[\left|\Rademacher_{\loss}(\H; S, \Xi_{[m]})\right|\right] = \E\left[\left| \Rademacher_{\loss}(\H; S \cap \region, \Xi_{[|S\cap\region|]}) \right| \frac{|S\cap\region|}{m}\right],
\end{equation}
where $\Xi_{[q]} = \{\xi_1,\ldots,\xi_{q}\}$ for any $q \in \{0,\ldots,m\}$.
By the classic desymmetrization inequality \citep*[see e.g.,][]{koltchinskii:08},
applied under the conditional distribution given $|S \cap \region|$,
the right hand side of \eqref{eqn:ophi-conditional-intersect} is at most
\begin{equation}
\label{eqn:conditional-desymmetrization}
\E\!\left[2 \phi_{\loss}(\H; |S \cap \region|, P_{\region}) \frac{|S\cap\region|}{m}\right] 
+ \sup_{h,g \in \H} \!\left|\Risk(h ; P_{\region}) - \Risk(g ; P_{\region})\right| \frac{\E\!\left[\!\sqrt{|S\cap\region|}\right]}{m}.
\end{equation}
By Jensen's inequality, the second term in \eqref{eqn:conditional-desymmetrization} is at most
\begin{equation*}
\sup_{h,g \in \H} \!\left|\Risk(h ; P_{\region}) - \Risk(g ; P_{\region})\right|\! \sqrt{\frac{P(\region)}{m}}
\leq \D_{\loss}(\H; P_{\region}) \sqrt{\frac{P(\region)}{m}}
= \D_{\loss}(\H; P) \sqrt{\frac{1}{m}}.
\end{equation*}
Decomposing based on $|S\cap\region|$, the first term in \eqref{eqn:conditional-desymmetrization} is at most
\begin{multline}
\label{eqn:ophi-conditional-decompose}
\E\left[ 2 \phi_{\loss}(\H; |S \cap \region|, P_{\region}) \frac{|S\cap\region|}{m} \ind\left[ |S\cap\region| \geq (1/2) P(\region) m\right]\right]
\\ + 2 \maxloss P(\region) \P\left(|S\cap\region| < (1/2)P(\region)m\right).
\end{multline}
Since $|S\cap\region| \geq (1/2) P(\region) m \Rightarrow |S\cap\region| \geq \lceil (1/2) P(\region) m \rceil$,
and $\phi_{\loss}(\H; q, P_{\region})$ is nonincreasing in $q$, 
the first term in \eqref{eqn:ophi-conditional-decompose} is at most
\begin{equation*}
2 \phi_{\loss}(\H; \lceil (1/2) P(\region) m\rceil, P_{\region}) \E\left[ \frac{|S\cap\region|}{m} \right]
= 2 \phi_{\loss}(\H; \lceil (1/2) P(\region) m\rceil, P_{\region}) P(\region) ,
\end{equation*}
while a Chernoff bound implies the second term in \eqref{eqn:ophi-conditional-decompose} is at most
\begin{equation*}
2 \maxloss P(\region) \exp\left\{ - P(\region)m / 8\right\}
\leq \frac{16 \maxloss}{m}. 
\end{equation*}
Plugging back into \eqref{eqn:conditional-desymmetrization}, we have
\begin{equation}
\label{eqn:conditional-phi-bound}
\phi_{\loss}(\H;m,P)
\leq 4 \phi_{\loss}(\H; \lceil (1/2)P(\region)m\rceil, P_{\region}) P(\region) + \frac{32 \maxloss}{m} + 2\D_{\loss}(\H;P) \sqrt{\frac{1}{m}}.
\end{equation}
Next, note that, for any $\sigma \geq \D_{\loss}(\H ; P)$, $\frac{\sigma}{\sqrt{P(\region)}} \geq \D_{\loss}(\H; P_{\region})$.
Also, if $\region = \region^{\prime} \times \Y$ for some $\region^{\prime} \supseteq \DISF(\H)$, then $\target_{P_{\region}} = \target_{P}$,
so that if $\target_{P} \in \H$, \eqref{eqn:ophi-2} implies
\begin{equation}
\label{eqn:conditional-phi-to-ophi}
\phi_{\loss}(\H; \lceil (1/2)P(\region)m\rceil, P_{\region}) \leq \spec{\phi}_{\loss}\left(\frac{\sigma}{\sqrt{P(\region)}}, \H; \lceil (1/2)P(\region)m \rceil, P_{\region}\right).
\end{equation}
Combining \eqref{eqn:conditional-phi-bound} with \eqref{eqn:conditional-phi-to-ophi},
we see that $\spec{\phi}_{\loss}^{\prime}$ satisfies the condition \eqref{eqn:ophi-2} of Definition~\ref{defn:abstract-ophi}.

Furthermore, by the fact that $\spec{\phi}_{\loss}$ satisfies \eqref{eqn:ophi-1} of Definition~\ref{defn:abstract-ophi},
combined with the monotonicity imposed by the infimum in the definition of $\spec{\phi}_{\loss}^{\prime}$,
it is easy to check that $\spec{\phi}_{\loss}^{\prime}$ also satisfies \eqref{eqn:ophi-1} of Definition~\ref{defn:abstract-ophi}.
In particular, note that any $\H^{\prime\prime} \subseteq \H^{\prime} \subseteq \Bracket{\F}$ and $\region^{\prime\prime} \subseteq \X$
have $\DISF(\H^{\prime\prime}_{\region^{\prime\prime}}) \subseteq \DISF(\H^{\prime})$, so that the range of $\region$ in the infimum is
never smaller for $\H=\H^{\prime\prime}_{\region^{\prime\prime}}$ relative to that for $\H=\H^{\prime}$.
\end{proof}

\begin{proof}[Proof of Corollary~\ref{cor:abstract-conditionals}]
Let $\spec{\phi}_{\loss}^{\prime}$ be as in Lemma~\ref{lem:ophi-conditional}, and define for any $m \in \nats$, $s \in [1,\infty)$, $\zeta \in [0,\infty]$, and $\H \subseteq \Bracket{\F}$,
\begin{align*}
& \spec{U}_{\loss}^{\prime}(\H,\zeta;\PXY,m,s)
\\ & = \tilde{K} \left( \spec{\phi}_{\loss}^{\prime}(\D_{\loss}(\Bracket{\H}(\zeta;\loss)),\H;m,\PXY) + \D_{\loss}(\Bracket{\H}(\zeta;\loss))\sqrt{\frac{s}{m}} + \frac{\maxloss s}{m}\right).
\end{align*}
That is, $\spec{U}_{\loss}^{\prime}$ is the function $\spec{U}_{\loss}$ that would result from using $\spec{\phi}_{\loss}^{\prime}$ in place of $\spec{\phi}_{\loss}$.
Let $\region = \DISF(\H)$, and suppose $\Px(\region) > 0$.
Then since $\DISF(\Bracket{\H}) = \DISF(\H)$ implies
\begin{align*}
\D_{\loss}(\Bracket{\H}(\zeta;\loss))
& = \D_{\loss}(\Bracket{\H}(\zeta;\loss);\PXYR{\region})\sqrt{\Px(\region)}
\\ & = \D_{\loss}(\Bracket{\H}(\zeta/\Px(\region);\loss,\PXYR{\region});\PXYR{\region})\sqrt{\Px(\region)},
\end{align*}
a little algebra reveals that for $m \geq 2 \Px(\region)^{-1}$,
\begin{equation}
\label{eqn:prime-to-conditional}
\spec{U}_{\loss}^{\prime}(\H,\zeta;\PXY,m,s)
\leq 33 \Px(\region) \spec{U}_{\loss}(\H, \zeta/\Px(\region);\PXYR{\region},\lceil (1/2)\Px(\region)m\rceil,s).
\end{equation}
In particular, for $j \geq j_{\loss}$, taking $\H = \F_{j}$, we have (from the definition of $\F_{j}$)
$\region = \DISF(\H) = \DIS(\H) = \region_{j}$, so that when $\Px(\region_{j}) > 0$,
any
\begin{equation*}
m \geq 2 \Px(\region_{j})^{-1} \spec{\SC}_{\loss}\left(\frac{2^{-j-1}}{33\Px(\region_{j})},\frac{2^{2-j}}{\Px(\region_{j})}; \F_{j}, \PXYR{\region_{j}}, \hat{\sfun}(2m)\right)
\end{equation*}
suffices to make the right side of \eqref{eqn:prime-to-conditional} (with $s = \hat{\sfun}(2m)$ and $\zeta = 2^{2-j}$) at most $2^{-j-1}$;
in particular, this means taking $u_{j}$ equal to $2m \lor u_{j-1} \lor 2u_{j-2}$ for any such $m$ (with $\log_{2}(m) \in \nats$) suffices to satisfy
\eqref{eqn:uj} (with the $\spec{\SC}_{\loss}$ in \eqref{eqn:uj} defined with respect to the $\spec{\phi}_{\loss}^{\prime}$ function);
monotonicity of $\zeta \mapsto \spec{\SC}_{\loss}\left(\zeta, \frac{2^{2-j}}{\Px(\region_{j})}; \F_{j}, \PXYR{\region_{j}}, \hat{\sfun}(2m)\right)$
implies \eqref{eqn:uj-conditional} is a sufficient condition for this.
In the special case where $\Px(\region_{j}) = 0$, $\spec{U}_{\loss}^{\prime}(\F_{j},2^{2-j};\PXY,m,s)$ $= \tilde{K} \frac{\maxloss s}{m}$,
so that taking $u_{j} \geq \tilde{K} \maxloss \hat{\sfun}(u_{j}) 2^{j+2} \lor u_{j-1} \lor 2u_{j-1}$ suffices to satisfy \eqref{eqn:uj} 
(again, with the $\spec{\SC}_{\loss}$ in \eqref{eqn:uj} defined in terms of $\spec{\phi}_{\loss}^{\prime}$).
Plugging these values into Theorem~\ref{thm:abstract-active} completes the proof.
\end{proof}

\section{Proofs of Results in Section~\ref{sec:explicit}}
\label{app:applications}

This appendix includes the remaining details of the proof of Theorem~\ref{thm:vc-subgraph-abstract},
to complete the derivations from Section~\ref{subsec:vc-derivation},
and also presents the remaining essential details for the proof of Theorem~\ref{thm:vc-subgraph-active-strong}.

\begin{proof}[Proof of Theorem~\ref{thm:vc-subgraph-abstract}]
Let $\tilde{j}_{\eps} = \lceil \log_{2}(1/\BJM(\eps)) \rceil$.
For $j_{\loss} \leq j \leq \tilde{j}_{\eps}$,
define $s_{j} = \Log\left(\frac{48 \left(2 + \tilde{j}_{\eps} - j\right)^{2}}{\conf}\right)$,
and let $u_{j} = 2^{\lceil \log_{2}(u_{j}^{\prime}) \rceil}$, where
\begin{equation}
\label{eqn:vc-subgraph-abstract-ujprime}
u_{j}^{\prime} = c^{\prime} \left( \tsybcb 2^{j(2-\tsybb)} + \maxloss 2^{j}\right) \left( \vc\left(\G_{\F}\right) \Log\left(\mixcap\maxloss\right) + s_{j}\right),
\end{equation}
for an appropriate universal constant $c^{\prime} \in [1,\infty)$.
A bit of calculus reveals that for $j_{\loss} + 2 \leq j \leq \tilde{j}_{\eps}$, 
$u_{j}^{\prime} \geq u_{j-1}^{\prime}$ and $u_{j}^{\prime} \geq 2 u_{j-2}^{\prime}$,
so that $u_{j} \geq u_{j-1}$ and $u_{j} \geq 2 u_{j-2}$ as well; this is also trivially satisfied for $j \in \{j_{\loss}, j_{\loss} + 1\}$
if we take $u_{j-2} = 1$ in these cases (as in Theorem~\ref{thm:abstract-active}).
Combining this fact with \eqref{eqn:vc-subgraph-mixed},  \eqref{eqn:spec-split-bound}, and \eqref{eqn:restinv-bound}, 
we find that, for an appropriate choice of the constant $c^{\prime}$, these $u_{j}$ satisfy \eqref{eqn:uj}
when we define $\hat{\sfun}$ such that, for every $j \in \{j_{\loss}, \ldots, \tilde{j}_{\eps}\}$,
$\forall m \in \{2u_{j-1},\ldots,u_{j}\}$ with $\log_{2}(m) \in \nats$, 
\begin{equation*}
\hat{\sfun}(m) = \Log\left( \frac{12 \log_{2}\left(4 u_{j} / m\right)^{2} \left( 2 + \tilde{j}_{\eps} - j\right)^{2}}{\conf}\right).
\end{equation*}
Additionally, let $s = \log_{2}(2 / \conf)$.

Next, note that, since $\BJM(\eps) \leq \Transform(\eps)$ and $u_{j}$ is nondecreasing in $j$, 
\begin{align*}
u_{j_{\eps}} 
\leq u_{\tilde{j}_{\eps}} 
\leq 26 c^{\prime} \left( \frac{\tsybcb}{\BJM(\eps)^{2-\tsybb}} + \frac{\maxloss}{\BJM(\eps)}\right) \left( \vc\left(\G_{\F}\right) \Log\left(\mixcap\maxloss\right) + \Log(1/\conf)\right),
\end{align*}
so that, for any $c \geq 26 c^{\prime}$, we have $u \geq u_{i_{\eps}}$, as required by Theorem~\ref{thm:abstract-active}.

For $\region_{j}$ as in Theorem~\ref{thm:abstract-active}, note that by Condition~\ref{con:tsybakov-01} and the definition of $\dc$,
\begin{align*}
\Px\left(\region_{j}\right)
& = \Px\left( \DIS\left( \F\left( \InvTransform\left(2^{2-j}\right); \zo\right) \right) \right)
\leq \Px\left( \DIS\left( \Ball\left( \target, \tsybca \InvTransform\left(2^{2-j}\right)^{\tsyba}\right) \right) \right)
\\ & \leq \dc \max\left\{ \tsybca \InvTransform\left(2^{2-j}\right)^{\tsyba}, \tsybca \eps^{\tsyba}\right\}
\leq \dc \max\left\{ \tsybca \BJM^{-1}\left(2^{2-j}\right)^{\tsyba}, \tsybca \eps^{\tsyba}\right\}.
\end{align*}
Because $\BJM$ is strictly increasing on $(0,1)$,
for $j \leq \tilde{j}_{\eps}$,
$\BJM^{-1}\left(2^{2-j}\right)$ $\geq \eps$, so that this last expression is equal to
$\dc \tsybca \BJM^{-1}\left(2^{2-j}\right)^{\tsyba}$.
This implies
\begin{align}
&\sum_{j=j_{\loss}}^{j_{\eps}} \Px\left(\region_{j}\right) u_{j}
\leq \sum_{j=j_{\loss}}^{\tilde{j}_{\eps}} \Px\left(\region_{j}\right) u_{j} \notag
\\ & \lesssim \sum_{j=j_{\loss}}^{\tilde{j}_{\eps}} \tsybca \dc \BJM^{-1}\left( 2^{2-j} \right)^{\tsyba} \left( \tsybcb 2^{j(2-\tsybb)} + \maxloss 2^{j}\right) \left( A_{1} + \Log\left(2 + \tilde{j}_{\eps} - j\right)\right). \label{eqn:vc-subgraph-PUjuj-sum-bound}
\end{align}
We can change the order of summation in the above expression by letting 
$i = \tilde{j}_{\eps} - j$ and summing from $0$ to $N=j_{\eps} - j_{\loss}$.
In particular, since $2^{\tilde{j}_{\eps}} \leq 2/\BJM(\eps)$,
\eqref{eqn:vc-subgraph-PUjuj-sum-bound} is at most
\begin{equation}
\label{eqn:vc-subgraph-PUjuj-sum-bound-i}
\sum_{i=0}^{N} \tsybca \dc \BJM^{-1}\left(2^{2-\tilde{j}_{\eps}}2^{i}\right)^{\tsyba} \left(\frac{4 \tsybcb 2^{i (\tsybb - 2)}}{\BJM(\eps)^{2-\tsybb}} + \frac{2 \maxloss 2^{- i}}{\BJM(\eps)}\right) \left( A_{1} + \Log(i + 2)\right).
\end{equation}

Since $x \mapsto \BJM^{-1}(x)/x$ is nonincreasing on $(0,\infty)$, 
we have $\BJM^{-1}\left( 2^{2 -\tilde{j}_{\eps}} 2^{i}\right) \leq 2^{i+2} \BJM^{-1}\left(2^{- \tilde{j}_{\eps}} \right)$, 
and since $\BJM^{-1}$ is increasing, this latter expression is at most\break $2^{i+2} \BJM^{-1}\left(\BJM(\eps)\right) = 2^{i+2} \eps$.
Thus, \eqref{eqn:vc-subgraph-PUjuj-sum-bound-i} is at most
\begin{equation}
\label{eqn:vc-subgraph-pre-split}
16 \tsybca \dc \eps^{\tsyba} \sum_{i=0}^{N} \left( \frac{\tsybcb 2^{i (\tsyba+\tsybb - 2)}}{\BJM(\eps)^{2-\tsybb}} + \frac{\maxloss 2^{i(\tsyba-1)}}{\BJM(\eps)}\right) \left( A_{1} + \Log(i+2)\right).
\end{equation}
In general, $\Log(i+2) \leq \Log(N+2)$, so that
$\sum_{i=0}^{N} 2^{i (\tsyba + \tsybb - 2)} \left( A_{1} + \Log(i+2)\right)$
$\leq (A_{1} + \Log(N+2)) (N+1)$
and
$\sum_{i=0}^{N} 2^{i (\tsyba - 1)} \left( A_{1} + \Log(i+2)\right)
\leq$\break $(A_{1} + \Log(N+2)) (N+1)$.
When $\tsyba + \tsybb < 2$ holds, we also have $\sum_{i=0}^{N} 2^{i (\tsyba + \tsybb - 2)} \leq$ $\sum_{i=0}^{\infty} 2^{i (\tsyba + \tsybb - 2)}$ $= \frac{1}{1 - 2^{(\tsyba+\tsybb-2)}}$
and furthermore $\sum_{i=0}^{N} 2^{i (\tsyba+\tsybb-2)} \Log(i+2) \leq$\break $\sum_{i=0}^{\infty} 2^{i (\tsyba+\tsybb-2)} \Log(i+2) \leq \frac{2}{1-2^{(\tsyba+\tsybb-2)}} \Log\left(\frac{1}{1-2^{(\tsyba+\tsybb-2)}}\right)$.
Similarly, if $\tsyba < 1$, $\sum_{i=0}^{N} 2^{i (\tsyba-1)} \leq \sum_{i=0}^{\infty} 2^{i (\tsyba-1)}$ $= \frac{1}{1-2^{(\tsyba-1)}}$ and likewise
$\sum_{i=0}^{N} 2^{i (\tsyba-1)} \Log(i+2) \leq \sum_{i=0}^{\infty} 2^{i (\tsyba-1)} \Log(i+2) \leq \frac{2}{1-2^{(\tsyba-1)}} \Log\left(\frac{1}{1-2^{(\tsyba-1)}}\right)$.
By combining these observations (along with a convention that $\frac{1}{1-2^{(\tsyba-1)}}=\infty$ when $\tsyba=1$,
and $\frac{1}{1-2^{(\tsyba+\tsybb-2)}}=\infty$ when $\tsyba=\tsybb=1$), and noting that $\frac{1}{1-2^{(\tsyba+\tsybb-2)}} / \min\left\{ \frac{1}{1-2^{(\tsyba-1)}}, \frac{1}{1-2^{(\tsybb-1)}}\right\} \in [1/2,1]$,
we find that \eqref{eqn:vc-subgraph-pre-split} is
\begin{equation*}
\lesssim \tsybca \dc \eps^{\tsyba} \left( \frac{\tsybcb (A_{1} + \Log(B_{1})) B_{1}}{\BJM(\eps)^{2-\tsybb}} + \frac{\maxloss (A_{1} + \Log(C_{1})) C_{1}}{\BJM(\eps)}\right).
\end{equation*}
Thus, for an appropriately large numerical constant $c$,
any $n$ satisfying \eqref{eqn:vc-subgraph-abstract-n} has
\begin{equation*}
n \geq s + 2 e \sum_{j=j_{\loss}}^{\tilde{j}_{\eps}} \Px(\region_{j}) u_{j},
\end{equation*}
as required by Theorem~\ref{thm:abstract-active}.

Finally, we need to show the success probability from Theorem~\ref{thm:abstract-active} is at least $1-\conf$, for $\hat{\sfun}$ and $s$ as above.
Toward this end, note that

\begin{align*}
\sum_{i=1}^{\log_{2}(u_{j_{\eps}})} 6 e^{-\hat{\sfun}(2^{i})}
& \leq \sum_{j=j_{\loss}}^{\tilde{j}_{\eps}} \sum_{i = \log_{2}(u_{j-1})+1}^{\log_{2}(u_{j})} \frac{\conf}{2 \left(2 + \log_{2}(u_{j}) - i\right)^{2} \left(2 + \tilde{j}_{\eps} - j\right)^{2}}
\\ & = \sum_{j=j_{\loss}}^{\tilde{j}_{\eps}} \sum_{t=0}^{\log_{2}(u_{j}/u_{j-1})-1} \frac{\conf}{2 (2+t)^{2} \left(2 + \tilde{j}_{\eps} - j\right)^{2}}
\\ & < \sum_{j=j_{\loss}}^{\tilde{j}_{\eps}} \frac{\conf}{2 \left( 2 + \tilde{j}_{\eps} - j\right)^{2}}
< \sum_{t = 0}^{\infty} \frac{\conf}{2 (2+t)^{2}}
< \conf / 2.
\end{align*}
Noting that $2^{-s} = \conf / 2$, we find that indeed
\begin{equation*}
1 - 2^{-s} - \sum_{i=1}^{\log_{2}(u_{j_{\eps}})} 6 e^{-\hat{\sfun}(2^{i})} \geq 1 - \conf.
\end{equation*}
Therefore, Theorem~\ref{thm:abstract-active} implies the stated result.
\end{proof}

We note that the values $\hat{\sfun}(m)$ used in the proof of Theorem~\ref{thm:vc-subgraph-abstract}
have a direct dependence on the parameters $\tsybcb$, $\tsybb$, $\tsybca$, $\tsyba$, and $\mixcap$. 
Such a dependence may be undesirable for many applications, where information about these values is not available.
However, one can easily follow this same proof, taking $\hat{\sfun}(m) = \Log\left(\frac{12 \log_{2}(2m)^{2}}{\conf} \right)$ instead,
which only leads to an increase by a $\log\log$ factor: specifically, replacing the factor of $A_{1}$ in \eqref{eqn:vc-subgraph-abstract-u},
and the factors $(A_{1}+\Log(B_{1}))$ and $(A_{1}+\Log(C_{1}))$ in \eqref{eqn:vc-subgraph-abstract-n},
with a factor of $(A_{1} + \Log(\Log(\maxloss/\BJM(\eps))))$. 
It is not clear whether it is always possible to achieve the slightly tighter
result of Theorem~\ref{thm:vc-subgraph-abstract} without having direct access to the values $\tsybcb$, $\tsybb$, $\tsybca$, $\tsyba$, and $\mixcap$ in the algorithm.

\begin{proof}[Proof Sketch of Theorem~\ref{thm:vc-subgraph-active-strong}]
The proof follows analogously to the proof of Theorem~\ref{thm:vc-subgraph-abstract}, with the exception that now,
for each integer $j$ with $j_{\loss} \leq j \leq \tilde{j}_{\eps}$,
we replace the definition of $u_{j}^{\prime}$ from \eqref{eqn:vc-subgraph-abstract-ujprime} with the following definition.
Letting 
\begin{equation*}
c_{j} = \vc(\G_{\F}) \Log\left( \left(\maxloss^{2} / \tsybcb\right)\left(\tsybca \dc 2^{j} \BJM^{-1}(2^{2-j})^{\tsyba}\right)^{\tsybb}\right),
\end{equation*}
define
\begin{equation*}
u_{j}^{\prime} = c^{\prime} \left( \tsybcb 2^{j (2-\tsybb)} \left( \tsybca \dc \BJM^{-1}(2^{2-j})^{\tsyba}\right)^{1-\tsybb} + \maxloss 2^{j}\right)\left( c_{j} + s_{j}\right),
\end{equation*}
where $c^{\prime} \in [1,\infty)$ is an appropriate universal constant, and $s_{j}$ is as in the proof of Theorem~\ref{thm:vc-subgraph-abstract}.
With this substitution in place, the values $u_{j}$ and $s$, and function $\hat{\sfun}$, are then defined as in the proof of Theorem~\ref{thm:vc-subgraph-abstract}.
Since $x \mapsto x \BJM^{-1}(1/x)$ is nondecreasing, a bit of calculus reveals $u_{j} \geq u_{j-1}$ and $u_{j} \geq 2 u_{j-2}$. 
Combined with \eqref{eqn:vc-subgraph-strong-conditional}, \eqref{eqn:restinv-bound}, \eqref{eqn:spec-split-bound}, and Lemma~\ref{lem:bjm-strong-convexity},
this implies we can choose the constant $c^{\prime}$ so that these $u_{j}$ satisfy \eqref{eqn:uj-conditional}.
By an identical argument to that used in Theorem~\ref{thm:vc-subgraph-abstract}, we have
\begin{equation*}
1 - 2^{-s} - \sum_{i=1}^{\log_{2}(u_{j_{\eps}})} 6 e^{-\hat{\sfun}(2^{i})} \geq 1 - \conf.
\end{equation*}
It remains only to show that any values of $u$ and $n$ satisfying \eqref{eqn:vc-subgraph-active-strong-u} and \eqref{eqn:vc-subgraph-active-strong-n}, respectively,
necessarily also satisfy the respective conditions for $u$ and $n$ in Corollary~\ref{cor:abstract-conditionals}.

Toward this end, note that since 
$x \mapsto x \BJM^{-1}(1 / x)$ is nondecreasing on $(0,\infty)$, we have that
\begin{align*}
u_{j_{\eps}} \leq u_{\tilde{j}_{\eps}} \lesssim \left( \frac{\tsybcb \left(\tsybca \dc \eps^{\tsyba}\right)^{1-\tsybb}}{\BJM(\eps)^{2-\tsybb}} + \frac{\maxloss}{\BJM(\eps)}\right) A_{2}.
\end{align*}
Thus, for an appropriate choice of $c$, any $u$ satisfying \eqref{eqn:vc-subgraph-active-strong-u} has 
$u \geq u_{j_{\eps}}$, as required by Corollary~\ref{cor:abstract-conditionals}.

Finally, note that for $\region_{j}$ as in Theorem~\ref{thm:abstract-active}, and $i_{j} = \tilde{j}_{\eps} - j$,
\begin{multline*}
\sum_{j=j_{\loss}}^{j_{\eps}} \Px(\region_{j}) u_{j}
\leq \sum_{j=j_{\loss}}^{j_{\eps}} \tsybca \dc \BJM^{-1}(2^{2-j})^{\tsyba} u_{j}
\\  \lesssim
\sum_{j=j_{\loss}}^{\tilde{j}_{\eps}}
\tsybcb \left( \tsybca \dc 2^{j} \BJM^{-1}(2^{2-j})^{\tsyba}\right)^{2-\tsybb} \left( A_{2} + \Log\left( i_{j} + 2 \right)\right)
\\ 
+ \sum_{j=j_{\loss}}^{\tilde{j}_{\eps}}
\maxloss \tsybca \dc 2^{j} \BJM^{-1}(2^{2-j})^{\tsyba}\left( A_{2} + \Log\left( i_{j} + 2 \right)\right).
\end{multline*}

By changing the order of summation, now summing over values of $i_{j}$ from $0$ to $N = \tilde{j}_{\eps} - j_{\loss} \leq \log_{2}(4 \maxloss / \BJM(\eps))$,
and noting $2^{\tilde{j}_{\eps}} \leq 2/\BJM(\eps)$, and $\BJM^{-1}(2^{-\tilde{j}_{\eps}} 2^{2+i}) \leq 2^{2+i} \eps$ for $i \geq 0$,
this last expression is
\begin{align}
\label{eqn:vc-subgraph-active-strong-summation}
\lesssim & \sum_{i=0}^{N}
\tsybcb \left( \frac{\tsybca \dc 2^{i (\tsyba-1)} \eps^{\tsyba}}{\BJM(\eps)}\right)^{2-\tsybb} \left( A_{2} + \Log\left( i + 2 \right)\right)
\\ & + \sum_{i=0}^{N}\frac{\maxloss \tsybca \dc 2^{i (\tsyba-1)} \eps^{\tsyba}}{\BJM(\eps)}\left( A_{2} + \Log\left( i + 2 \right)\right). \notag
\end{align}
Considering these sums separately, we have
$\sum_{i=0}^{N} 2^{i (\tsyba-1)(2-\tsybb)} (A_{2} + \Log(i+2)) \leq (N+1) (A_{2} + \Log(N+2))$
and $\sum_{i=0}^{N} 2^{i (\tsyba-1)}(A_{2} + \Log(i+2)) \leq$\break $(N\!+\!1) (A_{2} \!+\! \Log(N\!+\!2))$.
When $\tsyba \!<\! 1$, we have
$\sum_{i=0}^{N} 2^{i (\tsyba-1)(2-\tsybb)} (A_{2} \!+\! \Log(i\!+\!2))$\break 
$\leq \sum_{i=0}^{\infty} 2^{i (\tsyba-1)(2-\tsybb)} (A_{2} + \Log(i+2)) \leq \frac{2}{1-2^{(\tsyba-1)(2-\tsybb)}}\Log\!\left(\frac{1}{1-2^{(\tsyba-1)(2-\tsybb)}}\right) +$\break $\frac{A_{2}}{1-2^{(\tsyba-1)(2-\tsybb)}}$,
and 
$\sum_{i=0}^{N} 2^{i (\tsyba-1)} (A_{2} + \Log(i+2)) \leq \frac{2}{1-2^{(\tsyba-1)}} \Log\left(\frac{1}{1-2^{(\tsyba-1)}}\right) + \frac{A_{2}}{1-2^{(\tsyba-1)}}$.
Thus, noting that $\frac{1}{1-2^{(\tsyba-1)(2-\tsybb)}} / \frac{1}{1-2^{(\tsyba-1)}} \in [1/2,1]$, 
we generally have
$\sum_{i=0}^{N} 2^{i (\tsyba-1)(2-\tsybb)}(A_{2} + \Log(i+2)) \lesssim C_{1} (A_{2} + \Log(C_{1}))$
and\break
$\sum_{i=0}^{N} 2^{i (\tsyba-1)} (A_{2} + \Log(i+2)) \lesssim C_{1} (A_{2} + \Log(C_{1}))$.
Plugging this into \eqref{eqn:vc-subgraph-active-strong-summation}, we find that for an appropriately large numerical constant $c$,
any $n$ satisfying \eqref{eqn:vc-subgraph-active-strong-n} has $n \geq \sum_{j=j_{\loss}}^{j_{\eps}} \Px(\region_{j}) u_{j}$, as required by Corollary~\ref{cor:abstract-conditionals}.
\end{proof}

We note that, as in Theorem~\ref{thm:vc-subgraph-abstract}, the values $\hat{\sfun}$ used to obtain 
Theorem~\ref{thm:vc-subgraph-active-strong} have a direct dependence on
certain values, which are typically not directly accessible in practice: in this case, $\tsybca$, $\tsyba$, and $\dc$.
However, as was the case for Theorem~\ref{thm:vc-subgraph-abstract}, we can obtain only slightly worse results by
instead taking $\hat{\sfun}(m) = \Log\left(\frac{12 \log_{2}(2m)^{2}}{\conf}\right)$,
which again only leads to an increase by a $\log\log$ factor: replacing the factor of $A_{2}$ in \eqref{eqn:vc-subgraph-active-strong-u},
and the factor of $(A_{2}+\Log(C_{1}))$ in \eqref{eqn:vc-subgraph-active-strong-n}, with a factor of $(A_{2} + \Log(\Log(\maxloss/\BJM(\eps))))$. 
As before, it is not clear whether the slightly tighter result of Theorem~\ref{thm:vc-subgraph-active-strong} is always available,
without requiring direct dependence on these quantities.

\subsection{Derivations for Section~\ref{subsec:example}}
\label{subsec:example-derivations}

For completeness, we include here derivations of quantities appearing in the 
example given in Section~\ref{subsec:example}.
We begin with the claim that, for any $\omega \in (0,1/2]$, 
\eqref{eqn:bracketing-entropy-bound} is satisfied in Condition~\ref{con:entropy}
with the values $\entc = \frac{7}{\omega}$ and $\entrho = \frac{1}{3}+\omega$.
Specifically, for a given $\eps > 0$, let $i_{\eps} = \left\lceil \frac{3}{\eps^{2/3}} \right\rceil$,
and let $\G_{\eps}$ be the set of functions $g$ in $\FunctionsXY$ with 
$g(x,y) \in \{ j \eps / \sqrt{2} : j \in \{0,\ldots, \lceil 4\sqrt{2} / \eps \rceil-1 \} \}$
for each $x \in \{ x_{i} : 1 \leq i \leq i_{\eps}\}$ and $y \in \Y$, 
and $g(x,y) = 0$ for every $x \in \X \setminus \{x_{i} : 1 \leq i \leq i_{\eps}\}$ and $y \in \Y$.
For each $g \in \G_{\eps}$, let $g^{\prime}$ be the function in $\FunctionsXY$
with $g^{\prime}(x,y) = g(x,y) + \eps/\sqrt{2}$ for each $x \in \{ x_{i} : 1 \leq i \leq i_{\eps} \}$ and $y \in \Y$,
and $g^{\prime}(x,y) = 4$ for each $x \in \X \setminus \{x_{i} : 1 \leq i \leq i_{\eps}\}$ and $y \in \Y$.
Note that $\bigcup_{g \in \G_{\eps}} [ g, g^{\prime} ]$ contains all functions $g$ in $\FunctionsXY$
having $0 \leq g(x,y) \leq 4$ for all $x \in \X$ and $y \in \Y$; in particular, this implies it contains $\G_{\F}$.
Furthermore, for each $g \in \G_{\eps}$, 
$\| g - g^{\prime} \|_{\PXY}^{2}  = \sum_{i = 1}^{i_{\eps}} \frac{\eps^{2}}{2} \Px(\{x_{i}\}) + \sum_{i=i_{\eps}+1}^{\infty} 16 \Px(\{x_{i}\}) 
\leq \frac{\eps^{2}}{2} + \frac{16 \cdot 90}{\pi^{4}} \int_{i_{\eps}}^{\infty} \frac{1}{x^{4}} {\rm d}x
= \frac{\eps^{2}}{2} + \frac{16 \cdot 30}{\pi^{4}} \frac{1}{i_{\eps}^{3}}
\leq \frac{\eps^{2}}{2} + \frac{16 \cdot 30}{27 \pi^{4}} \eps^{2}
< \eps^{2}$,
so that $[g,g^{\prime}]$ is an $\eps$-bracket under $L_{2}(\PXY)$.
Therefore, $\covering_{[]}\left( \eps, \G_{\F}, L_{2}(\PXY)\right) \leq |\G_{\eps}| = \lceil 4\sqrt{2} / \eps \rceil^{2 i_{\eps}}$,
so that (taking $F = \maxloss = 4$, constant, in Condition~\ref{con:entropy})
$\ln \covering_{[]}\left( 4\eps, \G_{\F}, L_{2}(\PXY)\right) \leq 2 \left\lceil \frac{3}{(4\eps)^{2/3}} \right\rceil \ln\left( \left\lceil \frac{\sqrt{2}}{\eps} \right\rceil \right)$.
Since $\ln(x) \leq t x^{1/t}$ for any $x,t \geq 1$, this is at most
$\frac{7}{\omega} \eps^{-2\left( \frac{1}{3}+\omega\right)}$ when $\eps \in (0,1)$, for any value $\omega \in (0,1/2]$.
This is trivially also an upper bound on $\ln \covering_{[]}\left( 4\eps, \G_{\F}, L_{2}(\PXY)\right)$ for all $\eps \geq 1$ (since $\covering_{[]}\left( 4\eps, \G_{\F}, L_{2}(\PXY)\right) = 1$ in that case).
Thus, \eqref{eqn:bracketing-entropy-bound} is satisfied with $\entc = \frac{7}{\omega}$ and $\entrho = \frac{1}{3}+\omega$, for any choice of $\omega \in (0,1/2]$, as claimed.

Next, we present a proof of the claimed $\Omega(\eps^{-4/3})$ lower bound on the 
sample size required to obtain an $\eps$ bound on the minimax expected excess error rate of passive learning methods
in the example scenario.
We approach this with the classic technique of Assouad (see e.g., \citep*{tsybakov:09}).
Specifically, fix any $\eps \in (0,(1-2\bound)/64)$, 
and fix a sample size $m \in \nats$ with $m \leq  2^{-13} (1-2\bound)^{1/3} \eps^{-4/3}$. 
Let $j_{0} = \left\lfloor \left(\frac{72}{107 \pi^{4}} \right)^{1/4} \left(\frac{1-2\bound}{\eps}\right)^{1/3}\right\rfloor$,
$j_{1} = \left\lfloor \frac{1}{2^{4/3}} \left(\frac{1-2\bound}{\eps}\right)^{1/3} \right\rfloor$,  
and $k = j_{1}-j_{0}+1$.
In particular, a simple calculation reveals 
$k \geq \frac{27}{250} \left(\frac{1-2\bound}{\eps}\right)^{1/3}$.
Now for any binary vector $v = (v_{1},\ldots,v_{k}) \in \{0,1\}^{k}$, 
define $P_{v}$ as the probability measure on $\X \times \Y$ with marginal $\Px$ on $\X$ (as specified in the construction),
$\eta(x_{i};P_{v}) = 1$ for $i \in \nats \setminus \{j_{0},\ldots,j_{1}\}$, 
and $\eta(x_{i};P_{v}) = \bound + (1-2\bound) v_{i - j_{0} + 1}$ for $i \in \{j_{0},\ldots,j_{1}\}$.
Then note that for any $v,v^{\prime} \in \{0,1\}^{k}$ with $\|v-v^{\prime}\|_{1} = 1$, 
the total variation distance $\| P_{v} - P_{v^{\prime}} \|$ between the corresponding distributions
is at most $\frac{90}{\pi^{4} j_{0}^{4}} (1-2\bound)$.
This further implies $\| P_{v}^{m} - P_{v^{\prime}}^{m} \| \leq$\break $m \| P_{v} - P_{v^{\prime}} \| 
\leq 2^{-13} \frac{90}{\pi^{4} j_{0}^{4}} \left(\frac{1-2\bound}{\eps}\right)^{4/3} < \frac{1}{2}$.
Therefore, Theorem 2.12(ii) of \citep*{tsybakov:09} implies that,
for any estimator $\hat{v} : (\X \times \Y)^{m} \to \{0,1\}^{k}$ (possibly randomized),
there exists a choice $v \in \{0,1\}^{k}$ such that, defining $\PXY = P_{v}$, 
we have $\E\left[ \left\|\hat{v}(\Data_{m}) - v\right\|_{1} \right] \geq \frac{k}{4} \geq \frac{27}{1000} \left(\frac{1-2\bound}{\eps}\right)^{1/3}$.
In particular, for any passive learning algorithm $\alg :$\break $(\X \times \Y)^{m} \to \Functions$,
we can define a vector $\hat{v}$ based on the returned function $\hat{f}$ from $\alg$
by letting $\hat{v}_{i} = (\sign(\hat{f}(x_{i+j_{0}-1}))+1)/2$ for each $i \in \{1,\ldots,k\}$.
Then we note that for any $v \in \{0,1\}^{k}$, if $\PXY = P_{v}$, then $\er(\hat{f}) - \er(\target) \geq \frac{90}{\pi^{4} j_{1}^{4}} (1-2\bound) \| \hat{v} - v \|_{1}$.
Thus, there exists a choice of $v \in \{0,1\}^{k}$ such that, defining $\PXY = P_{v}$, 
we have that for $\hat{f} = \alg(\Data_{m})$, 
$\E\left[ \er(\hat{f}) - \er(\target) \right] \geq \frac{90}{\pi^{4} j_{1}^{4}} (1-2\bound) \cdot \frac{27}{1000} \left(\frac{1-2\bound}{\eps}\right)^{1/3} > \eps$.
Thus, since these $P_{v}$ distributions satisfy the description of the construction in Section~\ref{subsec:example},
we see that to guarantee expected excess error rate at most $\eps$ for all $\PXY$ fitting the description in the construction, 
any passive learning method would require the sample size $m$ for its input labeled data set to be greater than $2^{-13} (1-2\bound)^{1/3} \eps^{-4/3} = \Omega(\eps^{-4/3})$, as claimed.
In particular, this agrees with the dependence on $\eps$ derived for $\ERM_{\loss}$ in Section~\ref{subsec:example} (up to a logarithmic factor).
%
In contrast, the analysis of \ACAL~in Section~\ref{subsec:example} reveals that (by choosing $\conf = \eps/2$), \ACAL~can achieve 
$\E[\er(\hat{h})-\er(\target)] \leq \eps$ for all such $\PXY$ with a number of label requests $n$
having only $O(\eps^{-7/12} \Log(1/\eps))$ dependence on $\eps$, a significant decrease compared to the $\Omega(\eps^{-4/3})$ lower bound
we have just established for all passive learning methods.

\subsection{Derivations for Section~\ref{subsec:linsep}}
\label{subsec:linsep-derivations}

For completeness, we include here a derivation of the 
parameters $\tsybca$ and $\tsyba$ for which the
distributions $\PXY$ in the example in Section~\ref{subsec:linsep}
satisfy Condition~\ref{con:tsybakov-01}.  Specifically, 
as in Section~\ref{subsec:linsep}, let $\loss$ be the 
quadratic loss, fix an integer $k \geq 5$,
suppose $\Px$ is uniform on $\{ x \in \reals^{k} : \|x\| = 1 \}$,
and suppose $\PXY$ is such that $\target(x) = w^{*} \cdot x$ for some $w^{*} \in \reals^{k}$ with $\|w^{*}\| = 1$.
%
In particular, for this choice of $\loss$, this implies 
$\eta(x) = (w^{*} \cdot x + 1)/2$.
For any $f \in \Functions$, 
$\er(f) - \er(\target) 
= \E\left[ |1-2\eta(X)| \big| X \in \DIS(\{f,\target\}) \right] \dist(f,\target)$, for $X \sim \Px$.
Therefore, among functions $f \in \Functions$ with a given value $p$ of $\dist(f,\target)$,
the functions with minimal $\er(f)-\er(\target)$ are those that minimize 
$\E\!\left[ | 2\eta(X) \!-\! 1 | \big| X \!\in\! \DIS(\{f,\target\}) \right]$
subject to $\Px(\DIS(\{f,\target\})) = p$; since $|2 \eta(x) - 1| = |w^{*} \cdot x|$
is increasing in $|w^{*} \cdot x|$ and $t \mapsto \Px( x : |w^{*} \cdot x| \leq t )$ is continuous, 
any $f \in \Functions$ of minimal $\er(f) - \er(\target)$
subject to $\dist(f,\target) = p$ has $\DIS(\{f,\target\}) = \{ x : |w^{*} \cdot x| \leq \gamma_{p} \}$
(up to probability zero differences) for some $\gamma_{p} \in [0,1]$ 
chosen so that $\Px(x : |w^{*} \cdot x| \leq \gamma_{p}) = p$;
in particular, the minimum value of $\er(f) - \er(\target)$ among such functions $f$ is
$\E\left[ |w^{*} \cdot X| \ind[ |w^{*} \cdot X| \leq \gamma_{p} ] \right]$.
Fix such a function $f_{p}$ with $\DIS(\{f_{p},\target\}) =$\break $\{ x : |w^{*} \cdot x| \leq \gamma_{p} \}$.

For $X \sim \Px$, one can show that the $[0,1]$-valued random variable 
$|w^{*} \cdot X|$ has density function $g(t) = \frac{2\Gamma(k/2)}{\sqrt{\pi}\Gamma((k-1)/2)} (1-t^{2})^{\frac{k-3}{2}}$,
where $\Gamma$ is the usual gamma function 
(see \citep*{li:11} for a derivation of the CDF, from which this $g$ can be derived). 
Thus, 
\begin{align*}
\E\left[ |w^{*} \cdot X| \ind[ |w^{*} \cdot X| \leq \gamma_{p} ] \right]
&= \int_{0}^{\gamma_{p}} \frac{2 \Gamma(k/2)}{\sqrt{\pi}\Gamma((k-1)/2)} t (1-t^{2})^{\frac{k-3}{2}} {\rm d}t 
\\ & = \frac{2 \Gamma(k/2)}{\sqrt{\pi}\Gamma((k-1)/2)} \frac{1}{k-1} \left( 1 - (1-\gamma_{p}^{2})^{\frac{k-1}{2}} \right).
\end{align*}
When $\gamma_{p} \leq \frac{1}{\sqrt{k-3}}$, some basic calculus reveals $1 - (1-\gamma_{p}^{2})^{\frac{k-1}{2}} \geq \gamma_{p}^{2} \frac{k-1}{2e}$. 
Since one can also verify that $\frac{2 \Gamma(k/2)}{\sqrt{\pi}\Gamma((k-1)/2)} \geq \sqrt{k/3}$, we have that
if $p$ is such that $\gamma_{p} \leq \frac{1}{\sqrt{k-3}}$, then 
$\er(f_{p}) - \er(\target) \geq \frac{\sqrt{k} \gamma_{p}^{2}}{2 e \sqrt{3}}$.
It also holds that $\dist(f_{p},\target) = \Px( x : |w^{*} \cdot x| \leq \gamma_{p} ) \leq \sqrt{k} \gamma_{p}$
\citep*[see e.g.,][]{hanneke:07b}.  Together, we have that if $\gamma_{p} \leq \frac{1}{\sqrt{k-3}}$, 
then 
$\dist(f_{p},\target) \leq \sqrt{k} \gamma_{p} = \sqrt{2 e} (3k)^{1/4} \left( \frac{\sqrt{k} \gamma_{p}^{2}}{2 e \sqrt{3}} \right)^{1/2}
\leq \sqrt{2 e} (3k)^{1/4} \left( \er(f_{p}) - \er(\target) \right)^{1/2}$.

Noting that $\gamma_{p}$ is continuous in $p$, with $\gamma_{0} = 0$ and $\gamma_{1} = 1$,
the intermediate value theorem implies $\exists p_{*} \in [0,1]$ with $\gamma_{p_{*}} = \frac{1}{\sqrt{k-3}}$.
Since
$\sqrt{2 e} (3k)^{1/4} \left( \frac{\sqrt{k} \gamma_{p_{*}}^{2}}{2 e \sqrt{3}} \right)^{1/2}
= \sqrt{\frac{k}{k-3}} > 1$,
we have $\sqrt{2 e} (3k)^{1/4} \left( \er(f_{p_{*}}) - \er(\target) \right)^{1/2} > 1$.
Now for any $p$ with $\gamma_{p} > \frac{1}{\sqrt{k-3}}$, we have $\DIS(\{f_{p},\target\}) \supseteq \DIS(\{f_{p_{*}},\target\})$,
which implies $\er(f_{p}) \geq \er(f_{p_{*}})$.
Therefore, $\sqrt{2 e} (3k)^{1/4} \left( \er(f_{p}) - \er(\target) \right)^{1/2} > 1 \geq \dist(f_{p},\target)$.
%
Thus, we have established that $\dist(f_{p},\target) \leq \sqrt{2 e} (3k)^{1/4} \left( \er(f_{p}) - \er(\target) \right)^{1/2}$
for \emph{every} $p \in [0,1]$.  Since, for every $p \in [0,1]$, 
$f_{p}$ was chosen to minimize $\er(f_{p})-\er(\target)$ subject to $\dist(f_{p},\target) = p$,
we have $\dist(f,\target) \leq \sqrt{2 e} (3k)^{1/4} \left( \er(f) - \er(\target) \right)^{1/2}$
for \emph{every} $f \in \Functions$: that is, that Condition~\ref{con:tsybakov-01} 
holds with $\tsybca = \sqrt{2 e} (3k)^{1/4}$ and $\tsyba = 1/2$.

\section{Remarks on VC Major and VC Hull Classes}
\label{subsec:vc-major}

In addition to VC Subgraph classes, and scenarios satisfying general entropy conditions,
another widely-studied family of function classes includes
\emph{VC major} classes.  Specifically, we say $\G$ is a VC major class with index $\dim$ if
$\dim = \vc(\{ \{z : g(z)\geq t\} : g \in \G, t \in \reals\}) < \infty$.
We can derive results for VC major classes, analogously to the above, as follows.
For brevity, we leave many of the details as an exercise for the reader.
For any VC major class $\G \subseteq \FunctionsXY$ with index $\dim$, by reasoning similar to that of \citet*{gine:06},
one can show that if $\Env = \maxloss \ind_{\region} \geq \Env(\G)$ for some measurable $\region \subseteq \X\times \Y$,
then for any distribution $P$ and $\eps > 0$,
\begin{equation*}
\ln \covering\left(\eps \|\Env\|_{P}, \G, L_{2}(P) \right) \lesssim \frac{\dim}{\eps} \log\left(\frac{\maxloss}{\eps}\right)\log\left(\frac{1}{\eps}\right).
\end{equation*}
This implies that for $\F$ a VC major class, and $\loss$ classification-calibrated and either nonincreasing or Lipschitz on $[-\sup_{h\in\F}\sup_{x\in\X}|h(x)|,\sup_{h\in\F}\sup_{x\in\X}|h(x)|]$,
if $\target \in \F$ and $\PXY$ satisfies Condition~\ref{con:tsybakov-01} and Condition~\ref{con:tsybakov-sur},
then the conditions of Theorem~\ref{thm:abstract-active}
can be satisfied with the probability bound being at least $1-\conf$,
for some
$u = \tilde{O}\left( \frac{\dc^{1/2} \eps^{\tsyba/2}}{\BJM(\eps)^{2-\tsybb/2}} + \BJM(\eps)^{\tsybb-2}\right)$
and $n = \tilde{O}\left( \frac{\dc^{3/2} \eps^{3\tsyba/2}}{\BJM(\eps)^{2-\tsybb/2}} + \dc\eps^{\tsyba} \BJM(\eps)^{\tsybb-2}\right)$,
where $\dc = \dc(\tsybca \eps^{\tsyba})$, and $\tilde{O}(\cdot)$ hides logarithmic and constant factors.
Under Condition~\ref{con:strong-convexity},
with $\tsybb$ as in Lemma~\ref{lem:bjm-strong-convexity},
the conditions of Corollary~\ref{cor:abstract-conditionals}
can be satisfied with the probability bound being at least $1-\conf$,
for some
$u = \tilde{O}\left( \left(\frac{1}{\BJM(\eps)}\right)\left(\frac{\dc\eps^{\tsyba}}{\BJM(\eps)}\right)^{1-\tsybb/2} \right)$
and
$n = \tilde{O}\left(\left(\frac{\dc\eps^{\tsyba}}{\BJM(\eps)}\right)^{2-\tsybb/2} \right)$.
When $\dc$ is small, these values of $n$ (and indeed $u$) compare favorably to the value of 
$m = \tilde{O}\left( \BJM(\eps)^{\tsybb/2-2} \right)$, 
derived analogously from Theorem~\ref{thm:erm}, sufficient for $\ERM_{\loss}(\F,\Data_{m})$ to achieve the same \citep*[see][]{gine:06}.

For example, for $\X = [0,1]$ and $\F$ the class of all nondecreasing functions mapping $\X$ to $[-1,1]$,
$\F$ is a VC major class with index $1$, and $\dc(0) \leq 2$ for all distributions $\Px$.  Thus, for instance,
if $\eta$ is nondecreasing and $\loss$ is the quadratic loss, then $\target \in \F$, and \ACAL~achieves
excess error rate $\eps$ with high probability for some $u = \tilde{O}\left( \eps^{2\tsyba-3} \right)$
and $n = \tilde{O}\left( \eps^{3(\tsyba-1)} \right)$.

VC major classes are contained in special types of \emph{VC hull} classes, 
which are more generally defined as follows.
Let $\C$ be a VC Subgraph class of functions on $\X$, with bounded envelope,
and for $B \in (0,\infty)$, let
\begin{equation*}
\F = B\conv(\C) = \left\{ x \mapsto B \sum_{j} \lambda_j h_j(x)  : \sum_{j} |\lambda_j| \leq 1, h_j \in \C\right\}
\end{equation*}
denote the scaled symmetric convex hull of $\C$;
then $\F$ is called a VC hull class. For instance, these spaces
are often used in conjunction with the popular AdaBoost learning
algorithm.
One can derive results for VC hull classes following
analogously to the above, using established bounds on the uniform covering numbers
of VC hull classes \citep*[see][Corollary 2.6.12]{van-der-Vaart:96},
and noting that for any VC hull class $\F$ with envelope function $\Env$, 
and any $\region \subseteq \X$, $\F_{\region}$ is also a VC hull class, 
with envelope function $\Env \ind_{\region}$.
Specifically, one can use these observations to derive the following results.
For a VC hull class $\F = B\conv(\C)$,
if $\loss$ is classification-calibrated and Lipschitz on $[-\sup_{h\in\F}\sup_{x\in\X}|h(x)|,\sup_{h\in\F}\sup_{x\in\X}|h(x)|]$,
$\target \in \F$, and $\PXY$ satisfies Condition~\ref{con:tsybakov-01} and Condition~\ref{con:tsybakov-sur}, then
letting $\dim = 2\vc(\C)$,
the conditions of Theorem~\ref{thm:abstract-active} can be satisfied with the probability bound having value at least $1-\conf$,
for some
$u = \tilde{O}\left( \left(\dc \eps^{\tsyba}\right)^{\frac{\dim}{\dim+2}} \BJM(\eps)^{\frac{2\tsybb}{\dim+2}-2}\right)$
and
$n = \tilde{O}\left( \left(\dc\eps^{\tsyba}\right)^{\frac{2\dim+2}{\dim+2}} \BJM(\eps)^{\frac{2\tsybb}{\dim+2}-2} \right)$.
Under Condition~\ref{con:strong-convexity},
with $\tsybb$ as in Lemma~\ref{lem:bjm-strong-convexity},
the conditions of Corollary~\ref{cor:abstract-conditionals}
can be satisfied with the probability being at least $1-\conf$,
for some
$u = \tilde{O}\left( \left(\frac{1}{\BJM(\eps)}\right)\left(\frac{\dc\eps^{\tsyba}}{\BJM(\eps)}\right)^{1-\frac{2\tsybb}{\dim+2}}\right)$
and
$n = \tilde{O}\left( \left(\frac{\dc\eps^{\tsyba}}{\BJM(\eps)}\right)^{2-\frac{2\tsybb}{\dim+2}}\right)$.
Compare these to the value $m = \tilde{O}\left( \BJM(\eps)^{\frac{2\tsybb}{\dim+2}-2} \right)$, 
derived analogously from Theorem~\ref{thm:erm},
sufficient for $\ERM_{\loss}(\F,\Data_{m})$ to achieve the same general guarantee \citep*[see also][]{blanchard:03,bartlett:06}.
However, it is not clear whether these results for active learning with VC hull classes have any practical implications,
since we do not know of any scenarios where this sufficient value of $m$ reflects a \emph{tight} analysis of $\ERM_{\loss}(\F,\cdot)$
while simultaneously being significantly larger than either of the above sufficient $n$ values.

\section{Computationally Efficient Updates}
\label{sec:hatT}

As mentioned in Section~\ref{subsec:hatT-spec}, though convenient in the sense that it offers a completely abstract and unified approach, 
the choice of $\hat{T}_{\loss}(V;Q,m)$ given by \eqref{eqn:hatT-acal-U} may often make \ACAL~computationally inefficient.
However, for each of the applications studied in this work, we can relax this $\hat{T}_{\loss}$ function to a computationally-accessible
value, which will then allow the algorithm to be efficient under convexity conditions on the loss and class of functions. 

In particular, in the application to VC Subgraph classes, Theorem~\ref{thm:vc-subgraph-abstract} remains valid 
if we instead define $\hat{T}_{\loss}$ as follows.
If we let $V^{(m)}$ and $Q_{m}$ denote the sets $V$ and $Q$ upon reaching Step 5 
for any given value of $m$ with $\log_{2}(m) \in \nats$ realized in \ACAL,
then consider defining $\hat{T}_{\loss}$ in Step 6 inductively by 
letting 
\begin{equation*}
\hat{\gamma}_{m/2} = \frac{8(|Q_{m/2}|\lor 1)}{m} \left(\hat{T}_{\loss}(V^{(m/2)};Q_{m/2},m/2) \land \maxloss\right)
\end{equation*}
(or $\hat{\gamma}_{m/2} = \maxloss$ if $m = 2$), 
and taking 
(with a slight abuse of notation to allow $\hat{T}_{\loss}$ to depend 
on sets $V^{(m^{\prime})}$ and $Q_{m^{\prime}}$ with $m^{\prime} < m$)
\begin{multline}
\label{eqn:vc-subgraph-abstract-hatT}
\hat{T}_{\loss}(V^{(m)}; Q_{m}, m) = 
\\ c_{0} \frac{m/2}{|Q_{m}| \lor 1} 
\Vast( \sqrt{ \hat{\gamma}_{m/2}^{\tsybb} \frac{\tsybcb}{m} \left( \vc(\G_{\F}) \Log\left( \frac{\maxloss (|Q_{m}|+\hat{\sfun}(m))}{m \tsybcb \hat{\gamma}_{m/2}^{\tsybb}} \right) + \hat{\sfun}(m) \right)}
\\ + \frac{\maxloss}{m} \left( \vc(\G_{\F}) \Log\left( \frac{\maxloss (|Q_{m}|+\hat{\sfun}(m))}{m \tsybcb \hat{\gamma}_{m/2}^{\tsybb}} \right) + \hat{\sfun}(m)\right) \Vast),
\end{multline}
for an appropriate universal constant $c_{0}$.
This value is essentially derived by bounding $\frac{m/2}{|Q|\lor 1}\tilde{U}_{\loss}(V_{\DIS(V)}; \PXY, m/2, \hat{\sfun}(m))$ 
(which is a bound on \eqref{eqn:hatT-acal-U} by Lemma~\ref{lem:koltchinskii}),
based on \eqref{eqn:vc-subgraph-phi-bound} and Condition~\ref{con:tsybakov-sur} 
(and a Chernoff bound to argue $|Q_{m}| \approx \Px(\DIS(V)) m/2$); 
since the sample sizes derived for $u$ and $n$ in Theorem~\ref{thm:vc-subgraph-abstract} 
are based on these relaxations anyway, they remain sufficient (with slight changes to the constant factors) 
for these relaxed $\hat{T}_{\loss}$ values.
We include a more detailed proof that these values of $\hat{T}_{\loss}$ suffice to achieve 
Theorem~\ref{thm:vc-subgraph-abstract} in Appendix~\ref{app:vc-subgraph-abstract-hatT}.
Note that we have introduced a dependence on $\tsybcb$ and $\tsybb$ in \eqref{eqn:vc-subgraph-abstract-hatT}.
These values would indeed be available for some applications, such as when they are derived from Lemma~\ref{lem:bjm-strong-convexity} when Condition~\ref{con:strong-convexity} is satisfied;
however, in other cases, there may be more-favorable values of $\tsybcb$ and $\tsybb$ than given by Lemma~\ref{lem:bjm-strong-convexity}, dependent on the specific $\PXY$ distribution, 
and in these cases direct observation of these values might not be available.
Thus, there remains an interesting open question of whether there exists a function $\hat{T}_{\loss}(V;Q,m)$, 
which is efficiently computable (under convexity assumptions) and yet preserves the validity of Theorem~\ref{thm:vc-subgraph-abstract}.

In the special case where Condition~\ref{con:strong-convexity} is satisfied, it is also possible to define a value for $\hat{T}_{\loss}$
that is computationally accessible, and preserves the validity of Theorem~\ref{thm:vc-subgraph-active-strong}.
Specifically, consider instead defining $\hat{T}_{\loss}$ in Step 6 as
\begin{multline}
\label{eqn:vc-subgraph-active-strong-hatT}
\hat{T}_{\loss}(V; Q, m) 
\\ =
\maxloss \land c_{0} \max
\begin{cases} 
\left( \frac{\tsybcb}{|Q| \lor 1} \left( \vc(\G_{\F}) \Log\left(\frac{\maxloss^{2}}{\tsybcb}\left(\frac{|Q|}{\tsybcb \vc(\G_{\F})}\right)^{\frac{\tsybb}{2-\tsybb}}\right) + \hat{\sfun}(m) \right)\right)^{\frac{1}{2-\tsybb}}
\\ \frac{\maxloss}{|Q| \lor 1} \left( \vc(\G_{\F}) \Log\left(\frac{\maxloss^{2}}{\tsybcb}\left(\frac{|Q|}{\maxloss \vc(\G_{\F})}\right)^{\tsybb}\right) + \hat{\sfun}(m) \right)
\end{cases},
\end{multline}
for $\tsybcb$ and $\tsybb$ as in Lemma~\ref{lem:bjm-strong-convexity},
and for an appropriate universal constant $c_{0}$.
This value is essentially derived (following \citealp*{koltchinskii:06}) by using Lemma~\ref{lem:koltchinskii} 
under the conditional distribution $\PXYR{\DIS(V)}$, in conjunction with 
a localization technique similar to that employed in the derivation of Theorem~\ref{thm:erm}.
Appendix~\ref{app:vc-subgraph-active-strong-efficient} includes a proof that the conclusions of Theorem~\ref{thm:vc-subgraph-active-strong}
remain valid for this specification of $\hat{T}_{\loss}$ in place of \eqref{eqn:hatT-acal-U}.
That these conclusions remain valid for this bound on excess conditional risks should not be too surprising, since Theorem~\ref{thm:vc-subgraph-active-strong} 
is itself proven by considering concentration under the conditional distributions $\PXYR{\U_{j}}$ via Corollary~\ref{cor:abstract-conditionals}.
Note that, unlike the analogous result for Theorem~\ref{thm:vc-subgraph-abstract} based on \eqref{eqn:vc-subgraph-abstract-hatT}
above, in this case all of the quantities in $\hat{T}_{\loss}(V;Q,m)$ are directly observable (in particular, $\tsybcb$ and $\tsybb$),
aside from any possible dependence arising in the specification of $\hat{\sfun}$.

It is also possible to define computationally tractable values of $\hat{T}_{\loss}(V;Q,m)$ in scenarios satisfying the 
entropy conditions (Condition~\ref{con:entropy}), while preserving the validity of Theorem~\ref{thm:ent-active}.
This substitution can be derived analogously to 
\eqref{eqn:vc-subgraph-abstract-hatT} above, this time leading to the definition
\begin{multline}
\label{eqn:ent-active-hatT}
\hat{T}_{\loss}\left(V^{(m)};Q_{m},m\right) = 
\\ c_{0} \frac{m/2}{|Q_{m}| \lor 1} \Vast(
\max\left\{\frac{\sqrt{\entc} \|\Env\|_{\PXY}^{\entrho} \left( \tsybcb \hat{\gamma}_{m/2}^{\tsybb} \right)^{\frac{1-\entrho}{2}}}{(1-\entrho) m^{1/2}}, \frac{\maxloss^{\frac{1-\entrho}{1+\entrho}} \entc^{\frac{1}{1+\entrho}} \|\Env\|_{\PXY}^{\frac{2\entrho}{1+\entrho}}}{(1-\entrho)^{\frac{2}{1+\entrho}} m^{\frac{1}{1+\entrho}}}\right\}
\\ + \sqrt{\tsybcb \hat{\gamma}_{m/2}^{\tsybb} \frac{\hat{\sfun}(m)}{m}} + \frac{\maxloss \hat{\sfun}(m)}{m}
\Vast),
\end{multline}
where $\hat{\gamma}_{m/2}$ is defined (inductively) as above, and $c_{0}$ is an appropriately large universal constant.
By essentially the same argument used for \eqref{eqn:vc-subgraph-abstract-hatT} (see Appendix~\ref{app:vc-subgraph-abstract-hatT}), 
one can show that using \eqref{eqn:ent-active-hatT} in place of \eqref{eqn:hatT-acal-U} preserves the validity of Theorem~\ref{thm:ent-active}; 
for brevity, the details are omitted.

In the case that Condition~\ref{con:strong-convexity} and \eqref{eqn:uniform-entropy-bound} are satisfied,
it is possible to define a computationally accessible quantity $\hat{T}_{\loss}(V;Q,m)$, while preserving the validity of Theorem~\ref{thm:uniform-entropy-strong-convexity}.
Specifically, following the same reasoning used to arrive at \eqref{eqn:vc-subgraph-active-strong-hatT},
except using \eqref{eqn:ent-phi-bound} instead of \eqref{eqn:vc-subgraph-phi-bound}, we find that while replacing \eqref{eqn:hatT-acal-U} 
with the definition
\begin{multline}
\label{eqn:uniform-entropy-strong-convexity-hatT}
\hat{T}_{\loss}\left(V;Q,m\right) =
\\ \maxloss \land c_{0} \vast(
\max\left\{\left(\frac{\entc \maxloss^{2\entrho} \tsybcb^{1-\entrho}}{(1-\entrho)^{2} (|Q| \lor 1)}\right)^{\frac{1}{2-\tsybb(1-\entrho)}}, 
\frac{\maxloss \entc^{\frac{1}{1+\entrho}}}{(1-\entrho)^{\frac{2}{1+\entrho}} (|Q| \lor 1)^{\frac{1}{1+\entrho}}}\right\}
\\ + \left(\frac{\tsybcb \hat{\sfun}(m)}{|Q| \lor 1}\right)^{\frac{1}{2-\tsybb}} + \frac{\maxloss \hat{\sfun}(m)}{|Q| \lor 1} \vast),
\end{multline}
for $\tsybcb$ and $\tsybb$ as in Lemma~\ref{lem:bjm-strong-convexity}
and for an appropriate universal constant $c_{0}$,
the conclusions of Theorem~\ref{thm:uniform-entropy-strong-convexity} remain valid.
The proof follows similarly to the proof (in Appendix~\ref{app:vc-subgraph-active-strong-efficient})
that \eqref{eqn:vc-subgraph-active-strong-hatT} preserves the validity of Theorem~\ref{thm:vc-subgraph-active-strong},
and is omitted for brevity. 

Finally, in the case that Condition~\ref{con:strong-convexity} and \eqref{eqn:bracketing-entropy-bound} are satisfied, 
we can again derive an efficiently computable value of $\hat{T}_{\loss}(V;Q,m)$, which in this case 
preserves the validity of Theorem~\ref{thm:bracketing-entropy-strong-convexity}.
Specifically, noting that the reasoning preceding Theorem~\ref{thm:bracketing-entropy-strong-convexity} also implies 
$\ln \covering_{[]}\left( \eps \maxloss, \G_{V}, L_{2}(\PXYR{\DIS(V)})\right) \leq \entc \Px(\DIS(V))^{-\entrho} \eps^{-2\entrho}$,
and following the reasoning leading to \eqref{eqn:uniform-entropy-strong-convexity-hatT} while replacing $\entc$ with $\entc \Px(\DIS(V))^{-\entrho}$,
combined with a Chernoff bound to argue $\Px(\DIS(V)) \approx 2|Q|/m$ in the algorithm, we find that 
Theorem~\ref{thm:bracketing-entropy-strong-convexity} remains valid after replacing \eqref{eqn:hatT-acal-U} with the definition
\begin{multline*}
\hat{T}_{\loss}(V;Q,m) = 
\\ \maxloss \land c_{0} \vast(
\max\left\{\left(\frac{\entc m^{\entrho} \maxloss^{2\entrho} \tsybcb^{1-\entrho}}{(1-\entrho)^{2} (|Q| \lor 1)^{1+\entrho}}\right)^{\frac{1}{2-\tsybb(1-\entrho)}}, 
\frac{\maxloss \entc^{\frac{1}{1+\entrho}} m^{\frac{\entrho}{1+\entrho}}}{(1-\entrho)^{\frac{2}{1+\entrho}} (|Q| \lor 1)}\right\}
\\ + \left(\frac{\tsybcb \hat{\sfun}(m)}{|Q| \lor 1}\right)^{\frac{1}{2-\tsybb}} + \frac{\maxloss \hat{\sfun}(m)}{|Q| \lor 1} \vast),
\end{multline*}
for an appropriate universal constant $c_{0}$, and where $\tsybcb$ and $\tsybb$ are as in Lemma~\ref{lem:bjm-strong-convexity}.
The proof is essentially similar to that given for \eqref{eqn:vc-subgraph-active-strong-hatT} in Appendix~\ref{app:vc-subgraph-active-strong-efficient},
and is omitted for brevity. 

\section{Proofs for Efficiently Computable Updates}
\label{app:efficient-hatTs}

Here we include more detailed proofs of the arguments leading to computationally efficient variants of \ACAL,
for which the specific results proven in this work for the given applications remain valid.
Specifically, we focus on the application to VC Subgraph classes here; the applications
to scenarios satisfying the entropy conditions follow analogously.
Throughout this section, we adopt the notational conventions introduced in the proof of 
Theorem~\ref{thm:abstract-active} (e.g., $V^{(m)}$, $\tilde{V}^{(m)}$, $Q_{m}$, $\L_{m}$, $S$),
except in each instance here these are defined in the context of applying \ACAL~with 
the respective stated variant of $\hat{T}_{\loss}$.

\subsection{Proof of Theorem~\ref{thm:vc-subgraph-abstract} under \eqref{eqn:vc-subgraph-abstract-hatT}}
\label{app:vc-subgraph-abstract-hatT}

We begin by showing that if we specify $\hat{T}_{\loss}(V;Q,m)$ as in \eqref{eqn:vc-subgraph-abstract-hatT}, 
the conclusions of Theorem~\ref{thm:vc-subgraph-abstract} remain valid.
Fix any $\hat{\sfun}$ function (to be specified below), and fix any value of $\eps \in (0,1)$.
First note that, for any $m$ with $\log_{2}(m) \in \nats$, by a Chernoff bound and the law of total probability, 
on an event $E_{m}^{\prime\prime}$ of probability at least $1-2^{1-\hat{\sfun}(m)}$, if $m \in S$, then
\begin{equation}
\label{eqn:vc-subgraph-abstract-hatT-Q-bounds}
(1/2) m \Px(\DISV_{m}) - \sqrt{\hat{\sfun}(m) m \Px(\DISV_{m})} \leq |Q_{m}| \leq \hat{\sfun}(m) + e m \Px(\DISV_{m}).
\end{equation}
Also recall that, for any $m$ with $\log_{2}(m) \in \nats$, 
by Lemma~\ref{lem:koltchinskii} and the law of total probability, on an event $E_{m}$ of probability at least $1 - 6e^{-\hat{\sfun}(m)}$,
if $m \in S$ and $\target \in V^{(m)}$, then 
\begin{multline}
\label{eqn:vc-subgraph-abstract-hatT-kolt-1}
(|Q_{m}| \lor 1)\left(\Risk(\target; Q_{m}) - \inf_{g \in V^{(m)}} \Risk(g; Q_{m})\right)
\\ = \frac{m}{2}\left( \Risk(\target; \L_{m}) - \inf_{g_{\DISV_{m}} \in V^{(m)}_{\DISV_{m}}} \Risk(g_{\DISV_{m}}; \L_{m}) \right)
\\ < \frac{m}{2}\tilde{U}_{\loss}\left( V^{(m)}_{\DISV_{m}}; \PXY, m/2, \hat{\sfun}(m) \right)
\end{multline}
and $\forall h \in \tilde{V}^{(m)}$,
\begin{align}
& \frac{m}{2} \left( \Risk(h_{\DISV_{m}}) - \Risk(\target) \right) \notag
\\ & < \frac{m}{2} \left( \Risk(h_{\DISV_{m}}; \L_{m}) - \Risk(\target; \L_{m}) + \tilde{U}_{\loss}\left( V^{(m)}_{\DISV_{m}}; \PXY, m/2, \hat{\sfun}(m) \right) \land \maxloss \right) \notag
\\ & = |Q_{m}|\left( \Risk(h; Q_{m}) - \Risk(\target; Q_{m}) \right) + \frac{m}{2} \left(\tilde{U}_{\loss}\left( V^{(m)}_{\DISV_{m}}; \PXY, m/2, \hat{\sfun}(m) \right) \land \maxloss\right) \notag
\\ & \leq (|Q_{m}| \lor 1) \hat{T}_{\loss}\left( V^{(m)}; Q_{m}, m \right) + \frac{m}{2} \left(\tilde{U}_{\loss}\left( V^{(m)}_{\DISV_{m}}; \PXY, m/2, \hat{\sfun}(m) \right) \land \maxloss\right). \label{eqn:vc-subgraph-abstract-hatT-kolt-2}
\end{align}

Fix a value $i_{\eps} \in \nats$ (an appropriate value for which will be determined below),
and let $\mixcap = \mixcap( \BJM(\eps) )$.
For $m \in \nats$ with $\log_{2}(m) \in \nats$, let 
\begin{multline*}
\tilde{T}_{\loss}(m) = 
c_{2} \left( \frac{\tsybcb}{m} \left( \vc(\G_{\F}) \Log(\mixcap\maxloss) + \hat{\sfun}(m) \right) \right)^{\frac{1}{2-\tsybb}} 
\\ + c_{2} \frac{\maxloss}{m} \left( \vc(\G_{\F}) \Log(\mixcap\maxloss) + \hat{\sfun}(m) \right),
\end{multline*}
for an appropriate universal constant $c_{2} \in [1,\infty)$ (to be determined below);
for completeness, also define $\tilde{T}_{\loss}(1) = \maxloss$.
We will now prove by induction that, for an appropriate value of the constant $c_{0}$ in \eqref{eqn:vc-subgraph-abstract-hatT}, 
for any $m^{\prime}$ with $\log_{2}(m^{\prime}) \!\in\! \{1,\ldots,i_{\eps}\}$,
on the event $\bigcap_{i = 1}^{\log_{2}(m^{\prime})-1} E_{2^{i}} \cap E_{2^{i+1}}^{\prime\prime}$,
if $m^{\prime} \in S$, then $\target \in V^{(m^{\prime})}$, 
\begin{equation*}
V^{(m^{\prime})}_{\DISV_{m^{\prime}}} \subseteq \Bracket{\F}( \hat{\gamma}_{m^{\prime}/2}; \loss ) \subseteq \Bracket{\F}( 2 \tilde{T}_{\loss}(m^{\prime}/2) \lor \BJM(\eps); \loss),
\end{equation*} 
\begin{equation*}
V^{(m^{\prime})} \subseteq \F( \InvTransform( \hat{\gamma}_{m^{\prime}/2} ) ; \zo ) \subseteq \F( \InvTransform( 2 \tilde{T}_{\loss}(m^{\prime}/2) \lor \BJM(\eps)); \zo ),
\end{equation*}
\begin{equation*}
\tilde{U}_{\loss}\left( V^{(m^{\prime})}_{\DISV_{m^{\prime}}}; \PXY, m^{\prime}/2, \hat{\sfun}(m^{\prime}) \right) \land \maxloss
\leq \frac{|Q_{m^{\prime}}| \lor 1}{m^{\prime}/2} \left(\hat{T}_{\loss}\left( V^{(m^{\prime})}; Q_{m^{\prime}}, m^{\prime} \right) \land \maxloss\right),
\end{equation*}
and if $\hat{\gamma}_{m^{\prime}/2} \geq \BJM(\eps)$,
\begin{equation*}
\frac{|Q_{m^{\prime}}| \lor 1}{m^{\prime}/2} \left(\hat{T}_{\loss}\left( V^{(m^{\prime})}; Q_{m^{\prime}}, m^{\prime} \right) \land \maxloss\right)
\leq \tilde{T}_{\loss}(m^{\prime}).
\end{equation*}
As a base case for this inductive argument, we note that for $m^{\prime} = 2$, 
we have (by definition) $\hat{\gamma}_{m^{\prime}/2} = \maxloss$, and furthermore (if $c_{0} \land c_{2} \geq 2$)
$\hat{T}_{\loss}( V^{(2)}; Q_{2}, 2 ) \geq \maxloss$ and $\tilde{T}_{\loss}(1) \geq \maxloss$, 
so that the claimed inclusions and inequalities trivially hold.
Now, for the inductive step, take as an inductive hypothesis that the claim is satisfied for $m^{\prime} = m$ 
for some $m \in \nats$ with $\log_{2}(m) \in \{1,\ldots,i_{\eps}-1\}$.  Suppose the event 
$\bigcap_{i=1}^{\log_{2}(m)} E_{2^{i}} \cap E_{2^{i+1}}^{\prime\prime}$ occurs, and that $2m \in S$.
By the inductive hypothesis, combined with \eqref{eqn:vc-subgraph-abstract-hatT-kolt-1} and the fact that 
$(|Q_{m}| \lor 1) \Risk(\target; Q_{m}) \leq (m/2) \maxloss$, we have 
\begin{multline*}
(|Q_{m}| \lor 1)\left(\Risk(\target; Q_{m}) - \inf_{g \in V^{(m)}} \Risk(g; Q_{m})\right)
\\ \leq \frac{m}{2} \left(\tilde{U}_{\loss}\left( V^{(m)}_{\DISV_{m}}; \PXY, m/2, \hat{\sfun}(m) \right) \land \maxloss \right)
\leq (|Q_{m}| \lor 1) \hat{T}_{\loss}\left( V^{(m)}; Q_{m}, m \right).
\end{multline*}
Therefore, $\target \in \tilde{V}^{(m)}$ as well, which implies $\target \in V^{(2m)} = \tilde{V}^{(m)}$.
Furthermore, by \eqref{eqn:vc-subgraph-abstract-hatT-kolt-2}, the inductive hypothesis, and the definition of $\tilde{V}^{(m)}$ from Step 6,
$\forall h \in V^{(2m)} = \tilde{V}^{(m)}$,
\begin{equation*}
\Risk(h_{\DISV_{m}}) - \Risk(\target)
< 2 \frac{|Q_{m}| \lor 1}{m/2} \left(\hat{T}_{\loss}\left( V^{(m)}; Q_{m}, m \right) \land \maxloss \right),
\end{equation*}
and if $\hat{\gamma}_{m/2} \geq \BJM(\eps)$, then this is at most $2 \tilde{T}_{\loss}(m)$.

Since $\hat{\gamma}_{m} = 2 \frac{|Q_{m}| \lor 1}{m/2} \left(\hat{T}_{\loss}\left( V^{(m)}; Q_{m}, m \right) \land \maxloss\right)$, 
and $\Risk(h_{\DISV_{2m}}) \leq \Risk(h_{\DISV_{m}})$ for every $h \in V^{(2m)}$,
we have $V^{(2m)}_{\DISV_{2m}} \subseteq \Bracket{\F}( \hat{\gamma}_{m}; \loss ) \subseteq \Bracket{\F}( 2 \tilde{T}_{\loss}(m) \lor \BJM(\eps); \loss )$.
By definition of $\InvTransform(\cdot)$, we also have 
$\er(h_{\DISV_{2m}}) - \er(\target) \leq \InvTransform( \hat{\gamma}_{m} )$ for every $h \in V^{(2m)}$;
since $\target \in V^{(2m)}$, we have $\sign(h_{\DISV_{2m}}) = \sign(h)$, so that 
$\er(h) - \er(\target) \leq \InvTransform( \hat{\gamma}_{m} )$ as well: that is, 
$V^{(2m)} \subseteq \F( \InvTransform( \hat{\gamma}_{m} ); \zo ) \subseteq \F( \InvTransform( 2 \tilde{T}_{\loss}(m) \lor \BJM(\eps) ); \zo )$.
Combining these facts with \eqref{eqn:ophi-2}, \eqref{eqn:vc-subgraph-phi-bound}, Condition~\ref{con:tsybakov-sur}, 
monotonicity of $\vc(\G_{\H_{\region}})$ in both $\region$ and $\H$, 
and the fact that $\| \Env(\G_{V^{(2m)}_{\DISV_{2m}},\PXY}) \|_{\PXY}^{2} \leq \maxloss^{2} \Px(\DISV_{2m})$,
we have that
\begin{multline}
\label{eqn:vc-subgraph-abstract-hatT-tilde-bound}
\tilde{U}_{\loss}\left(V^{(2m)}_{\DISV_{2m}};\PXY,m,\hat{\sfun}(2m)\right) 
\leq c_{1}\sqrt{ \tsybcb \hat{\gamma}_{m}^{\tsybb} \frac{\vc(\G_{\F}) \Log\left( \frac{\maxloss \Px(\DISV_{2m})}{\tsybcb \hat{\gamma}_{m}^{\tsybb}} \right) + \hat{\sfun}(2m)}{m} } 
\\ + c_{1} \maxloss \frac{\vc(\G_{\F}) \Log\left( \frac{\maxloss \Px(\DISV_{2m})}{\tsybcb \hat{\gamma}_{m}^{\tsybb}} \right) + \hat{\sfun}(2m)}{m},
\end{multline}
for some universal constant $c_{1} \in [1,\infty)$.
By \eqref{eqn:vc-subgraph-abstract-hatT-Q-bounds}, we have $\Px(\DISV_{2m}) \leq \frac{3}{m}( |Q_{2m}| + \hat{\sfun}(2m) )$,
so that the right hand side of \eqref{eqn:vc-subgraph-abstract-hatT-tilde-bound} is at most
\begin{multline*}
c_{1}\sqrt{ \tsybcb \hat{\gamma}_{m}^{\tsybb} \frac{\vc(\G_{\F}) \Log\left( \frac{\maxloss 6 ( |Q_{2m}| + \hat{\sfun}(2m) )}{2m \tsybcb \hat{\gamma}_{m}^{\tsybb}} \right) + \hat{\sfun}(2m)}{m} } 
\\ + c_{1} \maxloss \frac{\vc(\G_{\F}) \Log\left( \frac{\maxloss 6 (|Q_{2m}| + \hat{\sfun}(2m))}{2m \tsybcb \hat{\gamma}_{m}^{\tsybb}} \right) + \hat{\sfun}(2m)}{m}
\\ \leq
8c_{1}\sqrt{ \tsybcb \hat{\gamma}_{m}^{\tsybb} \frac{\vc(\G_{\F}) \Log\left( \frac{\maxloss ( |Q_{2m}| + \hat{\sfun}(2m) )}{2m \tsybcb \hat{\gamma}_{m}^{\tsybb}} \right) + \hat{\sfun}(2m)}{2m} } 
\\ + 8 c_{1} \maxloss \frac{\vc(\G_{\F}) \Log\left( \frac{\maxloss (|Q_{2m}| + \hat{\sfun}(2m))}{2m \tsybcb \hat{\gamma}_{m}^{\tsybb}} \right) + \hat{\sfun}(2m)}{2m}.
\end{multline*}
Thus, if we take $c_{0} = 8 c_{1}$ in the definition of $\hat{T}_{\loss}$ in \eqref{eqn:vc-subgraph-abstract-hatT},
then we have 
\begin{equation*}
\tilde{U}_{\loss}\left( V^{(2m)}_{\DISV_{2m}}; \PXY, m, \hat{\sfun}(2m) \right) \land \maxloss
\leq \frac{|Q_{2m}| \lor 1}{m} \left(\hat{T}_{\loss}\left( V^{(2m)}; Q_{2m}, 2m \right) \land \maxloss\right).
\end{equation*}
Furthermore, \eqref{eqn:vc-subgraph-abstract-hatT-Q-bounds} implies $|Q_{2m}| \leq \hat{\sfun}(2m) + 2 e m \Px(\DISV_{2m})$.
In particular, if $\hat{\sfun}(2m)$ $> 2 e m \Px(\DISV_{2m})$, then 
\begin{equation*}
\frac{|Q_{2m}| \lor 1}{m} \left(\hat{T}_{\loss}\left( V^{(2m)}; Q_{2m}, 2m \right) \land \maxloss\right)
\leq \frac{\hat{\sfun}(2m) + 2 e m \Px(\DISV_{2m})}{m} \maxloss 
\leq \frac{2 \hat{\sfun}(2m) \maxloss}{m},
\end{equation*}
and taking any $c_{2} \geq 4$ guarantees this last quantity is at most $\tilde{T}_{\loss}(2m)$.
On the other hand, if $\hat{\sfun}(2m) \leq 2 e m \Px(\DISV_{2m})$, then $|Q_{2m}| \leq 4 e m \Px(\DISV_{2m})$,
and we have already established that $V^{(2m)} \subseteq \F( \InvTransform( \hat{\gamma}_{m} ); \zo )$, so that
\begin{multline}
\label{eqn:vc-subgraph-abstract-hatT-hatT-bound}
\frac{|Q_{2m}| \lor 1}{m} \left(\hat{T}_{\loss}\left( V^{(2m)}; Q_{2m}, 2m \right) \land \maxloss\right)
\\ \leq 
8c_{1}\sqrt{ \tsybcb \hat{\gamma}_{m}^{\tsybb} \frac{\vc(\G_{\F}) \Log\left( \frac{\maxloss 3 e \Px(\DIS(\F(\InvTransform(\hat{\gamma}_{m});\zo)))}{\tsybcb \hat{\gamma}_{m}^{\tsybb}} \right) + \hat{\sfun}(2m)}{2m} } 
\\ + 8 c_{1} \maxloss \frac{\vc(\G_{\F}) \Log\left( \frac{\maxloss 3 e \Px(\DIS(\F(\InvTransform(\hat{\gamma}_{m});\zo)))}{\tsybcb \hat{\gamma}_{m}^{\tsybb}} \right) + \hat{\sfun}(2m)}{2m}.
\end{multline}
If $\hat{\gamma}_{m} \geq \BJM(\eps)$, then this is at most 
\begin{multline*}
8c_{1}\left(\sqrt{ \tsybcb \hat{\gamma}_{m}^{\tsybb} \frac{\vc(\G_{\F}) \Log\left( 3 e \mixcap\maxloss\right) + \hat{\sfun}(2m)}{2m} } 
+ \maxloss \frac{\vc(\G_{\F}) \Log\left( 3 e \mixcap\maxloss \right) + \hat{\sfun}(2m)}{2m}\right)
\\ \leq 
48c_{1}\left(\sqrt{ \tsybcb \hat{\gamma}_{m}^{\tsybb} \frac{\vc(\G_{\F}) \Log\left( \mixcap\maxloss\right) + \hat{\sfun}(2m)}{2m} } 
+ \maxloss \frac{\vc(\G_{\F}) \Log\left( \mixcap\maxloss \right) + \hat{\sfun}(2m)}{2m}\right).
\end{multline*}
For brevity, let $K = \frac{\vc(\G_{\F}) \Log( \mixcap\maxloss ) + \hat{\sfun}(2m)}{2m}$.
As argued above, $\hat{\gamma}_{m} \leq 2 \tilde{T}_{\loss}(m)$, so that the right hand side of the above inequality is at most
\begin{equation*}
48 \sqrt{2} c_{1} \left(\sqrt{ \tsybcb \tilde{T}_{\loss}(m)^{\tsybb} K } 
+ \maxloss K \right).
\end{equation*}
Then since $\hat{\sfun}(m) \leq 2 \hat{\sfun}(2m)$,
the above expression is at most
\begin{equation}
\label{eqn:vc-subgraph-abstract-hatT-intermediate-tildeT}
48 \cdot 4 c_{1} \sqrt{c_{2}} \left(\sqrt{ \tsybcb \left( ( \tsybcb K )^{\frac{1}{2-\tsybb}} \lor \maxloss K \right)^{\tsybb} K } 
+ \maxloss K \right).
\end{equation}
If $\maxloss K \leq (\tsybcb K)^{\frac{1}{2-\tsybb}}$, then \eqref{eqn:vc-subgraph-abstract-hatT-intermediate-tildeT} is equal
\begin{equation*}
48 \cdot 4 c_{1} \sqrt{c_{2}} \left( ( \tsybcb K )^{\frac{1}{2-\tsybb}} + \maxloss K \right).
\end{equation*}
On the other hand, if $\maxloss K > (\tsybcb K)^{\frac{1}{2-\tsybb}}$, then \eqref{eqn:vc-subgraph-abstract-hatT-intermediate-tildeT} is equal
\begin{multline*}
48 \cdot 4 c_{1} \sqrt{c_{2}} \left(\sqrt{ \tsybcb K ( \maxloss K )^{\tsybb} } + \maxloss K \right)
\\ < 48 \cdot 4 c_{1} \sqrt{c_{2}} \left(\sqrt{ ( \maxloss K )^{2-\tsybb} ( \maxloss K )^{\tsybb} } + \maxloss K \right)
= 48 \cdot 8 c_{1} \sqrt{c_{2}} \maxloss K.
\end{multline*}
In all of the above cases, taking $c_{2} = 9 \cdot 2^{14} c_{1}^{2}$ in the definition of $\tilde{T}_{\loss}$ yields
\begin{equation*}
\frac{|Q_{2m}| \lor 1}{m} \left(\hat{T}_{\loss}\left( V^{(2m)}; Q_{2m}, 2m \right) \land \maxloss\right)
\leq \tilde{T}_{\loss}(2m).
\end{equation*}
This completes the inductive step, so that we have proven that the claim holds for all $m^{\prime}$ with $\log_{2}(m^{\prime}) \in \{1,\ldots,i_{\eps}\}$.

Let $j_{\loss} = - \lceil \log_{2}(\maxloss) \rceil$, $\tilde{j}_{\eps} = \lceil \log_{2}(1/\BJM(\eps)) \rceil$,
and for each $j \in \{j_{\loss},\ldots,\tilde{j}_{\eps}\}$, let
$s_{j} = \log_{2}\left( \frac{144 (2+\tilde{j}_{\eps}-j)^{2}}{\conf} \right)$, define
\begin{equation*}
m_{j}^{\prime} = 32 c_{2}^{2} \left(\tsybcb 2^{j(2-\tsybb)} + \maxloss 2^{j} \right)\left( \vc(\G_{\F}) \Log(\mixcap\maxloss) + s_{j} \right),
\end{equation*}
and let $m_{j} = 2^{\lceil \log_{2}(m_{j}^{\prime}) \rceil}$.
Also define $m_{j_{\loss}-1} = 1$.
Using this notation, we can now define the relevant values of the $\hat{\sfun}$ function as follows.
For each $j \in \{j_{\loss},\ldots,\tilde{j}_{\eps}\}$, and each $m \in \{m_{j-1}+1,\ldots,m_{j}\}$ with $\log_{2}(m) \in \nats$, define
\begin{equation*}
\hat{\sfun}(m) = \log_{2}\left( \frac{16 \log_{2}(4m_{j}/m)^{2} (2 + \tilde{j}_{\eps} - j)^{2}}{\conf} \right).
\end{equation*}

In particular, taking $i_{\eps} = \log_{2}(m_{\tilde{j}_{\eps}})$,
we have that $2\tilde{T}_{\loss}(2^{i_{\eps}-1}) \leq \BJM(\eps)$, so that on the event 
$\bigcap_{i=1}^{i_{\eps}-1} E_{2^{i}} \cap E_{2^{i+1}}^{\prime\prime}$, 
if we have $2^{i_{\eps}} \in S$, 
then $\hat{h} \in V^{(2^{i_{\eps}})} \!\subseteq\! \F( \InvTransform( 2 \tilde{T}_{\loss}(2^{i_{\eps}-1}) \lor \BJM(\eps));\zo) \!=\! \F( \InvTransform( \BJM(\eps) ) ; \zo ) \!\subseteq\! \F( \BJM^{-1}( \BJM(\eps) ) ; \zo ) \!= \! \F( \eps ; \zo )$,
so that $\er(\hat{h}) - \er(\target) \leq \eps$.

Furthermore, we established above that, on the event $\bigcap_{i=1}^{i_{\eps}-1} E_{2^{i}} \cap E_{2^{i+1}}^{\prime\prime}$,
for every $j \in \{j_{\loss},\ldots,\tilde{j}_{\eps}\}$ with $m_{j} \in S$, and every $m \in \{m_{j-1}+1,\ldots,m_{j}\}$ with $\log_{2}(m) \in \nats$,
$V^{(m)} \subseteq \F( \InvTransform( 2 \tilde{T}_{\loss}(m/2) \lor \BJM(\eps)); \zo ) \subseteq \F( \InvTransform( 2 \tilde{T}_{\loss}(m_{j-1}) \lor \BJM(\eps) ) ; \zo )$.
Noting that $2 \tilde{T}_{\loss}(m_{j-1}) \leq 2^{1-j}$, we have 
\begin{equation*}
\sum_{m \in S : m \leq m_{\tilde{j}_{\eps}}} |Q_{m}|
\leq \sum_{j=j_{\loss}}^{\tilde{j}_{\eps}} \sum_{m=m_{j-1}+1}^{m_{j}} \ind_{\DIS(\F( \InvTransform( 2^{1-j} ) ; \zo ))}(X_{m}).
\end{equation*}
A Chernoff bound implies that, on an event $E^{\prime}$ of probability at least $1-\conf/2$, the right hand side of the above inequality is at most
\begin{multline*}
\log_{2}(2/\conf) + 2 e \sum_{j=j_{\loss}}^{\tilde{j}_{\eps}} (m_{j} - m_{j-1}) \Px(\DIS(\F( \InvTransform( 2^{1-j} ) ; \zo )))
\\ \leq \log_{2}(2/\conf) + 2 e \sum_{j=j_{\loss}}^{\tilde{j}_{\eps}} m_{j} \Px(\DIS(\F( \BJM^{-1}( 2^{1-j} ) ; \zo ))).
\end{multline*}
By essentially the same reasoning used in the proof of Theorem~\ref{thm:vc-subgraph-abstract}, the right hand side of this inequality is
\begin{equation*}
\lesssim \tsybca \dc \eps^{\tsyba} \left( \frac{\tsybcb (A_{1}+\Log(B_{1}))B_{1}}{\BJM(\eps)^{2-\tsybb}} + \frac{\maxloss( A_{1} + \Log(C_{1}) ) C_{1}}{\BJM(\eps)} \right).
\end{equation*}
Since
\begin{equation*}
m_{\tilde{j}_{\eps}} \lesssim \left(\frac{\tsybcb}{\BJM(\eps)^{2-\tsybb}} + \frac{\maxloss}{\BJM(\eps)}\right) A_{1},
\end{equation*}
the conditions on $u$ and $n$ stated in Theorem~\ref{thm:vc-subgraph-abstract} (with an appropriate constant $c$) suffice 
to guarantee $\er(\hat{h}) - \er(\target) \leq \eps$ on the event $E^{\prime} \cap \bigcap_{i=1}^{i_{\eps}-1} E_{2^{i}} \cap E_{2^{i+1}}^{\prime\prime}$.
Finally, the proof is completed by noting that a union bound implies the event $E^{\prime} \cap \bigcap_{i=1}^{i_{\eps}-1} E_{2^{i}} \cap E_{2^{i+1}}^{\prime\prime}$ has probability 
at least 
\begin{align*}
& 1 - \frac{\conf}{2} - \sum_{i=1}^{i_{\eps}-1} 2^{1-\hat{\sfun}(2^{i+1})} + 6 e^{-\hat{\sfun}(2^{i})} 
\\ & \geq 1 - \frac{\conf}{2} - \sum_{j=j_{\loss}}^{\tilde{j}_{\eps}} \sum_{i=\log_{2}(m_{j-1})+1}^{\log_{2}(m_{j})} \frac{\conf}{2 (2 + \log_{2}(m_{j}) - i)^{2} (2 + \tilde{j}_{\eps} - j)^{2}}
\\ & \geq 1 - \frac{\conf}{2} - \sum_{j=j_{\loss}}^{\tilde{j}_{\eps}} \sum_{k = 0}^{\infty} \frac{\conf}{2 (2+k)^{2} (2 + \tilde{j}_{\eps} - j)^{2}}
\\ & \geq 1 - \frac{\conf}{2} - \sum_{j=j_{\loss}}^{\tilde{j}_{\eps}} \frac{\conf}{2 (2 + \tilde{j}_{\eps} - j)^{2}}
\geq 1 - \frac{\conf}{2} - \sum_{t=0}^{\infty} \frac{\conf}{2 (2 + t)^{2}}
\geq 1 - \conf.
\end{align*}

Note that, as in Theorem~\ref{thm:vc-subgraph-abstract}, the function $\hat{\sfun}$ in this proof has a direct dependence on $\tsybca$, $\tsyba$, and $\mixcap$, 
in addition to $\tsybcb$ and $\tsybb$.  As before, with an alternative definition of $\hat{\sfun}$, 
similar to that mentioned in the discussion following the proof of Theorem~\ref{thm:vc-subgraph-abstract}, 
it is possible to remove this dependence, at the expense of the same logarithmic factors mentioned above.

\subsection{Proof of Theorem~\ref{thm:vc-subgraph-active-strong} under \eqref{eqn:vc-subgraph-active-strong-hatT}}
\label{app:vc-subgraph-active-strong-efficient}

Next, consider the conditions of Theorem~\ref{thm:vc-subgraph-active-strong}, 
and suppose the definition of $\hat{T}_{\loss}$ from \eqref{eqn:vc-subgraph-active-strong-hatT} is used in Step 6.
For simplicity, we let $V^{(m)}$ and $Q_{m}$ be defined (though arbitrarily) even when $m \notin S$. 
Fix a function $\hat{\sfun}$ (to be specified below) and any value of $\eps \in (0,1)$.
We will prove by induction that there exist events $\hat{E}_{m^{\prime}}$, for values $m^{\prime}$ with $\log_{2}(m^{\prime}) \in \nats$,
each with respective probability at least $1 - 12 e^{-\hat{\sfun}(m^{\prime})}$ such that,
for every $m$ with $\log_{2}(m) \in \nats$, on $\bigcap_{i=1}^{\log_{2}(m)} \hat{E}_{2^{i}}$,
if $m \in S$, we have that $\target \in \tilde{V}^{(m)}$ and $\tilde{V}^{(m)} \subseteq V^{(m)}\left( 4\hat{T}_{m}; \loss, \PXYR{\DISV_{m}} \right)$,
where $\hat{T}_{m} = \hat{T}_{\loss}\left(V^{(m)};Q_{m},m\right)$.
This claim is trivially satisfied for $m=2$, since $\hat{T}_{2}=\maxloss$, so this will serve as our base case in the inductive proof.
Now fix any $m > 2$ with $\log_{2}(m) \in \nats$, and take as an inductive hypothesis that there exist events $\hat{E}_{m^{\prime}}$ 
for each $m^{\prime} < m$ with $\log_{2}(m^{\prime}) \in \nats$,
such that, on $\bigcap_{i=1}^{\log_{2}(m)-1} \hat{E}_{2^{i}}$, if $m/2 \in S$, then $\target \in \tilde{V}^{(m/2)}$.
Note that, since $V^{(m)} = \tilde{V}^{(m/2)}$ (if $m \in S$), we have that $\target \in V^{(m)}$ on $\bigcap_{i=1}^{\log_{2}(m)-1} \hat{E}_{2^{i}}$ by the inductive hypothesis. 

For any $T > 0$, let $\sfun\left(T,\gamma\right) = \Log\left( \frac{\gamma}{T} \right)+\hat{\sfun}(m)$.
Note that \eqref{eqn:oU-vs-tildeU}, \eqref{eqn:spec-split-bound}, \eqref{eqn:restinv-bound}, Lemma~\ref{lem:bjm-strong-convexity}, \eqref{eqn:vc-subgraph-pre-phiinv}, and monotonicity of $\H \mapsto \vc(\G_{\H})$ imply that,
if $\target \in V^{(m)} \subseteq \F$, then
\begin{multline}
\label{eqn:vc-strong-efficient-M-bound}
\sup_{\gamma \geq T} \tilde{M}_{\loss}\left( \gamma/8, \gamma; V^{(m)}, \PXYR{\DISV_{m}}, \sfun(T,\gamma) \right)
\\ \leq
\bar{c} \left( \frac{\tsybcb}{T^{2-\tsybb}} + \frac{\maxloss}{T} \right) \left( \vc(\G_{\F}) \Log\left( \frac{\maxloss^{2}}{\tsybcb T^{\tsybb}} \right) + \hat{\sfun}(m) \right),
\end{multline}
for an appropriate finite universal constant $\bar{c} \geq 1$.
If $m \in S$ and $\hat{T}_{m} = \maxloss$, then we trivially have $\Risk(\target;Q_{m}) - \inf_{g \in V^{(m)}} \Risk(g;Q_{m}) \leq \hat{T}_{m}$, so that $\target \in \tilde{V}^{(m)}$,
and furthermore $\tilde{V}^{(m)} = V^{(m)} = V^{(m)}\left( 4\hat{T}_{m}; \loss, \PXYR{\DISV_{m}} \right)$.
Otherwise, if $m \in S$ and $\hat{T}_{m} < \maxloss$, we have that
\begin{equation*}
|Q_{m}| \geq \max
\begin{cases}
\left(\frac{c_{0}}{\hat{T}_{m}}\right)^{2-\tsybb} \tsybcb \left( \vc(\G_{\F}) \Log\left(\frac{\maxloss^{2}}{\tsybcb}\left(\frac{|Q_{m}|}{\tsybcb \vc(\G_{\F})}\right)^{\frac{\tsybb}{2-\tsybb}}\right)+\hat{\sfun}(m)\right)\\
\frac{c_{0} \maxloss}{\hat{T}_{m}} \left( \vc(\G_{\F}) \Log\left(\frac{\maxloss^{2}}{\tsybcb} \left(\frac{|Q_{m}|}{\maxloss \vc(\G_{\F})}\right)^{\tsybb}\right)+\hat{\sfun}(m) \right)
\end{cases},
\end{equation*}
which implies
\begin{multline*}
|Q_{m}|
\geq \max\left\{ \left(\frac{c_{0}}{\hat{T}_{m}}\right)^{2-\tsybb} \tsybcb, \frac{c_{0} \maxloss}{\hat{T}_{m}} \right\} \left( \vc(\G_{\F}) \Log\left(\frac{\maxloss^{2}}{\tsybcb \hat{T}_{m}^{\tsybb}}\right)+\hat{\sfun}(m)\right)
\\ \geq \frac{c_{0}}{2} \left( \frac{\tsybcb}{\hat{T}_{m}^{2-\tsybb}} + \frac{\maxloss}{\hat{T}_{m}}\right) \left( \vc(\G_{\F}) \Log\left(\frac{\maxloss^{2}}{\tsybcb \hat{T}_{m}^{\tsybb}}\right)+\hat{\sfun}(m)\right).
\end{multline*}
Combined with \eqref{eqn:vc-strong-efficient-M-bound}, this implies that if we take $c_{0} \geq 2 \bar{c}$, and if $\target \!\in\! V^{(m)} \!\subseteq\! \F$, then
\begin{equation}
\label{eqn:vc-strong-efficient-Q-achieves-M}
|Q_{m}| \geq \sup_{\gamma \geq \hat{T}_{m}} \tilde{M}_{\loss}\left( \gamma/8, \gamma; V^{(m)}, \PXYR{\DISV_{m}}, \sfun(\hat{T}_{m},\gamma) \right).
\end{equation}
We now follow the derivation of localized risk bounds by \citet*{koltchinskii:06}. 
Specifically, applying Lemma~\ref{lem:koltchinskii} under the conditional distribution given $V^{(m)}$ and $|Q_{m}|$, combined with the law of total probability,
there is an event $E_{m}^{\prime\prime}$ of conditional probability at least $1 - 6 \sum_{j \in \ints_{\hat{T}_{m}}} e^{-\sfun(\hat{T}_{m},2^{j})}$ (given $V^{(m)}$ and $|Q_{m}|$),
such that on $E_{m}^{\prime\prime}$, if $m \in S$, $\target \in V^{(m)}$, and $\hat{T}_{m} < \maxloss$ (so that \eqref{eqn:vc-strong-efficient-Q-achieves-M} holds),
then $\forall j \in \ints_{\hat{T}_{m}}$, the following claims hold
for every $h \in V^{(m)}\left( 2^{j}; \loss, \PXYR{\DISV_{m}} \right)$.
\begin{align}
\Risk(h;\PXYR{\DISV_{m}}) - \Risk(\target; \PXYR{\DISV_{m}}) & \leq \Risk(h;Q_{m}) - \Risk(\target;Q_{m}) + 2^{j-3}, \label{eqn:vc-strong-efficient-kolt1}
\\ \Risk(h;Q_{m}) - \!\!\inf_{g \in V^{(m)}(2^{j};\loss,\PXYR{\DISV_{m}})} \!\!\Risk(g;Q_{m}) & \leq \Risk(h;\PXYR{\DISV_{m}}) - \Risk(\target; \PXYR{\DISV_{m}}) + 2^{j-3}. \label{eqn:vc-strong-efficient-kolt2}
\end{align}
Since $\sum_{j \in \ints_{\hat{T}_{m}}} e^{-\sfun(\hat{T}_{m},2^{j})} = e^{-\hat{\sfun}(m)} \sum_{j \in \ints_{\hat{T}_{m}}} 2^{-j} \hat{T}_{m} \leq 2 e^{-\hat{\sfun}(m)}$,
the law of total probability implies that there exists an event $\hat{E}_{m}$ of probability at least $1 - 12 e^{-\hat{\sfun}(m)}$, on which this implication holds.
In particular, for any $h_{0} \in V^{(m)}$ with $\Risk(h_{0};Q_{m}) - \Risk(\target;Q_{m}) \leq 0$, \eqref{eqn:vc-strong-efficient-kolt1} implies 
that for any $j \in \ints_{\hat{T}_{m}}$, if $\Risk(h_{0};\PXYR{\DISV_{m}})$ $- \Risk(\target;\PXYR{\DISV_{m}}) \leq 2^{j}$, then $\Risk(h_{0};\PXYR{\DISV_{m}}) - \Risk(\target;\PXYR{\DISV_{m}}) \leq 2^{j-3}$;
this inductively implies that\break $\Risk(h_{0};\PXYR{\DISV_{m}}) - \Risk(\target;\PXYR{\DISV_{m}}) \leq \hat{T}_{m}$, so that \eqref{eqn:vc-strong-efficient-kolt2} can more simply be stated as:
$\forall h \in V^{(m)}\left( 2^{j}; \loss, \PXYR{\DISV_{m}} \right)$, 
\begin{equation*}
\Risk(h;Q_{m}) - \inf_{g \in V^{(m)}} \Risk(g;Q_{m}) 
\leq \Risk(h;\PXYR{\DISV_{m}}) - \Risk(\target; \PXYR{\DISV_{m}}) + 2^{j-3}.
\end{equation*}
Furthermore, this implies
\begin{equation}
\label{eqn:vc-strong-efficient-f-saved}
\Risk(\target;Q_{m}) - \inf_{g \in V^{(m)}} \Risk(g;Q_{m}) \leq \hat{T}_{m},
\end{equation}
so that $\target \in \tilde{V}^{(m)}$ in this case as well.
Also, \eqref{eqn:vc-strong-efficient-kolt1} and the fact that $\target \in V^{(m)}$ further imply that for any $h \in V^{(m)}$ with $\Risk(h;Q_{m}) - \inf_{g \in V^{(m)}} \Risk(g;Q_{m}) \leq \hat{T}_{m}$,
for any $j \in \ints_{4\hat{T}_{m}}$, if $\Risk(h;\PXYR{\DISV_{m}}) - \Risk(\target;\PXYR{\DISV_{m}}) \leq 2^{j}$, then 
$\Risk(h;\PXYR{\DISV_{m}}) - \Risk(\target;\PXYR{\DISV_{m}}) \leq \hat{T}_{m} + 2^{j-3} \leq 2^{j-2} + 2^{j-3} \leq 2^{j-1}$;
this inductively implies that any such $h$ has $\Risk(h;\PXYR{\DISV_{m}}) - \Risk(\target;\PXYR{\DISV_{m}}) \leq 4\hat{T}_{m}$.
In particular, by definition of $\tilde{V}^{(m)}$, this implies $\tilde{V}^{(m)} \subseteq V^{(m)}\left( 4\hat{T}_{m}; \loss, \PXYR{\DISV_{m}} \right)$.
Since the inductive hypothesis implies $\target \in V^{(m)}$ on $\bigcap_{i=1}^{\log_{2}(m)-1} \hat{E}_{2^{i}}$ if $m \in S$, 
we have that on $\bigcap_{i=1}^{\log_{2}(m)} \hat{E}_{2^{i}}$, if $m \in S$, then $\target \in \tilde{V}^{(m)}$ and $\tilde{V}^{(m)} \subseteq V^{(m)}\left( 4\hat{T}_{m}; \loss, \PXYR{\DISV_{m}} \right)$,
which extends the inductive hypothesis.
By the principle of induction, we have established this claim for every $m$ with $\log_{2}(m) \in \nats$.

Let $\hat{j}_{\eps} = \left\lceil \log_{2}(\maxloss / \BJM(\eps)) \right\rceil$.
For each $j \in \nats \cup \{0\}$, let $\eps_{j} = \maxloss 2^{-j}$,\break 
$p_{j} = \Px\left(\DIS\left(\F\left(\BJM^{-1}\left(\eps_{j}\right);\zo\right)\right)\right)$, and $s_{j} = \log_{2}\left(\frac{192 (2 + \hat{j}_{\eps} - j)^{2}}{\conf}\right)$.
Let $m_{0} = 1$, and for each $j \in \nats$, define
\begin{equation*}
m_{j}^{\prime} = c^{\prime} \left(\frac{\tsybcb p_{j-1}^{1-\tsybb}}{\eps_{j}^{2-\tsybb}} + \frac{\maxloss}{\eps_{j}}\right) \left( \vc(\G_{\F}) \Log\left(\frac{\maxloss^{2} (c^{\prime})^{\tsybb} p_{j-1}^{\tsybb}}{\tsybcb \eps_{j}^{\tsybb}}\right)+s_{j} \right),
\end{equation*}
for an appropriate universal constant $c^{\prime} \in [1,\infty)$ (specified below),
and let $m_{j} = \max\left\{ 2 m_{j-1}, 2^{1+\lceil \log_{2}(m_{j}^{\prime}) \rceil} \right\}$.
Also, for every $j \in \nats$ and $m \in \{2 m_{j-1},\ldots,m_{j}\}$, define
\begin{equation*}
\hat{\sfun}(m) = \log_{2}\left( \frac{48 \log_{2}(4m_{j}/m)^{2} (2+\hat{j}_{\eps} - j)^{2}}{\conf} \right).
\end{equation*}
In particular, this definition implies $\hat{\sfun}(m_{j}) = s_{j}$.

We next prove by induction that there are events $\hat{E}_{j}^{\prime}$, for $j \in \nats \cup \{0\}$,
each with respective probability at least $1 - 2^{-s_{j}}$, such that 
for every $j \in \nats \cup \{0\}$, on\break $\bigcap_{i=1}^{\log_{2}(m_{j})} \hat{E}_{2^{i}} \cap \bigcap_{j^{\prime}=0}^{j} \hat{E}_{j^{\prime}}^{\prime}$,
if $m_{j} \in S \cup \{1\}$, then $\tilde{V}^{(m_{j})} \subseteq \F\left( \BJM^{-1}(\eps_{j}) ; \zo \right)$.
This claim is trivially satisfied for $j=0$, which therefore serves as the base case for this inductive proof.
Now fix any $j \in \nats$, and take as an inductive hypothesis that there exist events $\hat{E}_{j^{\prime}}^{\prime}$,
as above, for all $j^{\prime} < j$, such that on $\bigcap_{i=1}^{\log_{2}(m_{j-1})} \hat{E}_{2^{i}} \cap \bigcap_{j^{\prime}=0}^{j-1} \hat{E}_{j^{\prime}}^{\prime}$,
if $m_{j-1} \in S$, then $\tilde{V}^{(m_{j-1})} \subseteq \F\left( \BJM^{-1}(\eps_{j-1}) ; \zo \right)$.
By the above, we have that on $\bigcap_{i=1}^{\log_{2}(m_{j})} \hat{E}_{2^{i}}$, if $m_{j} \in S$, then
$\target \in \tilde{V}^{(m_{j})} \subseteq V^{(m_{j})}\left( 4 \hat{T}_{m_{j}}; \loss, \PXYR{\DISV_{m_{j}}} \right)$.
In particular, this implies that every $h \in \tilde{V}^{(m_{j})}$ has 
\begin{multline}
\label{eqn:vc-strong-efficient-truncated-risk}
\Risk(h_{\DISV_{m_{j}}};\PXY) - \Risk(\target;\PXY)
= \left( \Risk(h;\PXYR{\DISV_{m_{j}}}) - \Risk(\target;\PXYR{\DISV_{m_{j}}}) \right) \Px(\DISV_{m_{j}})
\\ \leq 4 \hat{T}_{m_{j}} \Px(\DISV_{m_{j}}).
\end{multline}
By a Chernoff bound and the law of total probability, on an event $\hat{E}_{j}^{\prime}$ of probability at least $1-2^{-s_{j}}$, if $m_{j} \in S$, 
\begin{equation}
\label{eqn:vc-strong-efficient-Q-LB}
(1/2) m_{j} \Px(\DISV_{m_{j}}) - \sqrt{ s_{j} m_{j} \Px(\DISV_{m}) } \leq |Q_{m_{j}}|.
\end{equation}
If $m_{j} \in S$ and $\Px(\DISV_{m_{j}}) \leq \frac{16s_{j}}{m_{j}}$, then 
$4 \hat{T}_{m_{j}} \Px(\DISV_{m_{j}}) \leq \frac{64 \maxloss s_{j}}{m_{j}} \leq \frac{32 \eps_{j}}{c^{\prime}}$,
so that with any $c^{\prime} \geq 32$, \eqref{eqn:vc-strong-efficient-truncated-risk} would give $\Risk(h_{\DISV_{m_{j}}};\PXY) - \Risk(\target;\PXY) \leq \eps_{j}$.
Otherwise, \eqref{eqn:vc-strong-efficient-Q-LB} implies that on $\hat{E}_{j}^{\prime}$,
if $m_{j} \in S$ and $\Px(\DISV_{m_{j}}) > \frac{16s_{j}}{m_{j}}$, then 
$|Q_{m_{j}}| \geq (1/4) m_{j} \Px(\DISV_{m_{j}})$.
In this latter case, we have
\begin{multline}
\label{eqn:vc-strong-efficient-TPDIS-raw-bound}
4 \hat{T}_{m_{j}} \Px(\DISV_{m_{j}}) \leq 
\\ 16 c_{0} \max
\begin{cases}
\Px(\DISV_{m_{j}})^{\frac{1-\tsybb}{2-\tsybb}} \left(\frac{\tsybcb}{m_{j}} \left( \vc(\G_{\F}) \Log\!\left(\frac{\maxloss^{2}}{\tsybcb} \left(\frac{ m_{j} \Px(\DISV_{m_{j}}) }{4\tsybcb\vc(\G_{\F})}\right)^{\frac{\tsybb}{2-\tsybb}}\right)+s_{j}\right)\right)^{\frac{1}{2-\tsybb}}
\\ \frac{\maxloss}{m_{j}} \left( \vc(\G_{\F}) \Log\!\left( \frac{\maxloss^{2}}{\tsybcb}\left(\frac{m_{j} \Px(\DISV_{m_{j}})}{4\maxloss\vc(\G_{\F})}\right)^{\tsybb}\right)+s_{j}\right)
\end{cases}\!\!\!\!\!\!\!\!.
\end{multline}
Since $m_{j} \geq 2m_{j-1}$, by the inductive hypothesis, on $\bigcap_{i=1}^{\log_{2}(m_{j-1})} \hat{E}_{2^{i}} \cap \bigcap_{j^{\prime}=0}^{j-1} \hat{E}_{j^{\prime}}^{\prime}$, 
if $m_{j} \in S$, we have $V^{(m_{j})} \subseteq \tilde{V}^{(m_{j-1})} \subseteq \F\left( \BJM^{-1}(\eps_{j-1}) ; \zo \right)$,
which implies $\Px(\DISV_{m_{j}}) \leq \Px\left(\DIS\left(\F\left(\BJM^{-1}\left(\eps_{j-1}\right) ; \zo\right)\right)\right) = p_{j-1}$.
In this case, the right hand side of \eqref{eqn:vc-strong-efficient-TPDIS-raw-bound} is at most
\begin{equation*}
16 c_{0} \max
\begin{cases}
p_{j-1}^{\frac{1-\tsybb}{2-\tsybb}} \left(\frac{\tsybcb}{m_{j}} \left( \vc(\G_{\F}) \Log\left(\frac{\maxloss^{2}}{\tsybcb} \left(\frac{ m_{j} p_{j-1} }{4\tsybcb\vc(\G_{\F})}\right)^{\frac{\tsybb}{2-\tsybb}}\right)+s_{j}\right)\right)^{\frac{1}{2-\tsybb}}
\\ \frac{\maxloss}{m_{j}} \left( \vc(\G_{\F}) \Log\left( \frac{\maxloss^{2}}{\tsybcb}\left(\frac{m_{j} p_{j-1}}{4\maxloss\vc(\G_{\F})}\right)^{\tsybb}\right)+s_{j}\right)
\end{cases}.
\end{equation*}
The value of $m_{j}^{\prime}$ was defined to make this value at most $\eps_{j}$, with any value of $c^{\prime} \geq 16 c_{0}$.
Altogether, we have that on $\bigcap_{i=1}^{\log_{2}(m_{j})} \hat{E}_{2^{i}} \cap \bigcap_{j^{\prime}=0}^{j} \hat{E}_{j^{\prime}}^{\prime}$,
if $m_{j} \in S$, then every $h \in \tilde{V}^{(m_{j})}$ has $\Risk(h_{\DISV_{m_{j}}};\PXY) - \Risk(\target;\PXY) \leq \eps_{j}$;
in particular, this also implies every $h \in \tilde{V}^{(m_{j})}$ has $\er(h_{\DISV_{m_{j}}}) - \er(\target) \leq \BJM^{-1}(\eps_{j})$.
Since we have already proven that $\target \in V^{(m_{j})}$ on this event, and since $\tilde{V}^{(m_{j})} \subseteq V^{(m)}$,
we have that every $h \in \tilde{V}^{(m)}$ has $\er(h) = \er(h_{\DISV_{m}})$, which therefore implies $\er(h) - \er(\target) \leq \BJM^{-1}(\eps_{j})$:
that is, $\tilde{V}^{(m_{j})} \subseteq \F\left( \BJM^{-1}(\eps_{j}) ; \zo \right)$.
This completes the inductive proof.

The above result implies that, on $\bigcap_{i=1}^{\log_{2}(m_{\hat{j}_{\eps}})} \hat{E}_{2^{i}} \cap \bigcap_{j=0}^{\hat{j}_{\eps}} \hat{E}_{j}^{\prime}$,
if $m_{\hat{j}_{\eps}} \in S$, then $\er(\hat{h}) - \er(\target) \leq \BJM^{-1}( \eps_{\hat{j}_{\eps}} ) \leq \BJM^{-1}\left( \BJM(\eps) \right) = \eps$.
In particular, we are guaranteed to have $m_{\hat{j}_{\eps}} \in S$ as long as $u \geq m_{\hat{j}_{\eps}}$ and 
\begin{equation}
\label{eqn:vc-strong-efficient-basic-n-bound}
n > \sum_{i=1}^{\log_{2}(m_{\hat{j}_{\eps}})} \sum_{m=2^{i-1}+1}^{\min\left\{2^{i},\max S\right\}} \ind_{\DIS\left(\tilde{V}^{(2^{i-1})}\right)}(X_{m}).
\end{equation}
By monotonicity of $m \mapsto \DIS\left(\tilde{V}^{(m)}\right)$,
the right hand side of \eqref{eqn:vc-strong-efficient-basic-n-bound} is at most
\begin{equation*}
\sum_{j=0}^{\hat{j}_{\eps}} \sum_{m=m_{j-1}+1}^{\min\left\{m_{j},\max S\right\}} \ind_{\DIS\left(\tilde{V}^{(m_{j-1})}\right)}(X_{m}).
\end{equation*}
Furthermore, on $\bigcap_{i=1}^{\log_{2}(m_{\hat{j}_{\eps}})} \hat{E}_{2^{i}} \cap \bigcap_{j=0}^{\hat{j}_{\eps}} \hat{E}_{j}^{\prime}$, 
the above result implies this is at most
\begin{multline*}
\sum_{j=1}^{\hat{j}_{\eps}} \sum_{m=m_{j-1}+1}^{\min\left\{m_{j},\max S\right\}} \ind_{\DIS\left(\F\left( \BJM^{-1}(\eps_{j-1});\zo\right)\right)}(X_{m})
\\ \leq \sum_{j=1}^{\hat{j}_{\eps}} \sum_{m=m_{j-1}+1}^{m_{j}} \ind_{\DIS\left(\F\left( \BJM^{-1}(\eps_{j-1});\zo\right)\right)}(X_{m}).
\end{multline*}
By a Chernoff bound, on an event $\hat{E}^{\prime\prime}$ of probability at least $1-\conf/2$, 
the right hand side of the above is at most
\begin{equation}
\label{eqn:vc-strong-efficient-abstract-query-bound}
\log_{2}(2/\conf) + \sum_{j=1}^{\hat{j}_{\eps}} (m_{j}-m_{j-1}) p_{j-1}.
\end{equation}
Since $\eps_{j-1} \geq \maxloss 2^{1-\hat{j}_{\eps}} \geq \BJM(\eps)$,
and therefore 
\begin{multline*}
p_{j-1} 
\leq \Px\left(\DIS\left( \Ball\left(\target, \tsybca \BJM^{-1}(\eps_{j-1})^{\tsyba}\right) \right) \right) 
\\ \leq \dc\left( \tsybca \BJM^{-1}(\eps_{j-1})^{\tsyba} \right) \tsybca \BJM^{-1}(\eps_{j-1})^{\tsyba} 
\leq \dc\left( \tsybca \eps^{\tsyba} \right) \tsybca \BJM^{-1}(\eps_{j-1})^{\tsyba},
\end{multline*}
letting $\hat{c}_{j} = \vc(\G_{\F}) \Log\left(\frac{\maxloss^{2}}{\tsybcb}\left(\frac{c^{\prime} \dc \tsybca \BJM^{-1}(\eps_{j-1})^{\tsyba}}{\eps_{j}}\right)^{\tsybb}\right)$,
we have that 
\begin{equation}
\label{eqn:vc-strong-efficient-mprime-bound}
2^{1+\lceil \log_{2}(m_{j}^{\prime}) \rceil} \leq 
4 c^{\prime} \left(\frac{\tsybcb}{\eps_{j}} \left(\frac{\dc \tsybca \BJM^{-1}(\eps_{j-1})^{\tsyba}}{\eps_{j}}\right)^{1-\tsybb} + \frac{\maxloss}{\eps_{j}}\right) \left( \hat{c}_{j}+s_{j} \right).
\end{equation}
Since $\BJM^{-1}(\eps_{j-1})^{\tsyba}/\eps_{j}$ is nondecreasing in $j$, 
the right hand side of \eqref{eqn:vc-strong-efficient-mprime-bound} 
at least doubles when $j$ is increased by one, so that by induction we have that
the right hand side of \eqref{eqn:vc-strong-efficient-mprime-bound} is also an upper bound on $m_{j}$.
This fact also implies that $\hat{c}_{j} + s_{j}$ is at most
\begin{equation*}
\vc(\G_{\F}) \Log\!\left(\frac{\maxloss^{2}}{\tsybcb} \left(\frac{2c^{\prime} \dc \tsybca \BJM^{-1}(2\BJM(\eps))^{\tsyba}}{\BJM(\eps)}\right)^{\tsybb}\right)+\Log\!\left(\frac{192}{\conf}\right) + 2\Log\!\left(2+\hat{j}_{\eps}-j\right),
\end{equation*}
and the fact that $x \mapsto \BJM^{-1}(x)/x$ is nonincreasing implies this is at most
\begin{multline*}
\vc(\G_{\F}) \Log\left(\frac{\maxloss^{2}}{\tsybcb} \left(\frac{4c^{\prime} \dc \tsybca \eps^{\tsyba}}{\BJM(\eps)}\right)^{\tsybb}\right)+\Log\left(\frac{192}{\conf}\right) + 2\Log\left(2+\hat{j}_{\eps}-j\right)
\\ \leq c^{\prime\prime} \left( A_{2} + \Log\left(2+\hat{j}_{\eps}-j\right) \right).
\end{multline*}
where $c^{\prime\prime} = \ln\left( 768 e c^{\prime} \right)$.
Furthermore, 
\begin{multline*}
\frac{\BJM^{-1}(\eps_{j-1})^{\tsyba}}{\eps_{j}}
= 2\frac{\BJM^{-1}( 2^{(\hat{j}_{\eps}-j)} \eps_{\hat{j}_{\eps}-1} )^{\tsyba}}{2^{(\hat{j}_{\eps}-j)} \eps_{\hat{j}_{\eps}-1}}
\\ \leq 2\frac{\BJM^{-1}( 2^{(\hat{j}_{\eps}-j)} \BJM(\eps))^{\tsyba}}{2^{(\hat{j}_{\eps}-j)} \BJM(\eps)}
\leq 2^{1+(\hat{j}_{\eps}-j)(\tsyba-1)} \frac{\eps^{\tsyba}}{\BJM(\eps)}.
\end{multline*}
Applying these inequalities to bound $m_{j} p_{j-1}$, and reversing the order of summation (now summing over $i=\hat{j}_{\eps}-j$), we have that
\begin{align*}
\sum_{j=1}^{\hat{j}_{\eps}} m_{j} p_{j-1} 
& \leq 16 c^{\prime} c^{\prime\prime} \sum_{i=0}^{\hat{j}_{\eps}-1} \tsybcb \left(\frac{\tsybca \dc 2^{i(\tsyba-1)} \eps^{\tsyba}}{\BJM(\eps)}\right)^{2-\tsybb} \left( A_{2} + \Log(i+2) \right)
\\ & + 16 c^{\prime} c^{\prime\prime} \sum_{i=0}^{\hat{j}_{\eps}-1} \frac{\maxloss \tsybca \dc 2^{i(\tsyba-1)} \eps^{\tsyba}}{\BJM(\eps)} \left( A_{2} + \Log(i+2) \right).
\end{align*}
Note that this is of the same form as \eqref{eqn:vc-subgraph-active-strong-summation} in the proof of Theorem~\ref{thm:vc-subgraph-active-strong},
so that following that proof, the right hand side above is at most
\begin{equation*}
144 c^{\prime} c^{\prime\prime}
\left( 
\tsybcb (A_{2} + \Log(C_{1}) C_{1} \left( \frac{\dc \tsybca \eps^{\tsyba}}{\BJM(\eps)} \right)^{2-\tsybb} 
+ \maxloss (A_{2}+\Log(C_{1})) C_{1} \left( \frac{\dc \tsybca \eps^{\tsyba}}{\BJM(\eps)} \right)
\right).
\end{equation*}
Therefore, since $\log_{2}(2/\conf) \leq 3 A_{2}$, \eqref{eqn:vc-strong-efficient-abstract-query-bound} is less than
\begin{equation*}
147 c^{\prime} c^{\prime\prime}\left( \tsybcb (A_{2} + \Log(C_{1}) C_{1} \left( \frac{\dc \tsybca \eps^{\tsyba}}{\BJM(\eps)} \right)^{2-\tsybb} + \maxloss (A_{2}+\Log(C_{1})) C_{1} \left( \frac{\dc \tsybca \eps^{\tsyba}}{\BJM(\eps)} \right) \right).
\end{equation*}
The above inequalities also imply that
\begin{equation*}
m_{\hat{j}_{\eps}} \leq
32 c^{\prime} c^{\prime\prime} \left(\frac{\tsybcb \left(\dc \tsybca \eps^{\tsyba}\right)^{1-\tsybb}}{\BJM(\eps)^{2-\tsybb}} + \frac{\maxloss}{\BJM(\eps)}\right) A_{2}.
\end{equation*}
Thus, taking $c = 147 c^{\prime} c^{\prime\prime}$ in the statement of Theorem~\ref{thm:vc-subgraph-active-strong} suffices to guarantee 
that, for any $u$ and $n$ satisfying the given size constraints, 
$u \geq m_{\hat{j}_{\eps}}$, and on the event 
$\bigcap_{i=1}^{\log_{2}(m_{\hat{j}_{\eps}})} \hat{E}_{2^{i}} \cap \bigcap_{j=0}^{\hat{j}_{\eps}} \hat{E}_{j}^{\prime} \cap \hat{E}^{\prime\prime}$,
\eqref{eqn:vc-strong-efficient-basic-n-bound} is satisfied, 
which (as discussed above) implies $\er(\hat{h})-\er(\target) \leq \eps$ on this event.
We complete the proof by noting that, by a union bound, the event 
$\bigcap_{i=1}^{\log_{2}(m_{\hat{j}_{\eps}})} \hat{E}_{2^{i}} \cap \bigcap_{j=0}^{\hat{j}_{\eps}} \hat{E}_{j}^{\prime} \cap \hat{E}^{\prime\prime}$
has probability at least
\begin{equation*}
1 - \sum_{i=1}^{\log_{2}(m_{\hat{j}_{\eps}})} 12 e^{-\hat{\sfun}(2^{i})} - \sum_{j=0}^{\hat{j}_{\eps}} 2^{-s_{j}} - \frac{\conf}{2},
\end{equation*}
which is greater than $1-\conf$, since
\begin{align*}
& \sum_{i=1}^{\log_{2}(m_{\hat{j}_{\eps}})} 12 e^{-\hat{\sfun}(2^{i})} 
\leq \sum_{j=1}^{\hat{j}_{\eps}} \sum_{i=\log_{2}(m_{j-1})+1}^{\log_{2}(m_{j})} \frac{\conf}{4 \log(4 m_{j} / 2^{i})^{2} (2+\hat{j}_{\eps}-j)^{2}}
\\ & \leq \sum_{j=1}^{\hat{j}_{\eps}} \sum_{k=0}^{\infty} \frac{\conf}{4 (2 + k)^{2} (2+\hat{j}_{\eps}-j)^{2}}
\leq \sum_{j=1}^{\hat{j}_{\eps}} \frac{\conf}{4 (2+\hat{j}_{\eps}-j)^{2}}
\leq \sum_{k=0}^{\infty} \frac{\conf}{4 (2+k)^{2}}
\leq \frac{\conf}{4},
\end{align*}
and
$\sum_{j=0}^{\hat{j}_{\eps}} 2^{-s_{j}}
\leq \sum_{j=0}^{\hat{j}_{\eps}} \frac{\conf}{192(2+\hat{j}_{\eps}-j)^{2}}
\leq \sum_{k=0}^{\infty} \frac{\conf}{192(2+k)^{2}}
\leq \frac{\conf}{192}$.

\section{Remarks on the Assumption that $\texorpdfstring{\MakeLowercase{f}}{f}^{\star} \in \F$}
\label{app:fstar-assumption}

We conclude with some remarks on the assumption that $\target \in \F$ (used throughout this article).
As noted in Section~\ref{subsec:surrogate-losses}, 
this assumption is often \emph{very} strong.  
While the specific assumption that $\target \in \F$ 
adds a certain elegance to the theory developed in this work, one natural question is to what extent it can 
be relaxed without changing the essence of the approach considered here.
For instance, in passive learning, one can generalize the abstract results on 
empirical risk minimization (stated in Theorem~\ref{thm:erm}) to hold under the weaker condition
that $\argmin_{h \in \F} \Risk(h) = \argmin_{h \in \F} \er(h)$.  However, this simple
relaxation appears insufficient for the approach to active learning considered 
here.  Specifically, for our analysis, we would require that an error minimizer $\argmin_{h \in \F} \er(h)$
also be an (approximate) minimizer of 
$\Risk(h;P)$ in $\F$,
not merely for $P = \PXY$, but also for certain conditional distributions 
$\PXY( \cdot | \DIS(V)\times\Y )$,
for sets $V \subseteq \F$ arising in the algorithm.  In principle, 
the results in this work can be generalized to 
provide guarantees when this condition (suitably formalized) is satisfied.  
However, the statements of the results become considerably more involved, 
and moreover we do not know of concise, general, \emph{a priori} conditions on $\F$, $\loss$, and $\PXY$, 
under which this property will hold.
Beyond this,
it appears our analysis does not easily extend to 
the important problem of active learning with surrogate losses
in the \emph{general} case, where
results would presumably need to be expressed in terms of the approximation
loss $\inf_{f \in \F} \Risk(f) - \Risk(\target)$ or related quantities (as observed 
for passive learning \citep*{bartlett:06}).  
It seems such a generalization would require a significantly different approach.


\end{document}